\documentclass[11pt]{amsart}
\usepackage{amssymb} 
\usepackage{amsmath} 
\usepackage{graphicx} 
\usepackage[english]{babel} 
\usepackage{epsfig}
\usepackage{enumerate}
\usepackage{color}
\usepackage{amsmath}
\usepackage{hyperref}

\input xy
\xyoption{all}

\usepackage{mathrsfs}
\usepackage{graphics}
\usepackage{amsthm}
\usepackage{graphicx}
\usepackage{fullpage}

\usepackage[usenames,dvipsnames]{pstricks} 
\usepackage{pst-grad} 
\usepackage{pst-plot} 

\swapnumbers

\swapnumbers

\theoremstyle{definition}
\newtheorem{definition}[subsubsection]{Definition}
\newtheorem{example}[subsubsection]{Example}
\newtheorem{notation}[subsubsection]{Notation}
\newtheorem{remark}[subsubsection]{Remark}
\theoremstyle{plain}
\newtheorem{theorem}[subsubsection]{Theorem}

\newtheorem{proposition}[subsubsection]{Proposition}
\newtheorem{lemma}[subsubsection]{Lemma}
\newtheorem{corollary}[subsubsection]{Corollary}

\def\K{{\mathbb K}}

\def\ov{\overline}

\def\ot{\otimes}

\begin{document}

\author[Dalia Artenstein, Ana Gonz\'alez, M. Ronco]{Dalia Artenstein, Ana Gonz\'alez, Mar\'\i a Ronco}
\address{DA: IMERL, Universidad de la Rep\'ublica\\ Facultad de Ingeniería, Av. Julio Herrera y Reissig  565\\ Montevideo, Uruguay}
\email{darten@fing.edu.uy}
\address{AG: IMERL, Universidad de la Rep\'ublica\\ Facultad de Ingeniería, Av. Julio Herrera y Reissig  565\\ Montevideo, Uruguay}
\email{anagon@fing.edu.uy}
\address{MOR: IMAFI, Universidad de Talca\\ Campus Norte, Avda. Lircay s/n\\ Talca, Chile}
\email{mronco@utalca.cl}

\title{Primitive elements in infinitesimal bialgebras}



\maketitle
\tableofcontents
	
\today

\begin{abstract}
For any set $S$, the free magmatic algebra spanned by ${\mbox{card}(S)}$ binary products is the vector space spanned by the set of all planar rooted binary trees with the internal nodes colored by the elements of $S$, graded by the number of leaves of a tree. We show that it has a unique structure of coassociative coalgebra such that the coproduct 
satisfies the unital infinitesimal condition with each magmatic product, and prove an analog of Aguiar-Sottile\rq s formula in this context, describing the coproduct in terms of the Moebius basis for the Tamari order. The last result allows us to compute the subspace of primitive elements of any unital infinitesimal $S$-magmatic bialgebra. As an example, we construct a set of generators of the dual of Pilaud and Pons bialgebra of integer relations and compute an explicit basis of its subspace of primitive elements.
\end{abstract}
\section{Introduction}
\medskip

In \cite{JoRo}, S. A. Joni and G.-C. Rota introduced the definition of infinitesimal bialgebra (see also \cite{Agu}). An infinitesimal bialgebra is a vector space $A$ equipped with an associative product $\cdot$ and a coassociative coproduct $\Delta$ satisfying the following relation
\begin{equation}  \Delta (x\cdot y) = \sum x_{(1)}\otimes (x_{(2)}\cdot y) + \sum (x\cdot y_{(1)})\otimes y_{(2)},\end{equation}
which means that $\Delta$ is a coderivation for the product $\cdot$.

This notion was modified in \cite{LoRo2}, by  J.-L. Loday and the third author, to study certain coassociative coalgebras equipped with a unital associative product. A unital infinitesimal bialgebra is a coassociative counital coalgebra $(H, \Delta)$ equipped with an associative product $\cdot$ which satisfies the equation
\begin{equation}  \Delta (x\cdot y) = \sum x_{(1)}\otimes (x_{(2)}\cdot y) + \sum (x\cdot y_{(1)})\otimes y_{(2)} - x\otimes y,\end{equation}
for any elements $x$ and $y$ in $H$, where $\Delta (z) = \sum z_{(1)}\otimes z_{(2)}$ for any $z\in H$.

The product of an infinitesimal bialgebra has no unit, but the extra term of equation $(2)$, allows to work with units. 
 \medskip

The goal of the present work is to study the structure of conilpotent coalgebras equipped with many magmatic products, where a magmatic  product (following Bourbaki\rq s definition) is just a binary operation. For any set $S$, an $S$ magmatic bialgebra is coassociative coalgebra $(H, \Delta)$ equipped with a family of magmatic products $\{*_s\}_{s\in S}$ satisfying the unital infinitesimal relation with $\Delta$.

In  \cite{LoRo2}, the authors showed that any conilpotent unital infinitesimal bialgebra $H$ is isomorphic to the cotensor coalgebra over the subspace ${\mbox{Prim}(H)}$ of its primitive elements. This result was extended in \cite{HoLoRo} to any conilpotent infinitesimal magmatic bialgebra, that is a counital coalgebra equipped with a binary product (not necessarily associative) which satisfies relation $(2)$. We want to compute its subspace of primitive elements.
\medskip

For any positive integer $n$,  let ${\mathbf {PBT}}_n^{S}$ be the set of planar binary rooted trees with $n$ leaves and its $n-1$ internal nodes colored by the elements of $S$. The free $S$-magmatic algebra generated by one element is the graded vector space $\K\bigl[{\mathbf {PBT}}^{S}\bigr]:=\bigoplus_{n\geq 1} \K[{\mathbf {PBT}}_n^{S}]$, equipped with the products $\veebar_s$ for $s\in S$, where $t\veebar_s w$ is the tree obtained from $t$ and $s$ by joining their roots to a new root colored with $s$. When $\vert S\vert =1$, we simply denote $\K\bigl[{\mathbf {PBT}}\bigr]$  and $\veebar $, instead of $\K\bigl[{\mathbf {PBT}}^{S}\bigr]$ and $\veebar_s$.

For $S\neq \emptyset$, there exists a unique coassociative coproduct $\overset{\circledast}{\Delta }$, such that $\bigl(\K\bigl[{\mathbf {PBT}}^{S}\bigr], \cdot_s, \overset{\circledast}{\Delta }\bigr)$ is a magmatic infinitesimal bialgebra. A standard argument shows that this bialgebra structure exists on any free $S$-magmatic algebra.
\medskip

As the free $S$-magmatic algebra over one element acts on any $S$-magmatic algebra $A$, any element  of degree $n$ in $\K\bigl[{\mathbf {PBT}}^{S}\bigr]$ defines a $n$-ary product on $A$. Moreover, if $H$ is a unital $S$-magmatic bialgebra and an element $x$ is primitive in $(\K\bigl[{\mathbf {PBT}}^{S}\bigr], \overset{\circledast}{\Delta })$, then the subspace ${\mbox{Prim}(H)}$ of primitive elements of $H$ is close under the  action of $x$. So, we compute the primitive elements of the coalgebra $\K\bigl[{\mathbf {PBT}}^{S}\bigr]$.
\medskip

In \cite{AgSo}, M. Aguiar and F. Sottile used the Tamari order to describe the coproduct of the graded Hopf algebra of planar binary rooted trees $\K[Y_{\infty}]$, introduced in \cite{LoRo1}, in terms of the elements $M_t$ of the Moebius basis. They proved that
\begin{equation} \Delta (M_t) =\sum_{t_1/t_2 =t} M_{t_1}\otimes M_{t_2},\end{equation}
where $\Delta$ is the coproduct of $\K[Y_{\infty}]$ and $t_1/t_2$ is the tree obtained by joining the root of $t_1$ to the leftmost leaf of $t_2$.

In particular their result proved that the set $\{M_t\}$, where $t =\vert \veebar t\rq $, is a basis of the subspace of primitive elements of the Hopf algebra $\K[Y_{\infty}]$. However, the vector space $\K\bigl[{\mathbf {PBT}}\bigr]$ is the suspension  of the underlying graded vector space of $\K[Y_{\infty}]$, and consequently the coproduct $\Delta$ differs from the coproduct $\overset{\circledast}{\Delta }$ we deal with. 

We reformulate Aguiar and Sottile\rq s formula by introducing a graded associative product $\rightthreetimes$ on $\K\bigl[{\mathbf {PBT}}\bigr]$, and prove a similar result in our framework: the collection of elements $\{M_{t}\}$, with ${t}$ a $\rightthreetimes$-irreducible tree, is a basis of the subspace ${\mbox{Prim}(\K\bigl[{\mathbf {PBT}}\bigr]})$. These results generalize easily to colored trees.

As a byproduct, we give a recursive formula for the elements of the Moebius basis $\{M_t\}_{t\in {\mathbf{PBT}}}$ in terms of the products $\veebar$, $\rightthreetimes$ and the graftings of trees. Even if the Moebius function of the Tamari order was yet computed in different ways (see for instance \cite{BlSa}), we did not find in previous publications the recursive description obtained in the present work. 
\medskip

Finally, we get that if $H$ is a unital infinitesimal $S$-magmatic bialgebra generated by a set $X$ of primitive elements, then $H$ is isomorphic as a coalgebra to $T^c({\mbox{Prim}(H)})$, where ${\mbox{Prim}(H)}$ is generated by $X$ under the action of the products $M_{\underline t}$, where ${\underline t}$ is a colored tree satisfying certain conditions. 
\bigskip

Our motivation and main example to study these type of algebras comes from the Hopf algebras of integer relations introduced in \cite{PiPo} and \cite{PiPo1}.
In \cite{ChPiPo}, G. Chatel, V. Pilaud and V. Pons defined a weak order on the set ${\mathcal R}_n$, of all the reflexive relations on the set $\{1,\dots ,n\}$, which extends the weak Bruhat order of the set $\Sigma_n$ of permutations.

Afterwards, V. Pilaud and V. Pons defined a Hopf algebra structure on the graded vector space $\K[{\mathcal R}]:= \bigoplus_{n\geq 0}\K[{\mathcal R}_n]$, spanned by the set all reflexive integer relations, we denote it by ${\mathbb H}_{{\mathcal R}_{PP}}$.  The authors proved that the product of ${\mathbb H}_{{\mathcal R}_{PP}}$ is determined by intervals of the partial order introduced in \cite{ChPiPo}. Studying sub-Hopf algebras and quotients of ${\mathbb H}_{{\mathcal R}_{PP}}$, they obtained some well-known combinatorial Hopf algebras, as the Malvenuto-Reutenauer algebra (see\cite{MaRe}), the Hopf algebra of planar binary trees (see \cite{LoRo1}) and the Hopf algebra of ordered partitions (see \cite{Cha}), as well as new Hopf algebras as the algebra of special partially ordered sets, the algebra of weak order intervals and the algebra of Tamari order intervals. All these Hopf algebras are conilpotent.

In fact, the dual of the Hopf algebra ${\mathbb H}_{{\mathcal R}_{PP}}^*$ is easy to describe: the coproduct is given by the restrictions $\{1,\dots ,n\}\longrightarrow \{1,\dots ,i\}\times \{i+1,\dots ,n\}$, for $0\leq i\leq n$, and the associative product is obtained by shuffling the unital infinitesimal associative product $\uparrow$, which associates to any pair of integer relations $P\in {\mathcal R}_n$ and $Q\in {\mathcal R}_m$ the minimal reflexive relation defined on the disjoint union $P\sqcup Q$ satisfying that $(P\uparrow Q )\vert_{\{1,\dots ,n\}} =P$, $(P\uparrow Q )\vert_{\{n+1,\dots ,n+m\}} =Q$ and $(i,j)\in P\uparrow Q$, for any $1\leq i\leq n$ and $n+1\leq j\leq n+m$. 
\medskip

We prove that the vector space $\K[{\mathcal R}]$ is generated by the unique element $\#$ of ${\mathcal R}_1$ under the action of an infinite family of binary products 
$*_{\alpha}$, where $\alpha$ is a an element of the set of maps $\bigcup_{r\geq 1}\{\sqcup , \uparrow , \downarrow , \updownarrow\}^{\{1,\dots ,r\}}$. The products $*_{\alpha}$ are not associative, but they verify the unital infinitesimal relation with the coproduct of ${\mathbb H}_{{\mathcal R}_{PP}}^*$ . Therefore, the coalgebra ${\mathbb H}_{{\mathcal R}_{PP}}^*$  is isomorphic to the cotensor algebra over the vector space spanned by the element $\#$ under the action of the operations $M_{t}(\alpha_1,\dots ,\alpha_n)$, for any $\rightthreetimes$-irreducible planar binary rooted tree $t\in {\mathbf {PBT}}_{n+1}$. 
\medskip

As the disjoint union $\sqcup$ of reflexive integer relations defines an associative product on the vector space $\K[{\mathcal R}]$, it is easy to prove that $\K[{\mathcal R}]$ is isomorphic to the tensor vector space $T(\K[{\mbox{$\sqcup$-Irr}}])$, where $\K[{\mbox{$\sqcup$-Irr}}]$ denotes the vector space spanned by the set of all reflexive integer relations which are not the disjoint union of two other relations of smaller size. Therefore, we get that as vector spaces $\K[{\mbox{$\sqcup$-Irr}}]$ and the subspace ${\mbox{Prim}({\mathbb H}_{{\mathcal R}_{PP}}^*)}$ of primitive elements of the coalgebra $\K[{\mathcal R}]^*$ are isomorphic. We construct an explicit map which associates to any $\sqcup$-irreducible integer relation a primitive element of $\K[{\mathcal R}]^*$, in such a way that the images of all the elements of ${\mbox{$\sqcup$-Irr}}$ give a basis of ${\mbox{Prim}(\K[{\mathcal R}]^*)}$.\bigskip

The manuscript is organized as follows:

In section 2, we reformulate  Aguiar and Sottile\rq s formula  in our framework and prove it. In order to do so, we study the relationship between the Tamari order and the operations $\veebar$ and $\rightthreetimes$, which allows us to give a recursive formula for the Moebius function of the Tamari partial order.
\medskip

In section 3 we recall the basic definitions of infinitesimal magmatic bialgebras and the main results of \cite{HoLoRo}. 
We describe the  $S$-magmatic bialgebra structure of $\K\bigl[{\mathbf {PBT}}^{S}\bigr]$. 

Finally, we look at an element $(t,(s_{1},\cdots, s_{n-1})) \in \mathbf {PBT}_{n}^S$ as an $n$-ary operator on any unital $S$-magmatic infinitesimal bialgebra $H$.  When $H$ is generated, as an $S$-magmatic algebra, by a set $X$ of primitive elements, the subspace of primitive elements ${\mbox{Prim}(H)}$ is generated by the action of the operators $\{M_t(s_1,\dots, s_{n-1})\mid t\in {\mbox{$\rightthreetimes$-Irr}_n}\ {\rm and}\ s_1,\dots, s_{n-1}\in S\}$ and some binary products $M(s)$ on $X$. The recursive description of the Moebius elements $M_t$ obtained in section 2 shows that certain operators $M_t(s_1,\dots, s_{n-1})$ are the compostion of other ones, we exhibit a smaller set of generators of ${\mbox{Prim}(H)}$.
\medskip

In section 4 we consider the  dual Hopf algebra of integer relations $\K[{\mathcal R}]^{*}$ defined by V. Pilaud and V. Pons, and describe a collection of binary magmatic products on the coalgebra  which generates $\K[{\mathcal R}]$ from the unique integer relation of size $1$. Using the previous results, we construct an explicit map which associates to any $\sqcup$-irreducible integer relation a primitive element of $\K[{\mathcal R}]^*$, in such a way that the images of all the elements of ${\mbox{$\sqcup$-Irr}}$ give a basis of ${\mbox{Prim}(\K[{\mathcal R}]^*)}$.

\medskip

{\bf Basic Notation}  We denote by $\K$ a base field and, for any set $X$, the vector field  spanned by $X$ is denoted $\K[X]$.

Given a $\K$-vector space, we denote $T(V) =\bigoplus _{n\geq o}V^{\otimes n}$ the free associative algebra over $V$, with the concatenation coproduct
\begin{equation*} (v_1\otimes \dots \otimes v_n)\smile (w_1\otimes \dots \otimes w_m) := v_1\otimes \dots \otimes v_n\otimes w_1\otimes \dots \otimes w_m.\end{equation*} 
We denote by $T^c(V)$ its dual coalgebra structure. That is, the underlying vector space of $T^c(V)$ is $T(V)$, but considered as the free conilpotent coalgebra over $V$, with the deconcatenation coproduct.

For $n\geq 1$, we denote by $[n]$ the set $\{1,\dots ,n\}$ of the first $n$ natural numbers and by $\Sigma_n$ the set of permutations of $n$ elements. For any vector space $V$, the natural (left) action of $\Sigma _n$ on $V^{\otimes n}$ is given by
\begin{equation*} \sigma (v_1\otimes \dots \otimes v_n) := v_{\sigma^{-1}(1)}\otimes \dots \otimes v_{\sigma^{-1}(n)}.\end{equation*}

A permutation $\sigma \in \Sigma_{n+m}$ satisfying that $\sigma(1) <\dots  <\sigma(n)$ and $\sigma(n+1)<\dots <\sigma(n+m)$ is called an {\it $(n,m)$-shuffle}. We denote the set of all $(n,m)$-shuffles by ${\mbox{sh}(n,m)}$.
\medskip

Given a partially ordered set $(P, \leq )$ and two elements $p, q \in P$, we denote $p\lessdot q$ when $q$ covers $p$, that is $p < q$ and there does not exist an element $z\in P$ such that $p < z < q$.

\section{Aguiar-Sottile formula} 

In the present section we recall basic notions about coalgebras, planar binary rooted trees and the Tamari order. We use them to introduce the coalgebra structure on the vector space $\K\bigl[{\mathbf {PBT}}^{S}\bigr]$, for any set $S$, and prove M. Aguiar and F. Sottile\rq s formula in this context. We give an easy recursive formula to compute the Moebius basis $\{M_t\}_{t\in {\mathbf{PBT}}}.$ 
\medskip

We begin by recalling some basic definitions of conilpotent coalgebras. 

\subsection{Unital augmented coalgebras} Let us recall the basic definitions of unital augmented coalgebra and primitive elements.

\begin{definition} \label{def:coalg} An augmented unital coalgebra over $\K$ is a vector space $C$, equipped with a linear map  $\Delta:C\longrightarrow C\otimes C$ and linear maps $\epsilon: C\longrightarrow \K$ and $\iota: \K\longrightarrow C$ satisfying that \begin{enumerate}
\item $\Delta $ is coassociative, that is $(\Delta\otimes {\mbox{id}_{C}})\circ \Delta = ({\mbox{id}_{C}}\otimes \Delta)\circ \Delta$,
\item $I_1\circ (\epsilon\otimes {\mbox{id}_{C}})\circ \Delta = {\mbox{id}_{C}} =I_2\circ ({\mbox{id}_{C}}\otimes \epsilon)\circ \Delta = {\mbox{id}_{C}}$, where $I_1$ (respectively $I_2$) is the natural isomorphism $I_1: \K\otimes C\longrightarrow C$ (respectively, $I_2: C\otimes \K\longrightarrow C$),
\item $\Delta (\iota(1_{\K}))= \iota(1_{\K})\otimes \iota(1_{\K})$ and $\epsilon\circ \iota = {\mbox{id}_{\K}}$\end{enumerate}
where ${\mbox{id}_{C}}$ is the identity morphism of $C$ (respectively  ${\mbox{id}_{\K}}$ is the identity of $\K$).\end{definition}

It is clear that, as a vector space,  $C = \K\oplus {\mbox{Ker}(\epsilon)}$. 	
\medskip

\begin{notation}\label{not:gral} Let $C$ be an augmented unital coalgebra. We denote by ${\ov C}$ the kernel of the aumentation map $\epsilon$. 
	
The  reduced coproduct ${\Delta}^{red} :{\ov C}\longrightarrow {\ov C}\otimes {\ov C}$ is the morphism given by:
\begin{equation*} {\Delta}^{red} := \Delta -  (\epsilon\otimes {\mbox{id}_{C}}\ +\ {\mbox{id}_{C}}\otimes \epsilon)\circ \Delta.\end{equation*}
It is immediate to see that if $\Delta$ is coassociative, then ${ \Delta}^{red}$ is coassociative, too.
	
We use Sweddler's notation for the reduced coproduct, that is ${\Delta}^{red}(x)=\sum x_{(1)}\otimes x_{(2)}$, for any $x\in {\ov C}$.
	
In general, for $n\geq 0$, we denote by ${ {\Delta}}^n:{\ov C}\longrightarrow {\ov C}^{\ot (n+1)}$ the linear map defined recursively by
\begin{enumerate}
\item ${{\Delta}}^0 :={\mbox{id}_{\ov C}}$,
\item for $n\geq 1$, ${{\Delta}}^{n+1} := \bigl({\mbox{id}_{\ov C^{\ot n}}}\ot { {\Delta}^{red}}\bigr)\ot { {\Delta}}^{n}$.
\end{enumerate}
\end{notation}
\medskip

\begin{definition} \label{def:primitive} 
Let $C$ be an augmented unital coalgebra. An element $x\in{\ov C}$ is  primitive if ${{\Delta}^{red}}(x) = 0$.  The subspace of primitive elements of $C$ is denoted ${\mbox{Prim}(C)}$. 

For $n\geq 1$, the subspace ${\mathcal F}_n(C)$ is given by
\begin{enumerate}
\item ${\mathcal F}_1(C) := {\mbox{Prim}(C)}$,
\item for $n > 1$, 
\begin{equation*} {\mathcal F}_n(C) := \{ x\in {\ov C}\mid {\Delta}^{red}(x)\in {\mathcal F}_{n-1}(C)\otimes {\mathcal F}_{n-1}(C)\}.
\end{equation*}
\end{enumerate}

We say that $C$ is conilpotent  when ${\ov C} = {\displaystyle \bigcup_{n\geq 1}{\mathcal F}_n(C)}$. 

If $C$ is conilpotent, then for any element $x\in {\ov C}$ there exists a positive integer $n_0$ such that ${{\Delta}}^n(x)=0$ for $n\geq n_0$. 

\end{definition}

\begin{definition}\label{def:unitinfbialg} A unital infinitesimal bialgebra is an augmented unital coalgebra $(C, \Delta, \epsilon, \iota)$ equipped with an associative product $\cdot$ which satisfies the following condition
\begin{equation*} \Delta^{red}(x\cdot y) = \sum x_{(1)}\otimes (x_{(2)}\cdot y) + \sum (x\cdot y_{(1)})\ot y_{(2)} + x\ot y,
\end{equation*}
for $x,y\in {\ov C}$.
\end{definition}
\bigskip

\subsection{Planar rooted trees}

\begin{definition} \label{trees} A {\it planar rooted binary tree} is a planar oriented connected graph with a maximal vertex, satisfying that any internal vertex has two incoming edges and one outgoing one. The maximal vertex of a planar binary rooted tree is called the {\it root}, while the minimal external vertices are called the {\it leaves} of the tree.\end{definition}

\begin{notation} \label{notn:coloredtrees} For $n\geq 1$, denote by ${\mbox{\bf PBT}_n}$ the set of all planar rooted binary trees with $n$ leaves (and $n-1$ internal vertices). The graded set ${\mbox{\bf PBT}}:= \bigcup _{n\geq 1} {\mbox{\bf PBT}_n}$ is the set of all planar rooted binary trees, graded by the number of leaves of a tree. We denote by $\vert t\vert$ the degree of a planar binary rooted tree $t$.
	
We denote by ${\mathbb K} [{\mbox{\bf PBT}}]$ the ${\mathbb K}$-vector space spanned by the graded set ${\mbox{\bf PBT}}$.\end{notation} 

\begin{example}\label{ex:trees} We have that $\mathbf{PBT_1}:=\bigl\{ \vert\bigr\}$, 

\begin{center}
\includegraphics[width=0.6\textwidth]{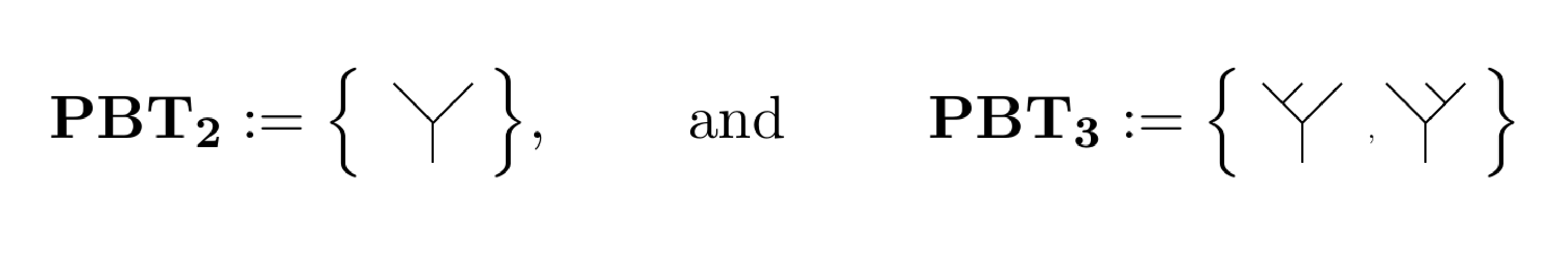}
\end{center}
\end{example}

Given a planar rooted tree $t$ with $n$ leaves, we assume that its leaves are numbered from $1$ to $n$, from left to right.
	
\begin{definition}\label{def:wedge} For any pair $t$ and $w$ of planar binary rooted trees,  the planar binary rooted tree $t\veebar w$ is obtained joining the roots of $t$ and $w$ to a new root, where $t$ is on the left side and $w$ on the right one.

Given a tree $t\in {\mbox{\bf PBT}_n}$,  a tree $w\in {\mbox{\bf PBT}_r}$ and and integer $1\leq j\leq n$, the {\it grafting} of $w$ at the $j^{th}$ leaf of $t$ is the tree $t\circ_jw\in 
{\mbox{\bf PBT}_{n+r-1}}$ obtained by joining the root of $w$ and the $j^{th}$ leaf of $t$. 

More in general, for any collection of trees $(w_1,\dots ,w_n)$, with $w_j\in {\mbox{\bf PBT}_{r_j}}$ for $1\leq j\leq n$, the {\it grafting} of $(w_1,\dots ,w_n)$ on $t$ is the tree 
\begin{equation*}t{\underline{\circ}}(w_1,\dots, w_n) := (((t\circ _n w_n)\circ_{n-1}w_{n-1})\dots )\circ_1w_1.\end{equation*}
\end{definition}

\begin{remark}\label{rem:propsofprods} The binary products $\veebar$ and $\circ_j$, as well as the $n$-ary operations ${\underline{\circ}}$ are extended by linearity to the vector space ${\mathbb K} [{\mbox{\bf PBT}}]$. The operation $\veebar$ defines a graded non-associative product on the vector space ${\mbox{\bf PBT}}$. Any tree $t\in {\mbox{\bf PBT}_n}$, with $n\geq 2$, may be written in a unique way as $t = t^l\veebar t^r$, with $\vert t^l\vert + \vert t^r\vert = \vert t\vert $.
\end{remark} 
\bigskip

For $n\geq 1$, D. Tamari defined in \cite{Tam} a partial order on the set  ${\mbox{\bf PBT}_n}$. This partial order is closely related to the weak Bruhat order defined on the set $\Sigma_n$ of permutations of $n$ elements, via the surjective map $\Gamma_n: \Sigma_n\longrightarrow {\mbox{\bf PBT}_n}$ defined by A. Tonks in \cite{Ton}. Let us recall the definition of both orders, as well as the construction of Tonk\rq s map.

\begin{definition} \label{def:Tamari} For $n\geq 1$, \begin{enumerate}
\item The weak Bruhat order on $\Sigma_n$ is the partial order transitively generated by the following covering relation:
Given permutations $\sigma$ and $\omega$ is $\Sigma_n$, we say that $\omega $ covers $\sigma$ if there exists $1\leq i <n$ such that $\sigma^{-1}(i) <\sigma^{-1}(i+1)$ and $\omega = \tau_i\cdot \sigma$, where $\tau_i$ is the transposition of $i$ and $i+1$. We denote it by $\leq_{wB}$.
\item the {\it Tamari order} on the set  ${\mbox{\bf PBT}_n}$ is the transitive relation spanned by the following conditions:\begin{enumerate}[(a)]
\item for any planar rooted trees $t_1, t_2$ and $t_3$ in ${\mbox{\bf PBT}_n}$, we have that $(t_1\veebar t_2)\veebar t_3 \leq_T t_1\veebar(t_2\veebar t_3)$,
\item if $t_1\leq_T w_1$ in ${\mbox{\bf PBT}_n}$ and $t_2\leq w_2$ in ${\mbox{\bf PBT}_m}$, then $t_1\veebar t_2\leq_T t_2\veebar w_2$ in ${\mbox{\bf PBT}_{n+m}}$.
\end{enumerate}
\end{enumerate}	
\end{definition}

It is well-known (and easy to check) that the partially ordered sets $(\Sigma_n , \leq_{wB})$ and $({\mbox{\bf PBT}_{n+1}}, \leq_T)$ are lattices.
\medskip

The description of Tonk\rq s map requires the notion of standardization.

\begin{definition} \label{def:stand} Given an injective map $\alpha:\{p+1, \dots, p+q\}\longrightarrow {\mathbb N}$, for positive integers $p,q$, the {\it standardization} of $\alpha$ is the unique permutation ${\mbox{std}(\alpha)}$ in $\Sigma _q$ satisfying that ${\mbox{std}(\alpha)}(i) < {\mbox{std}(\alpha)}(j)$ if, and only if, $\alpha (p+i) < \alpha(p+j)$.

For any permutation $\sigma\in \Sigma _n$ and any subset $S= \{i_1<\dots <i_r\}\subseteq \{1,\dots ,n\}$, the restricted permutation $\sigma^S\in \Sigma_r$ is the standardization of the map ${\tilde{\sigma}}$ given by ${\tilde{\sigma}}(j) := \sigma(i_j)$, for $1\leq j\leq r$.\end{definition}

For instance, if the image of $\alpha: \{ 3, 4, 5, 6, 7, 8\} \longrightarrow {\mathbb N}$ is the sequence $\alpha = (5, 3, 10, 7, 4, 12)$, the image of the permutation ${\mbox{std}(\alpha)}\in \Sigma_6$ is $(3, 1, 5, 4, 2, 6)$.
\medskip

\begin{definition} \label{def:Tonksmap} The surjective map $\Gamma_n:\Sigma_n\longrightarrow {\mbox{\bf PBT}_{n+1}}$ is recursively defined as follows\begin{enumerate}
\item $\Gamma_1((1)) :=$ $\scalebox{0.05}{\includegraphics[width=0.6\textwidth]{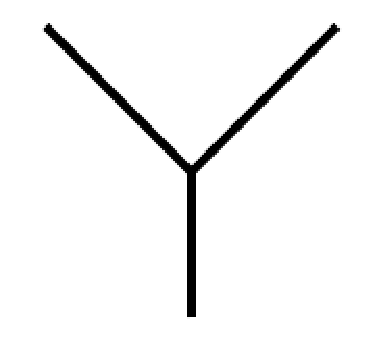}}$
\item For $n>1$, let $\sigma \in \Sigma_n$ be a permutation such that $\sigma^{-1}(n) = k$. Denote by $\sigma_1$ the restriction of $\sigma $ to the set $\{1,\dots ,k-1\}$ and by $\sigma_2$ its restriction to the set $\{k+1,\dots ,n\}$. The image of $\sigma$ under $\Gamma_n$ is the tree
\begin{equation*} \Gamma_n(\sigma ) =\begin{cases} \Gamma_{k_n-1}({\mbox{std}(\sigma_1)})\ \veebar\ \Gamma_{n-k_n}({\mbox{std}(\sigma_2)}),&{\rm for}\ k\notin\{1,n\},\\
\vert\veebar \Gamma_{n-1}({\mbox{std}(\sigma_2)}),&{\rm for}\ k= 1,\\
\sigma_1\veebar \vert,&{\rm for}\ k= n.\end{cases}
\end{equation*}
\end{enumerate}\end{definition}
\medskip

As shown in \cite{LoRo2} and \cite{AgSo}, for any $t\in {\mbox{\bf PBT}_{n+1}}$ the inverse image $\Gamma_n^{-1}(t)$ is an interval for the weak Bruhat order. Moreover, given two planar binary rooted trees $t$ and $w$ in ${\mbox{\bf PBT}_{n+1}}$, we have that $t <_Tw$ if for any permutation $\sigma \in \Sigma_n$ such that $\Gamma_n(\sigma ) = t$ there exist at least one permutation $\omega\in \Sigma_n$ such that $\Gamma_n(\omega ) = w$ and $\sigma < \omega$ for the weak Bruhat order of permutations. 

\begin{remark} \label{rem:proprestic} Given a pair of permutations $\sigma\leq_{wB}\tau\in \Sigma_n$ and a subset $S$ of $\{1,\dots, n\}$, we have that $\sigma^S\leq_{wB}\tau^S$\end{remark}

\begin{lemma} \label{lem:graftfortrees} Let $t$ and $w$ be a pair of trees in ${\mbox{\bf PBT}_n}$ such that $t = t_0{\underline{\circ}} (t_1,\dots ,t_r)$ and $w = w_0{\underline{\circ}} (w_1,\dots, w_r)$ for trees $t_0, w_0 \in {\mbox{\bf PBT}_r}$ and  $t_i, w_i\in {\mbox{\bf PBT}_{n_i}}$, for $1\leq i\leq r$ and $\sum_{i=1}^r n_i = n$.
 \begin{enumerate} 
\item if $t \leq_T w$, then $t_0\leq _Tw_0$ and $t_i\leq_T w_i$ for any $1\leq i\leq r$. 
\item any element $u$ satisfying that $t\leq_T u\leq_T w$ is of the form $u = u_0{\underline{\circ}} (u_1,\dots , u_r)$ with $u_0\in {\mbox{\bf PBT}_r}$ and $u_i\in {\mbox{\bf PBT}_{n_i}}$, for $1\leq i\leq r$.
\end{enumerate} \end{lemma}

\begin{proof} For the first point, let $\sigma$ and $\tau$ be permutations in $\Sigma_{n-1}$ such that $\Gamma_{n-1}(\sigma) = t_0{\underline{\circ}} (t_1,\dots ,t_r)$ and $\Gamma_{n-1}(\tau) = w_0{\underline{\circ}}(w_1,\dots ,w_r)$ and $\sigma\leq _{wB}\tau$.
Let $S=\{n_1,n_1+n_2+1,\dots ,n_1+\dots +n_{r-1}+r-2\}$, $T_1 =\{1,\dots ,n_1-1\}$ and $T_i =\{n_1+\dots +n_{i-1}+i-1,\dots , n_1+\dots +n_i+i-2\}$, for $2\leq i\leq r$. 

Applying Remark \ref{rem:proprestic}, we get that $\sigma^S\leq_{wB}\tau^S$ and $\sigma^{T_i}\leq_{wB}\tau^{T_i}$, for $1\leq i\leq n$.  

But, we have that\begin{enumerate}[(i)]
\item $\Gamma_{r-1}(\sigma^S) = t_0$ and $\Gamma_{r-1}(\tau^S) = w_0$, therefore $t_0\leq_T w_0$,
\item $\Gamma_{n_i-1}(\sigma^{T_i}) = t_i$ and $\Gamma_{n_i-1}(\tau^{T_i})= w_i$, for $1\leq i\leq r$, and therefore $t_i\leq_Tw_i$, for $1\leq i\leq r$,\end{enumerate}
which implies the result.
\medskip

For the second assertion, let $z = z_0\cdot (z_1,\dots , z_r)$ be a tree in ${\mbox{\bf PBT}_n}$.

If $u\lessdot z$, then $u$ is obtained from $z$ by one of the following movements\begin{enumerate}[(i)]
\item $u = u_0{\underline{\circ}} (z_1,\dots  ,z_r)$, with $u_0\lessdot z_0$,
\item for some $1\leq i\leq r$, there exists $u_i\lessdot z_i$ and $u=z_0{\underline{\circ}} (z_1,\dots ,z_{i-1}, u_i, z_{i+1}, \dots ,z_r)$,
\item There exists an internal node $v$ of $z_0$ and $1\leq i <r$ such that the subtree $t_v$ of $z$, whose root is $v$, is of the form $z_i \veebar z_{i+1}$ and $u$ is obtained by replacing $t_v$ in $z$ by the tree $(z_i \veebar z_{i+1}^l)\veebar z_{i+1}^r$, for $z_{i+1} = z_{i+1}^l\veebar z_{i+1}^r$. So, $u = z_0{\underline{\circ}} (z_1,\dots , z_i\veebar z_{i+1}^l, z_{i+1}^r, \dots , z_r)$. \end{enumerate} 
We may conclude that, for any $z = z_0{\underline{\circ}} (z_1,\dots ,z_r)$ and any $u\leq_T z$, we have that $u = u_0{\underline{\circ}} (u_1,\dots ,u_r)$, with $u_0\leq_T z_0$, and the trees $u_1,\dots ,u_r$ satisfy one of the following conditions  \begin{enumerate}[(i)]
\item $\vert u_i\vert = \vert z_i\vert $, for all integers $1\leq i\leq r$, and therefore $u_i\leq_T z_i$ for $1\leq i\leq r$,
\item there exists al least one integer $1\leq k\leq r-1$ such that $\vert u_k\vert < \vert z_k\vert$, and for any $1\leq j < r$, we have that $\sum_{i=1}^j\vert u_i\vert \leq \sum_{i=1}^j\vert z_i\vert$.\end{enumerate} 

Applying the result of the paragraph above to $t\leq_T u\leq_T w$, as $t = t_0{\underline{\circ}} (t_1,\dots ,t_r)$ and $w = w_0{\underline{\circ}} (w_1,\dots , w_r)$ with $\vert t_i\vert =n_i = \vert w_i\vert$ for all $1\leq i\leq r$, we conclude that $u = u_0{\underline{\circ}} (u_1,\dots , u_r)$ with $\vert u_i\vert = n_i$.
 \end{proof}
 \medskip

In \cite{LoRo1} and \cite{LoRo2}, the authors defined three associative graded products $/$, $\backslash$ and $*$  and a coproduct $\Delta $ on the vector space ${\mathbb K}\bigl[{\mbox{\bf PBT}}][-1]:=\oplus_{n\geq 1} {\mathbb K}\bigl[{\mbox{\bf PBT}_{n-1}}\bigr]$, where the degree of a tree is given by the number of its internal vertices. The product $*$ may be defined, as shown in \cite{LoRo2}, in terms of $/$, $\backslash$ and the Tamari order. Moreover, the data $({\mathbb K}\bigl[{\mbox{\bf PBT}}][-1], *, \Delta)$ is a coassociative bialgebra, and M. Aguiar and F. Sottile proved that 	
\begin{equation*} \Delta (M_t) = \sum _{t_1 /t_2 = t} M_{t_1}\otimes M_{t_2},\end{equation*}
for any $n\geq 2$ and any $t\in {\mbox{\bf PBT}_{n-1}}$, where $\{M_t\}_{t\in {\mbox{\bf PBT}_{n-1}}}$ is the Moebius basis of ${\mathbb K}\bigl[{\mbox{\bf PBT}_{n-1}}\bigr]$. 
Their formula provides a basis for the subspace of primitive elements of the coalgebra $({\mathbb K}[{\mbox{\bf PBT}}][-1], \Delta)$.
\medskip

Let us introduce a graded associative product  $\rightthreetimes $ and a coassociative graded coproduct $\overset{\circledast}{\Delta}$ on the vector space ${\mathbb K}\bigl[{\mbox{\bf PBT}}]$, and prove that Aguiar and Sottile\rq s formula holds in this new framework.

\begin{definition}  \label{def:topbot} \begin{enumerate}\item Define the binary product $\rightthreetimes $ on the vector space ${\mathbb K}\bigl[{\mbox{\bf PBT}}\bigr]$ by\begin{enumerate}
\item $1_{\K}\ \rightthreetimes\ t :=t =: t\ \rightthreetimes\ 1_{\K}$, for any planar binary rooted tree $t$, 
\item $t\ \rightthreetimes\  w = w\circ_1(t\veebar \vert)$, for any pair ot trees $t, w$.\end{enumerate} 

In fact, we have that $t\ \rightthreetimes \ w:= \bigl(t\ \rightthreetimes\  w^l\bigr)\veebar w^r = (\vert  \rightthreetimes w)\circ_1 t$, for $w = w^l\veebar w^r$ and $\vert w\vert \geq 2$. 

\item The linear map $\overset{\circledast}{\Delta}: {\mathbb K}[{\mbox{\bf PBT}}]\longrightarrow {\mathbb K}[{\mbox{\bf PBT}}]\otimes {\mathbb K}[{\mbox{\bf PBT}}]$ is defined recursively by\begin{enumerate}
\item $\overset{\circledast}{\Delta}(1_{\mathbb K} ):=1_{\mathbb K} \otimes 1_{\mathbb K}$,
\item $\overset{\circledast}{\Delta}(\vert):= 1_{\mathbb K} \otimes \vert + \vert \otimes 1_{\mathbb K}$,
\item In general, $\overset{\circledast}{\Delta}\bigl(t\veebar w\bigr) :=\sum t_{(1)}\otimes \bigl(t_{(2)}\veebar w\bigr) + \sum \bigl(t\veebar w_{(1)}\bigr)\otimes w_{(2)} - t \otimes w$, 
		
\noindent for any colored trees $t$ and $w$, where $\overset{\circledast}{\Delta}(t) = \sum t_{(1)}\otimes t_{(2)}$ for any planar binary rooted tree.\end{enumerate}
\end{enumerate}
\end{definition}
\medskip

It is immediate to check that the product $\rightthreetimes$ is associative and that $\overset{\circledast}{\Delta}$ is coassociative. 

\begin{remark} \label{rem:descriptcoprod} 
Given a planar rooted binary tree $t$ with $n$ leaves, assume that the internal vertices of $t$ are labelled, from left to right, by the elements of the set $\{1,\dots , n-1\}$. The coproduct has the following expression: 
\begin{equation*} \overset{\circledast}{\Delta}(t) = \sum_{i=0}^{n} t_{(1)}^i\otimes t_{(2)}^i, \end{equation*}
where $t_{(1)}^i$, respectively $t_{(2)}^i$, is the tree obtained by eliminating the vertex labeled with $i$ and keeping the tree on the left side, respectively keeping 
the tree on the right side, for $1\leq i < n$, $t_{(1)}^0 = 1_{\K} = t_{(2)}^n$ and $t_{(1)}^n = t = t_{(2)}^0$.  
\end{remark}
\medskip

For example, we have

\begin{center}
	\includegraphics[width=0.9\textwidth]{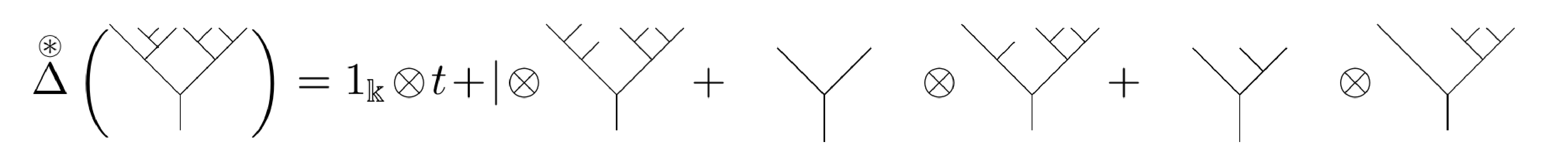}
\end{center}
\begin{center}
	\includegraphics[width=0.9\textwidth]{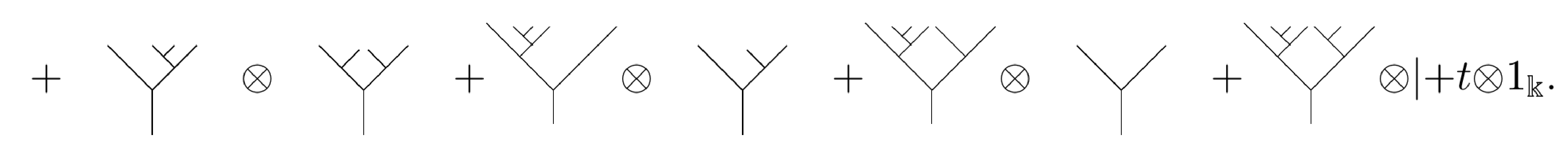}
\end{center}

\subsection{Aguiar-Sottile\rq s formula in ${\mathbb K}[{\mbox{\bf PBT}}]$}
\medskip

\begin{definition} \label{def:Moebiusbasis} For any tree $t \in {\mbox{\bf PBT}_n}$, let $M_t$ be the following element in ${\mathbb K}[{\mbox{\bf PBT}_n}]$,
\begin{equation*} M_t = \sum_{w\leq_T t}\mu(w; t) w ,\end{equation*}
where $\mu$ is the Moebius function of the Tamari order.\end{definition}
\medskip

Note that $M_t = t + \sum_{w<_T t}\mu(w; t) w$, which implies that the set $\{M_t\}_{t\in {\mbox{\bf PBT}_n}}$is a basis of ${\mathbb K}[{\mbox{\bf PBT}_n}]$, for $n\geq 2$.

We want to prove the following Lemma to give a recursive formula for $M_t$.

\begin{lemma} \label{lem:fund} Let $(P, \leq )$ be a finite lattice and $P\rq \subseteq P$ a subset closed under infimus ($\wedge$). Suppose that an element $p_0\in P\rq$ satisfies that any element $p\in P$ such that $p$ is covered by $p_0$ ($p\lessdot p_0$) belongs to $P\rq$, then  $\mu(q, p_0) = 0$ for any $q\notin P\rq$.\end{lemma}

\begin{proof} For $q\not\leq p_0$ the result is evident.

Assume that $q\notin P\rq$ and $q < p_0$. A {\it chain} from $q$ to $( P\rq, p_0)$ is a sequence 
\begin{equation*} q = q_1 <q_2 <\dots < q_r < p \leq p_0,\end{equation*}
such that $q_i\notin P\rq$, for $1\leq i\leq r$ and $p\in P\rq$.

Let $r$ be the largest positive integer such that there exists a chain from $q$ to $( P\rq, p_0)$ of length $r$. We proceed by recursion on $r$. As $q\notin P\rq$, we have that $r\geq 1$.

For $r = 1$, let $S_{q,P\rq}:=\{ p\in P\rq\mid q < p\}$, the element $p_1:=\wedge_{p\in S_{q,P\rq}} p$ belongs to $P\rq$ and $q\lessdot p_1$. Therefore, we have
\begin{equation*} \mu(q, p_0) = - \sum_{ p_1\leq p\rq\leq p_0}\mu (p\rq, p_0) =-\delta_{p_1,p_0}= 0 ,\end{equation*}
 because $p_1 <p_0$.
 \medskip

Assume that $r > 1$ and the result is true for any $q\rq\notin P\rq$ such that $q < q\rq < p_0$. As in the previous case, the element $p_1:=\wedge_{p\in S_{q,P\rq}} p$ belongs to $P\rq$ and $p_1 < p_0$, therefore
\begin{equation*} \mu (q, p_0)= - \sum_{q< z\leq p_0}\mu(z,p_0)= - \sum _{p_1 \leq z\leq p_0}\mu(z; p_0) = -\delta_{p_1,p_0}=0,\end{equation*}
which ends the proof.\end{proof}

\begin{notation} \label{notn:elementsinPB}  Let ${\mathbf 1}_{n}$ denote the largest element of ${\mbox{\bf PBT}_n}$ for the Tamari order and ${\mathbf 0}_{n}$ be the smallest one. We have that \begin{enumerate}
\item ${\mathbf 1}_{2} = \scalebox{0.05}{\includegraphics[width=0.6\textwidth]{dib4.eps}}
= {\mathbf 0}_{2}$ is the unique planar binary rooted tree with two leaves.
\item  ${\mathbf 1}_{n} = \vert \veebar  {\mathbf 1}_{n-1} $, while ${\mathbf 0}_{n} =  {\mathbf 0}_{n-1} \veebar \vert$, for $n\geq 3$.\end{enumerate}
\medskip

For any $n\geq 3$, a tree $t\in {\mbox{\bf PBT}_n}$ may be written as $t = {\mathbf 0}_{s}{\underline{\circ}} (\vert , t_2,\dots ,t_s)$, for a unique positive integer $s\geq 2$ and unique trees $t_i\in {\mbox{\bf PBT}_{n_i}}$, for $2\leq i\leq s$, 
with $\sum_{i=1}^sn_i  = n-1$.

In a similar way, there exist a unique way to write down a tree $t\in {\mbox{\bf PBT}_n}$, for $n\geq 3$, as $t = {\mathbf 1}_{r}{\underline{\circ}} (t_1,\dots, t_{r-1}, \vert)$, with $t_i\in {\mbox{\bf PBT}_{m_i}}$, for $1\leq i\leq r-1$, 
with $\sum_{i=1}^{r-1}m_i = n-1$.

\end{notation}
\medskip

For $n > 2$ consider the subsets $P_{1n} :=\{ \vert \veebar t\mid t\in {\mbox{\bf PBT}_{n-1}}\}$ and $P_{2n} := \{ \vert \rightthreetimes t\mid t\in {\mbox{\bf PBT}_{n-1}}\}$ of ${\mbox{\bf PBT}_n}$.

We denote by $P_n\rq$ the disjoint union of $P_{1n}$ and $P_{2n}$. Clearly, the tree ${\mathbf 1}_{n}$ belongs to $P_{1n}$ and any planar rooted tree $t$ covered by ${\mathbf 1}_{n}$ belongs to $P_n\rq$.
\begin{remark} We have that\begin{enumerate}[(a)]
\item $(\vert\veebar t)\ \wedge\ (\vert\veebar w ) = \vert \veebar (t\wedge w)$, for any pair of trees $t,w \in {\mbox{\bf PBT}_{n-1}}$,
\item $(\vert\ \rightthreetimes\  t)\ \wedge\ (\vert \rightthreetimes \ w ) = \vert\ \rightthreetimes \ (t\wedge w)$, for any $t,w\in {\mbox{\bf PBT}_{n-1}}$,
\item Any element $z\leq_T (\vert\ \rightthreetimes\ w)$ is of the form $z = \vert\ \rightthreetimes\ {\overline z}$ in ${\mbox{\bf PBT}_{n}}$. So the minimum 
$(\vert\veebar t)\ \wedge\ (\vert\ \rightthreetimes\ w)$ in ${\mbox{\bf PBT}_{n}}$ is of the form $\vert\ \rightthreetimes\ {\overline z}$ for some ${\overline z}\in  {\mbox{\bf PBT}_{n-2}}$, for any pair of trees $t\in  {\mbox{\bf PBT}_{n-1}}$ and $w\in  {\mbox{\bf PBT}_{n-1}}$.\end{enumerate}
Therefore, $P_n\rq $ is closed under the $\wedge$. 
\end{remark}
\medskip

To give a recursive formula for the element $M_{{\mathbf 1}_{n}}$, we need the following Lemma.

\begin{lemma} \label{lem:auxP} Let $t =  \vert\ \rightthreetimes\ {\overline t}$ be an element in $P_{2n}$, with ${\overline t} = {\mathbf 0}_{r-1}{\underline{\circ}} (\vert , t_3,\dots ,t_r) \in  {\mbox{\bf PBT}_{n-1}}$. 
Any tree $z\in P_{1n}$ such that $t <_T z$ satisfies that $t\rq = {\mathbf 0}_{r-1}{\underline{\circ}} (\vert , \vert \veebar t_3,\dots ,t_r) \leq _T z$.\end{lemma}

\begin{proof} Note that $t\lessdot t\rq$.

The result for $r=2$ is clear. To prove that $t\rq \leq_T z$, we proceed by induction on $r$.

For $r = 3$, assume that $z = \vert \veebar {\overline z}$. Since $t =  \vert\ \rightthreetimes\ {\overline t}$,
 it is immediate to see that $t\rq = \vert \veebar (\vert \veebar {\overline t} )\leq_T z$.
\medskip

For $r > 3$, there exists a tree $w = ( \vert\veebar w_1)\veebar w_2$ such that $t\leq_Tw\leq_T \vert\veebar (w_1\veebar w_2)\leq_T z$. In this case, there exists $2\leq j < r$ such that 
${\mathbf 0}_{i} {\underline{\circ}}(\vert ,\vert, t_3,\dots , t_j)\leq_T  \vert\veebar w_1$ and ${\mathbf 0}_{r-j+1} {\underline{\circ}} (\vert , t_{j+1},\dots , t_r)\leq_T \vert\veebar w_2$. Applying a recursive argument, we get that 
\begin{equation*} {\mathbf 0}_{i-1} {\underline{\circ}} (\vert ,\vert\veebar t_3,\dots , t_i)\leq_T \vert\veebar w_1.\end{equation*}

And therefore,  
\begin{equation*}  t\rq = ({\mathbf 0}_{i-1} {\underline{\circ}} (\vert ,\vert\veebar t_3,\dots , t_i)) / ({\mathbf 0}_{i-1} {\underline{\circ}} (\vert ,\vert\veebar t_3,\dots , t_i)) \leq_T  (\vert\veebar w_1) \ (\vert \veebar w_2)\leq _T z,\end{equation*} which proves that $t\rq \leq_T z$.
\end{proof}

\begin{proposition} \label{Moebiuisofcomb} The element $M_{{\mathbf1}_{n}}\in {\mathbb K}[{\mbox{\bf PBT}_n}]$ is defined recursively by the following formulas\begin{enumerate}
\item $M_{{\mathbf 1}_{2}} = \scalebox{0.05}{\includegraphics[width=0.6\textwidth]{dib4.eps}
}$,
\item $M_{{\mathbf 1}_{n}} = \vert \veebar M_{{\mathbf1}_{n-1}}\ -\ \vert\ \rightthreetimes\ M_{{\mathbf 1}_{n-1}}$.\end{enumerate}
Moreover, we have that ${\displaystyle M_{{\mathbf1}_{n}} = \sum_{i=1}^{n-1}(-1)^{i-1} {\mathbf 0}_{i} \veebar M_{{\mathbf 1}_{ n-i}}}$, for $n\geq 2$.\end{proposition} 

\begin{proof} 

Applying Lemma \ref{lem:fund}, we get that $\mu(t, {\mathbf 1}_{n}) = 0$ for any $t\notin P_n\rq$. So, we have that
\begin{equation*} M_{{\mathbf 1}_{n}} =\sum _{t\in P_{1n}}\mu(t, {\mathbf 1}_{n})t + \sum_{t\in P_{2n}} \mu(t, {\mathbf 1}_{n})t .\end{equation*}

For the first term, we get 
\begin{align*} \sum _{t\in P_{1n}}\mu(t, {\mathbf 1}_{n})t = & \sum_{{\overline t}\in {\mathbf {PBT}}_{n-1}} \mu (\vert \veebar {\overline t}, {\mathbf 1}_{n}) \left(\vert \veebar {\overline t} \right)=\\
&\sum_{{\overline t}\in {\mathbf {PBT}}_{n-1}} \mu ({\overline t}, {\mathbf 1}_{n-1})\left( \vert \veebar {\overline t} \right)= \vert \veebar M_{{\mathbf 1}_{n-1}}.\end{align*}

For the second one, let $t =  \vert\ \rightthreetimes\ {\overline t} \in P_{2n}$. We have that
\begin{align*} \mu (t , {\mathbf 1}_{n}) &= - \sum_{t <_T w\in P_{1n}}\mu (w, {\mathbf 1}_{n}) - \sum_{t <_T w\in P_{2n}}\mu (w, {\mathbf 1}_{n}) =\\
&- \sum_{t\rq \leq_T w\in P_{1n}}\mu (w, {\mathbf 1}_{n})  -  \sum_{{\overline t} <_T \bar{w}\leq_T {\mathbf 1}_{n-1}} \mu (\vert\ \rightthreetimes\ \bar{w} ,  {\mathbf 1}_{n}),\end{align*}
where $t =  {\mathbf 0}_{r}{\underline{\circ}} (\vert , \vert , t_3,\dots ,t_r)$ and $t\rq = {\mathbf 0}_{r-1}{\underline{\circ}} (\vert , \vert \veebar t_3,\dots ,t_r)$.

But 
 \begin{equation*}\sum_{t\rq \leq_T w\in P_{1n}}\mu (w, {\mathbf 1}_{n}) = \begin{cases} 1,& {\rm for}\ t\rq =  {\mathbf 1}_{n},\\
0,&\ {\rm otherwise.}\end{cases}\end{equation*}

So, we have that 
\begin{equation*} \mu (t , {\mathbf 1}_{n}) = -\delta_{t\rq,  {\mathbf 1}_{n}} - \sum_{t <_T w\in P_{2n}}\mu (w, {\mathbf 1}_{n}).\end{equation*}

For any $w\in P_{2n}$, let $h(w)$ be the number of elements $z\in P_{2n}$ such that $w\leq _T z\leq_T {\mathbf 1}_{n}$.
We want to prove that $\mu (w, {\mathbf 1}_{n}) = - \mu ({\overline w},  {\mathbf 1}_{n-1})$  for $w = \vert\ \rightthreetimes\ {\overline w}$ by induction in $h(w)$.
 \begin{enumerate}[(i)]
\item If $h(w) = 1$, then $w =  \vert\ \rightthreetimes\ {\mathbf 1}_{n-1}$ and $\mu (\vert\ \rightthreetimes\ {\mathbf 1}_{n-1} , {\mathbf 1}_{n}) = -1 = -\mu ( {\mathbf 1}_{n-1},  {\mathbf 1}_{n-1})$.
\item \bigbreak Assume that $\mu (w, {\mathbf 1}_{n}) = - \mu ({\overline w},  {\mathbf 1}_{n-1})$ with $h(w) < n$. 

Consider $w$ with $h(w)=n$, then,
\begin{equation*}\mu (w,  {\mathbf 1}_{n}) = - \sum_{{w< _T z\in P_{2n}}} \mu( z,  {\mathbf 1}_{n}).\end{equation*}
Using that $h(z) < h(w)$ and the inductive hypothesis,
\begin{equation*}\mu (w,  {\mathbf 1}_{n}) =\sum_{{\overline w} < _T{\overline z}\leq_T{\mathbf 1}_{n-1}}\mu ({\overline z}, {\mathbf 1}_{n-1}) = -\mu({\overline w}, {\mathbf 1}_{n-1}).\end{equation*}\end{enumerate}

The previous argument shows that
\begin{equation*}\sum_{t\in P_{2n}} \mu(t, {\mathbf 1}_{n})\left(\vert\ \rightthreetimes\ {\overline t}\right) = - \sum_{{\overline t}\in {\mathbf {PBT}_{n-1}}}\mu ({\overline t},  {\mathbf 1}_{n-1})\left(\vert\ \rightthreetimes\ {\overline t}\right)  =
-\ \vert\ \rightthreetimes\ M_{{\mathbf 1}_{n-1}} \end{equation*}
which ends the proof of the first formula.
\medskip

For the second point, it follows using a recursive argument. The result is trivial for $n = 2$. 

For $n\geq 2$, assume that the result is true for $n-1$, that is
\begin{equation*} M_{{\mathbf1}_{n-1}} = \sum_{i=1}^{n-2}(-1)^{i+1} {\mathbf 0}_{ i} \veebar M_{{\mathbf 1}_{ n-i-1}}.\end{equation*}

When we compute $M_{{\mathbf1}_{n}}$, we obtain
\begin{align*} M_{{\mathbf1}_{n}} = &\vert \veebar M_{{\mathbf1}_{n-1}} - \vert\ \rightthreetimes\  (\sum_{i=1}^{n-2} (-1)^{i-1}{\mathbf 0}_{ i} \veebar M_{{\mathbf 1}_{ n-i-1}})=\\
&M_{{\mathbf1}_{1}} \veebar M_{{\mathbf1}_{n-1}} + \sum_{i=1}^{n-2}(-1)^{i} ( {\mathbf 0}_{ i} \veebar M_{{\mathbf 1}_{n-i-1}})\circ _1 \scalebox{0.05}{\includegraphics[width=0.6\textwidth]{dib4.eps}
}=\\
&M_{{\mathbf1}_{1}} \veebar M_{{\mathbf1}_{n-1}} + \sum_{i=2}^{n-1}(-1)^{i-1} {\mathbf 0}_{i} \veebar M_{{\mathbf 1}_{n-i}} =\sum_{i=1}^{n-1}(-1)^{i-1}{\mathbf 0}_{ i} \veebar M_{{\mathbf 1}_{n-i}} \end{align*}

\end{proof}

We want now to give a recursive formula for $M_t$ for any planar binary rooted tree $t\in {\mathbf {PBT}_{n}}$. 

\begin{proposition} \label{prop:Mtforanyt} Let $t$ be a planar binary rooted tree with $n$ leaves, such that $ t = {\mathbf 1}_{r}{\underline{\circ}} (t_1,\dots , t_{r-1}, \vert)$, with $r < n$ and $t_j\in {\mathbf {PBT}_{n_j}}$, for $1\leq j < r$. 
The Moebius element is given by the recursive formula
\begin{equation*} M_t = M_{{\mathbf 1}_{r}}{\underline{\circ}} (M_{t_1},\dots ,M_{t_{r-1}}, \vert).\end{equation*}\end{proposition}

\begin{proof} For $t = {\mathbf 1}_{r} {\underline{\circ}} (t_1,\dots , t_{r-1},\vert)$, let 
\begin{equation*}P\rq :=\{ w{\underline{\circ}} (w_1,\dots , w_{r-1},\vert )\mid w\in {\mathbf {PBT}_{r}}\ {\rm and}\ w_i\leq_T t_i,\ {\rm for}\ 1\leq i < r\} \subseteq {\mathbf {PBT}_{n}}.\end{equation*}

Note that any tree $z\in {\mathbf {PBT}_{n}}$ satisfying that $z\lessdot t$ fulfills one of the following conditions\begin{enumerate}[(i)]
\item $z = {\mathbf 1}_{r} {\underline{\circ}} (t_1,\dots , z_i,\dots  , t_{r-1}, \vert)$ with $z_i\lessdot t_i$ for some $1\leq i < r$. 
\item $z = {\mathbf 1}_{r-1} {\underline{\circ}} (t_1,\dots, t_i\veebar t_{i+1}, \dots, t_{r-1}, \vert)$, for some $1\leq i < r-1$.  

\noindent In this case, $z = ({\mathbf 1}_{r-1}\circ_i \scalebox{0.04}{\includegraphics[width=0.5\textwidth]{dib4.eps}
}) {\underline{\circ}} (t_1,\dots ,t_{r-1}, \vert)$, with ${\mathbf 1}_{r-1}\circ_i \scalebox{0.04}{\includegraphics[width=0.5\textwidth]{dib4.eps}
}  \lessdot \ {\mathbf 1}_{r}$.\end{enumerate}

In both cases, $z\in P\rq$.
\medskip

Given two elements $w = w_0{\underline{\circ}} (w_1,\dots , w_{r-1}, \vert)$ and $z = z_0{\underline{\circ}} (z_1,\dots , z_{r-1}, \vert)$, consider the elements $w_0\wedge z_0\in {\mathbf {PBT}_{r}}$, and $w_i\wedge z_i\in {\mathbf {PBT}_{n_i}}$, for $1\leq i < r$. 

\noindent We want to prove that $q= (w_0\wedge z_0) {\underline{\circ}} (w_1 \wedge z_1, \dots ,w_{r-1}\wedge z_{r-1}, \vert )$ is $w\wedge z$. 
We know that $q\leq_{T} z$ and $q\leq_{T} w$. Consider $u$ such that $q\leq_{T} u\leq_{T} w\wedge z$. By the second part of Lemma \ref{lem:graftfortrees}, $u = u_0 {\underline{\circ}} (u_1,\dots , u_r)$ with $u_0\in {\mbox{\bf PBT}_r}$ and $u_i\in {\mbox{\bf PBT}_{n_i}}$, for $1\leq i\leq r$. Using the first part of Lemma \ref{lem:graftfortrees} $w_i\wedge z_i\leq_{T} u_i\leq_{T} w_i \wedge z_i$, for all $i=0,1,\dots r$ and $q=u$. So, $P\rq$ is closed under $\wedge$. 

Applying Lemma \ref{lem:fund}, we know that $\mu (q, t) = 0$, for $q\notin P\rq$. 

On the other hand, suppose that $w = w_0 {\underline{\circ}} (w_1,\dots ,w_{r-1}, \vert )\in P\rq$, an easy recursive argument shows that
\begin{equation*} \mu (w, t) = \mu(w_0, t_0) \cdot \dots \cdot \mu (w_{r-1}, t_{r-1}).\end{equation*}

Therefore, 
\begin{align*} M_t =& \sum_{w_i\leq_T t_i} \mu(w_0, t_0)\cdot \dots \cdot \mu(w_{r-1},t_{r-1})\ w_0{\underline{\circ}} (w_1,\dots ,w_{r-1},\vert ) =\\
& M_{t_0}{\underline{\circ}} (M_{t_1},\dots ,M_{t_{r-1}}, \vert),\end{align*}
which ends the proof.

\end{proof}
\medskip 

In the following example we compute some Moebius elements using Proposition \ref{Moebiuisofcomb}  and Proposition \ref{prop:Mtforanyt}.

\begin{example} Using Proposition \ref{Moebiuisofcomb}, we get 
\begin{center}
	\includegraphics[width=0.6\textwidth]{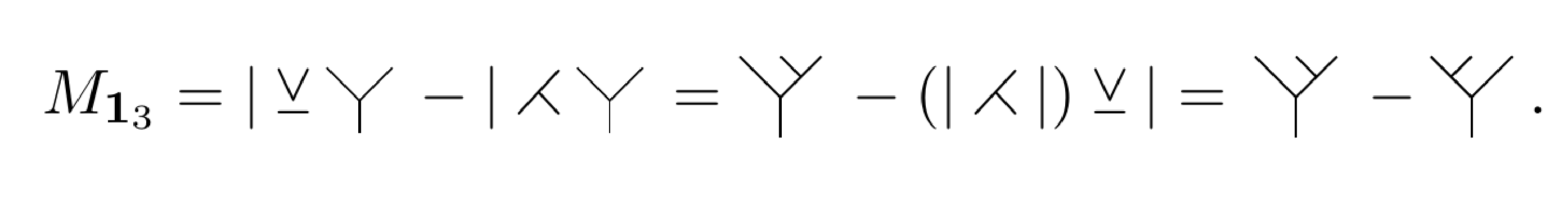}
\end{center}
	
	
Also, applying Proposition \ref{prop:Mtforanyt} to the tree: 
	
\begin{center}
\scalebox{0.4}{	\includegraphics[width=0.8\textwidth]{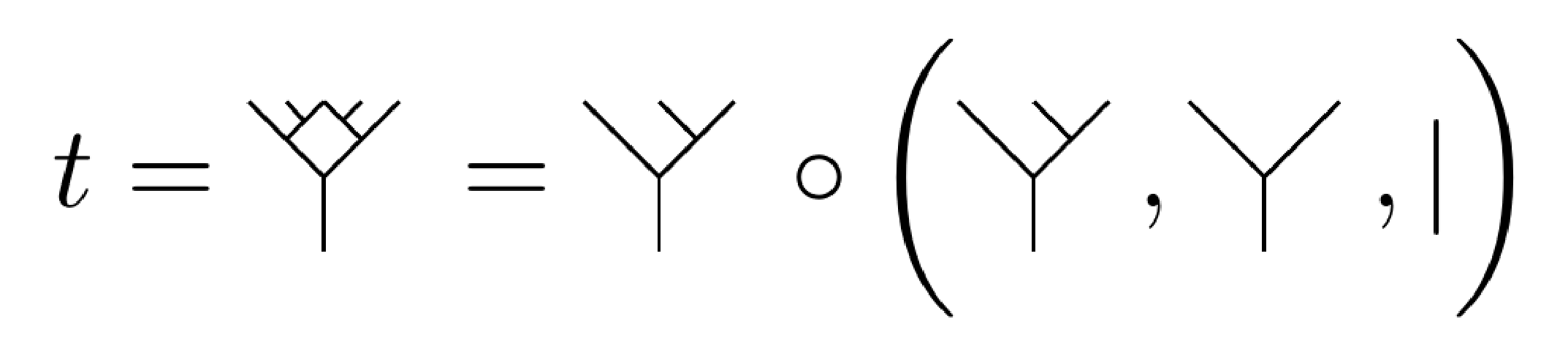}}
\end{center}
	
we have that
\begin{center}
	\scalebox{0.4}{	\includegraphics[width=1\textwidth]{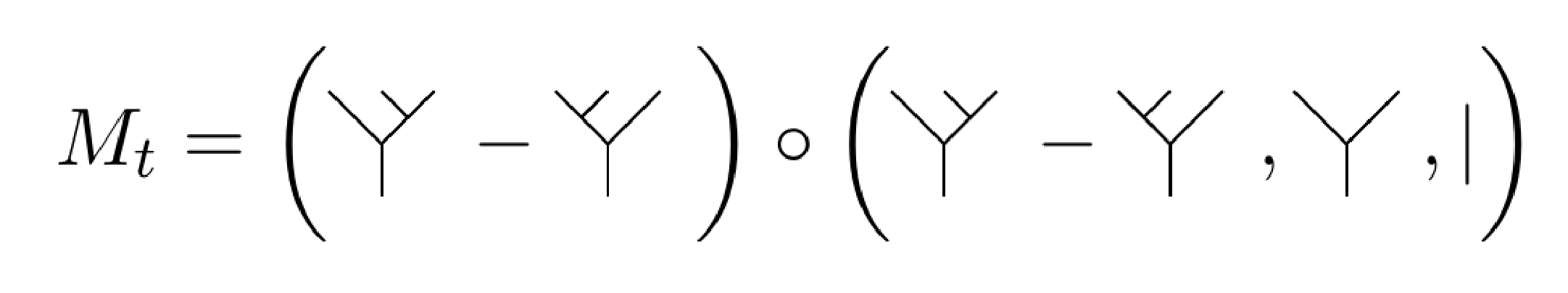}}
\end{center}
\begin{center}
	\scalebox{0.4}{	\includegraphics[width=1\textwidth]{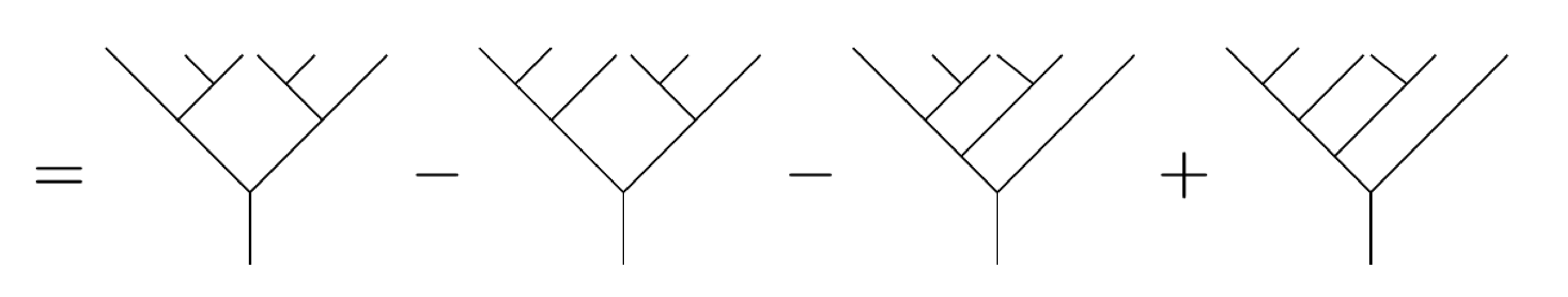}}
\end{center}
\end{example}

To prove the main result of this section, we need some previous definitions and lemmas. 
\medskip

For any $t\in {\mathbf {PBT}_{n}}$, we define $\delta(t) \in \bigoplus_{i=0}^{n-1} \K[{\mathbf {PBT}_{i+1}}]\otimes \K[{\mathbf {PBT}_{n-i}}]$ as 
\begin{equation*} \delta(t) := \sum_{i=0}^{n-1}t_{(1)}^{i+1}\otimes t_{(2)}^i,\end{equation*}
where $\overset{\circledast}{\Delta}(t) =\sum_{i=0}^{n} t_{(1)}^i\otimes t_{(2)}^i$ with $\vert t_{(1)}^i\vert = i$ and $\vert t_{(2)}^i\vert =  n-i$.
\medskip

\begin{lemma} \label{lem:adcoprod}  The coproducts $\overset{\circledast}{\Delta}$ and $\delta$ satisfy that\begin{enumerate}
\item $(\K[{\mathbf {PBT}}], \rightthreetimes , \overset{\circledast}{\Delta})$ is a unital infinitesimal bialgebra, 
\item $\delta $ is a coderivation for the product $\veebar$, that is 
\begin{equation*}\delta \circ \veebar = ({\mbox{id}}\otimes \veebar)\circ (\delta\otimes {\mbox{id}}) + (\veebar\otimes {\mbox{id}})\circ({\mbox{id}}\otimes \delta),\end{equation*}
\item $(\K[{\mathbf {PBT}}], \rightthreetimes , \delta)$ is an infinitesimal bialgebra.
\end{enumerate}
\end{lemma}
\medskip

Note that we cannot say that $(\K[{\mathbf {PBT}}], \veebar , \delta)$ is an infinitesimal bialgebra because $\veebar$ is not associative.

\begin{proof} We prove the first two points, the proof of the third one follows the same arguments than the second one. 

To see that $(\K[{\mathbf {PBT}}], \rightthreetimes , \overset{\circledast}{\Delta})$ is a unital infinitesimal bialgebra we need to prove that 
\begin{equation*} \overset{\circledast}{\Delta}\circ \rightthreetimes = ({\mbox{id}}\otimes  \rightthreetimes)\circ (\overset{\circledast}{\Delta}\otimes {\mbox{id}}) + (\rightthreetimes \otimes {\mbox{id}})\circ ({\mbox{id}}\otimes \overset{\circledast}{\Delta}) - {\mbox{id}}\otimes {\mbox{id}}.\end{equation*}

Let $t$ and $w$ be two planar rooted binary trees, we proceed by induction on the degree of $w$.

If $\vert w\vert =0$, then $w = 1_{\K}$. In this case the result is immediate because $t\rightthreetimes 1_{\K} = t$ for any $t$. 

For $w = \vert$ and $t\in  \K[{\mathbf {PBT}_{n}}]$ , we have that $t\rightthreetimes \vert = t\veebar \vert$. From Definition \ref{def:topbot} , we have that
\begin{align*}  \overset{\circledast}{\Delta}(t\rightthreetimes \vert ) &= \overset{\circledast}{\Delta}(t\veebar \vert ) =\\
&\sum_{i=0}^{n}t_{(1)}^i\otimes (t_{(2)}^i\veebar \vert) + (t\veebar \vert )\otimes 1_{\K} =\\
& \sum_{i=0}^{n}t_{(1)}^i\otimes (t_{(2)}^i\veebar \vert) + (t\rightthreetimes 1_{\K})\otimes \vert +(t \rightthreetimes \vert)\otimes 1_{\K} - t\otimes \vert =\\
&\qquad ({\mbox{id}}\otimes  \rightthreetimes)\circ (\overset{\circledast}{\Delta}(t)\otimes \vert ) + (\rightthreetimes \otimes {\mbox{id}})\circ (t\otimes \overset{\circledast}{\Delta}(\vert)) - t\otimes \vert,\end{align*}
where $ \overset{\circledast}{\Delta}(t) = \sum t_{(1)}\otimes t_{(2)}$, which proves the formula.
\medskip

If $\vert w\vert \geq 2$, then $w = w^l\veebar w^r$ with $\vert w^l\vert < \vert w\vert$. Note that $t \rightthreetimes w = (t\rightthreetimes w^l)\veebar w^r$. So, \begin{enumerate}[(i)]
\item the recursive hypothesis states that 
\begin{equation*} \overset{\circledast}{\Delta} ( t\rightthreetimes w^l) = \sum t_{(1)}\otimes (t_{(2)} \rightthreetimes w^l) +\sum (t \rightthreetimes w_{(1)}^l)\otimes w_{(2)}^l - t\otimes w^l,\end{equation*}
\item by definition $\overset{\circledast}{\Delta}$ satisfy 
\begin{equation*} \overset{\circledast}{\Delta}(w) = \sum w_{(1)}^l\otimes (w_{(2)}^l\veebar w^r) + (w^l\veebar w_{(1)}^r)\otimes w_{(2)}^r - w^l\otimes w^r,\end{equation*}
 \end{enumerate}

So, we get
\begin{align*} \overset{\circledast}{\Delta}(t\rightthreetimes w) &= \overset{\circledast}{\Delta}((t\rightthreetimes w^l)\veebar w^r) =\\
&\sum (t\rightthreetimes w^l)_{(1)}\otimes ( (t\rightthreetimes w^l)_{(2)}\veebar w^r) + \sum ((t\rightthreetimes w^l)\veebar w_{(1)}^r)\otimes w_{(2)}^r - (t\rightthreetimes w^l)\otimes w^r =\\
&\qquad \sum t_{(1)}\otimes ((t_{(2)}\rightthreetimes w^l)\veebar w^r) + \sum (t\rightthreetimes w_{(1)}^l)\otimes (w_{(2)}^l \veebar w^r) -t\otimes (w^l\veebar w^r) +\\
&\qquad \qquad  \sum ((t\rightthreetimes w^l)\veebar w_{(1)}^r)\otimes w_{(2)}^r -  (t\rightthreetimes w^l)\otimes w^r =\\
&\qquad \qquad \qquad \qquad \sum t_{(1)}\otimes (t_{(2)}\rightthreetimes w) + \sum (t\rightthreetimes w_{(1)})\otimes w_{(2)} -  t\otimes w,\end{align*}
which ends the proof.
\bigskip

Let us prove the second point. Suppose that $t\in {\mathbf {PBT}_n}$ and $w\in {\mathbf {PBT}_m}$. 

\noindent We have that $ \overset{\circledast}{\Delta}(t) =\sum_{i=0}^nt_{(1)}^i\otimes t_{(2)}^i$ and $ \overset{\circledast}{\Delta}(w) = \sum_{j=0}^mw_{(1)}^j\otimes w_{(2)}^j$, with $t_{(1)}^0 = w_{(1)}^0 = t_{(2)}^n = w_{(2)}^m = 1_{\K}$ and $t_{(1)}^1 = w_{(1)}^1 = t_{(2)}^{n-1} = w_{(2)}^{m-1}=\vert$.

By definition of $\overset{\circledast}{\Delta}$,
\begin{equation*}  \overset{\circledast}{\Delta}(t\veebar w) = \sum_{i=0}^n t_{(1)}^i\otimes (t_{(2)}^i\veebar w) + \sum_{j=0}^m (t\veebar w_{(1)}^j)\otimes w_{(2)}^j - t\otimes w, \end{equation*}
for $t\in {\mathbf {PBT}_n}$ and $w\in {\mathbf {PBT}_m}$, where $t_{(1)}^1=\vert = w_{(1)}^1$ and $t_{(2)}^{n-1}=\vert = w_{(2)}^{n-1}$.

Therefore,
\begin{align*} \delta (t\veebar w) =& \sum_{i=0}^{n-2} t_{(1)}^{i+1}\otimes (t_{(2)}^i\veebar w) +
 t\otimes (\vert \veebar w) + (t\veebar \vert)\otimes w + \sum_{j=1}^{m-1}(t\veebar w_{(1)}^{j+1})\otimes w_{(2)}^{j}=\\
&[({\mbox{id}}\otimes \veebar)\circ(\delta \otimes {\mbox{id}}) + (\veebar\otimes {\mbox{id}})\circ({\mbox{id}}\otimes \delta)](t\otimes w),\end{align*}
which implies the result.
\end{proof}

In fact, identifying $\vert$ with $1_{\K}$ in $\K[{\mathbf {PBT}}][-1]$, it is not difficult to see that $\delta$ coincides with the coproduct $\Delta$ defined in \cite{LoRo1}.
\medskip

\begin{lemma} \label{lem:zerosdelta} For any $n\geq 2$, we have that $\delta (M_{{\mathbf 1}_{n}}) = \vert \otimes M_{{\mathbf 1}_{n}} + M_{{\mathbf 1}_{n}}\otimes \vert$.\end{lemma}

\begin{proof} The result is immediate for $n = 2$. For $n > 2$ we apply a recursive argument. 

We know that $M_{{\mathbf 1}_{n}} = \vert \veebar M_{{\mathbf 1}_{n-1}} - \vert \rightthreetimes M_{{\mathbf 1}_{n-1}}$, for $n\geq 2$.

As $\delta (\vert) = \vert \otimes \vert$, applying Lemma \ref{lem:adcoprod}, we get
\begin{align*} \delta (M_{{\mathbf 1}_{n}}) &= \vert \otimes (\vert \veebar M_{{\mathbf 1}_{n-1}}) + (\vert\veebar \vert)\otimes M_{{\mathbf 1}_{n-1}} +\\
&(\vert \veebar M_{{\mathbf 1}_{n-1}})\otimes \vert - \vert\otimes (\vert \rightthreetimes  M_{{\mathbf 1}_{n-1}}) - (\vert \rightthreetimes \vert)\otimes M_{{\mathbf 1}_{n-1}}+
(\vert  \rightthreetimes M_{{\mathbf 1}_{n-1}})\otimes \vert =\\
&\vert \otimes M_{{\mathbf 1}_{n}} + M_{{\mathbf 1}_{n}}\otimes \vert,\end{align*}
which ends the proof.\end{proof}

We apply the previous result to describe the behaviour of $\overset{\circledast}{\Delta}$ with the grafting product ${\underline{\circ}}$.

\begin{proposition} \label{prop:graftdiag} For any pair of trees $t\in {\mathbf {PBT}_{n}}$ and $w$ and any integer $1\leq k\leq n$, we have that
\begin{equation*} \overset{\circledast}{\Delta}\bigl(t\circ_k w\bigr) = \sum_{i=0}^{k-1}t_{(1)}^i \otimes \bigl(t_{(2)}^i\circ_{k-i} w\bigr) + \sum  \bigl(t_{(1)}^k\circ_k{\overline {w}}_{(1)}\bigr)\otimes  \bigl(t_{(2)}^{k-1} \circ _1{\overline w}_{(2)}\bigr) + \sum_{i = k}^{n} \bigl(t_{(1)}^i\circ_k w\bigr)\otimes t_{(2)}^i,\end{equation*}
where ${\overset{\circledast}{\Delta}}\ ^{red}(w) = \overset{\circledast}{\Delta}(w) - 1_{\K}\otimes w - w\otimes 1_{\K} = \sum {\overline w}_{(1)}\otimes {\overline w}_{(2)}$ and $t\circ_k 0 = 0$, for any $1\leq k\leq n$. \end{proposition}               

\begin{proof} We proceed by induction on the number of leaves of $t$. For $t = \vert$, the result is immediate using that $t\circ_1 w = w$ and $ \overset{\circledast}{\Delta}(\vert) = 1_{\K}\otimes \vert + \vert\otimes 1_{\K}$.

For $n\geq 2$ we have that $t = t^l\veebar t^r$, with $\vert t^l\vert < n$ and $\vert t^r\vert < n$. If $1\leq k\leq \vert t^l\vert$, then 

\noindent $t\circ_k w = (t^l\circ_k w)\veebar t^r$. Using that $\veebar$ and the coproduct  $\overset{\circledast}{\Delta}$ satisfy the relation of unital infinitesimal bialgebras (even if $\veebar$ is not associative) and the inductive hypotesis  we get that,
\begin{align*} \overset{\circledast}{\Delta}\bigl(t\circ_k w\bigr) =& \sum_{i=1}^{k-1} t_{(1)}^{l,i}\otimes \bigl((t_{(2)}^{l, i}\circ_{k-i} w)\veebar t^r\bigr) + \sum \bigl(t_{(1)}^{l,k} \circ_k {\overline w}_{(1)}\bigr)\otimes \bigl((t_{(2)}^{l,k-1} \circ_1 {\overline w}_{(2)}) \veebar t^r\bigr) +\\
&\sum_{i=k}^{\vert t^l\vert}\bigl(t_{(1)}^{l,i}\circ_kw\bigr)\otimes \bigl(t_{(2)}^{l,i}\veebar t^r\bigr) + \sum_{i =1}^{\vert t^r\vert}\bigl((t^l\circ_k w\bigr)\veebar t_{(1)}^{r,i})\otimes t_{(2)}^{r,i} =\\
& \sum_{i=0}^{k-1}t_{(1)}^i \otimes \bigl(t_{(2)}^i\circ_{k-i} w\bigr) + \sum \bigl(t_{(1)}^k\circ _k {\overline w}_{(1)}\bigr)\otimes \bigl(t_{(2)}^{k-1}\circ_1 {\overline w}_{(2)}\bigr) +\sum_{i=k}^n \bigl(t_{(1)}^i\circ_k w\bigr) \otimes t_{(2)}^i,\end{align*}
which proves the formula. 
 
 The proof for $\vert t^l\vert < k$ is obtained in a similar way.\end{proof}
 \medskip
 
 \begin{corollary} \label{cor:delteofgraft} Let $t\in {\mathbf {PBT}_{n}}$ and $w_1,\dots ,w_n$ be a family of planar rooted trees. We have that
 \begin{align*}   \overset{\circledast}{\Delta}(t{\underline{\circ}} (w_1,\dots ,w_n)) =& \sum_{i= 0}^n (t_{(1)}^i{\underline{\circ}} (w_1,\dots ,w_i))\otimes (t_{(2)}^i{\underline{\circ}} (w_{i+1},\dots ,w_n) )+\\
& \sum _{j=1}^n ({\tilde {t}}_{(1)}^j {\underline{\circ}} (w_1,\dots  ,{\overline w}_{j(1)}))\otimes  ({\tilde {t}}_{(2)}^j {\underline{\circ}} ({\overline w}_{j(2)},\dots  ,w_n)),\end{align*}
 where $\displaystyle  \overset{\circledast}{\Delta}(t) = \sum_{i=0}^nt_{(1)}^i \otimes t_{(2)}^i$, $\displaystyle  \overset{\circledast}{\Delta}\ ^{red}(w) =\sum {\overline w}_{(1)}\otimes {\overline w}_{(2)}$, and
 $\displaystyle \delta(t) = \sum _{j=0}^{n-1}t_{(1)}^{j+1}\otimes t_{(2)}^j = \sum_{j=1}^n {\tilde {t}}_{(1)}^j\otimes {\tilde {t}}_{(2)}^j$, with $\vert  {\tilde {t}}_{(1)}^j\vert + \vert  {\tilde {t}}_{(2)}^j\vert = n+1$ for $1\leq j\leq n$.\end{corollary}

\begin{theorem} \label{th:formAS} For any tree $t\in {\mathbf {PBT}_{n}}$, we have that
\begin{equation*} 
\overset{\circledast}{\Delta}\bigl(M_t\bigr) = \sum_{z_1\rightthreetimes z_2 = t} M_{z_1}\otimes M_{z_2},
\end{equation*}
where $1_{\K}\rightthreetimes t = t = t \rightthreetimes 1_{\K}$ for any planar rooted tree $t$.
\end{theorem}

\begin{proof} We prove first the result for the elements $t ={\mathbf 1}_{n}$, $n\geq 1$. Note that if ${\mathbf 1}_{n} = t_1\rightthreetimes t_2$, then either $t_1=1_{\K}$ and $t_2= 
{\mathbf 1}_{n}$, or $t_1= {\mathbf 1}_{n}$ and $t_2 = 1_{\K}$. 
\medskip
	
For $n=1$ or $n=2$ the result is immediate.
\medskip 
	
For $n \geq 3$, the recursive argument asserts that $M_{{\mathbf 1}_{n-1}}$ is primitive. Using the description of the coproduct $\overset{\circledast}{\Delta}$ given in Definition \ref{def:topbot}, and applying Proposition \ref{Moebiuisofcomb} and Lemma \ref{lem:adcoprod}, we get that 
\begin{align*}  \overset{\circledast}{\Delta}\bigl( M_{{\mathbf 1}_{n}}\bigr)  =&   \overset{\circledast}{\Delta} (\vert \veebar M_{{\mathbf 1}_{n-1}} - \vert \rightthreetimes M_{{\mathbf 1}_{n-1}} ) =\\
1_{\K} \otimes M_{{\mathbf 1}_{n}} &+ \vert \otimes M_{{\mathbf 1}_{n-1}} - \vert \otimes M_{{\mathbf 1}_{n-1}} + M_{{\mathbf 1}_{n}}\otimes 1_{\K} =\\
&1_{\K} \otimes M_{{\mathbf 1}_{n}} + M_{{\mathbf 1}_{n}}\otimes 1_{\K}. \end{align*}
Therefore the result holds for the trees of type ${\mathbf 1}_{n}$.
\bigskip

Suppose that $t = {\mathbf 1}_{r}{\underline{\circ}} (w_1,\dots , w_{r-1}, \vert )$ for some collection of trees $w_1,\dots , w_{r-1}$ such that $\vert w_i\vert >1$ for at least one integer $i$, $1\leq i\leq r-1$. 

Proposition \ref{prop:Mtforanyt} states that $M_t = M_{{\mathbf 1}_{r}}{\underline{\circ}}(M_{w_1},\dots ,M_{w_{r-1}},\vert)$.

From the previous paragraph, we know that  $\overset{\circledast}{\Delta}\bigl(M_{{\mathbf 1}_{r}}\bigr) = 1_{\K}\otimes M_{{\mathbf 1}_{r}} + M_{{\mathbf 1}_{r}}\otimes 1_{\K}$, and Lemma \ref{lem:zerosdelta} states that $\delta (M_{{\mathbf 1}_{r}}) = \vert \otimes M_{{\mathbf 1}_{r}} + M_{{\mathbf 1}_{r}}\otimes \vert $. So, applying Corollary \ref{cor:delteofgraft} and using that $\overset{\circledast}{\Delta}\ ^{red}(\vert) = 0$, we get
\begin{equation*} \overset{\circledast}{\Delta}\bigl(M_t\bigr) = 1_{\K}\otimes M_t + M_t\otimes 1_{\K} + \sum (\vert {\underline{\circ}} {\overline {M_{w_1}}}_{(1)} )\otimes (M_{{\mathbf 1}_{T,r}}{\underline{\circ}} ({\overline {M_{w_1}}}_{(2)}, M_{w_2},\dots ,M_{w_{r-1}},\vert )),\end{equation*}
where
\begin{equation*}\overset{\circledast}{\Delta} (M_{w_1}) = 1_{\K}\otimes M_{w_1} + \sum {\overline{M_{w_1}}}_{(1)} \otimes  {\overline{M_{w_1}}}_{(2)}  + M_{w_1}\otimes 1_{\K}.\end{equation*}
\medskip

If $t = \vert$, there exist two ways to write down $\vert = z_1\rightthreetimes z_2$, either $z_1=1_{\K}$ and $z_2 = \vert$, or $z_1 =\vert$ and $z_2 = 1_{\K}$. As $M_{\vert} = \vert$ and 
$\overset{\circledast}{\Delta}(\vert) = 1_{\K}\otimes \vert + \vert \otimes 1_{\K}$, the result is clearly true.
\medskip

Let $t = {\mathbf 1}_{r}{\underline{\circ}} (w_1,\dots , w_{r-1}, \vert )$ be such that $\vert t\vert =n > 1$. We have that $t = w_1\veebar t^r$, with $t^r = {\mathbf 1}_{r-1}{\underline{\circ}} (w_2,\dots , w_{r-1}, \vert )$. 

Moreover, $t =  z_1\rightthreetimes z_2$ if, and only if, $z_1 = u_1$ and $z_2 = u_2\veebar t^r$, for a pair of elements $z_1$ and $z_2$ satisfying that 
$w_1 = u_1\rightthreetimes u_2$. As $\vert w_1\vert < n$, we may assume that 

\begin{equation*} \overset{\circledast}{\Delta}\bigl(M_{w_1}\bigr) = \sum_{u_1\rightthreetimes u_2 = w_1} M_{u_1}\otimes M_{u_2},\end{equation*}

Using equality $(3)$, we get
\begin{equation*}  \overset{\circledast}{\Delta}(M_t) = \sum _{u_1\rightthreetimes u_2 =w_1}M_{u_1} \otimes (M_{u_2}\veebar \vert) = \sum_{z_1\rightthreetimes z_2 = t}M_{z_1}\otimes M_{z_2},\end{equation*}
which ends the proof.
\end{proof} 
\medskip

\begin{definition} \label{defn:irred} A {\it $\rightthreetimes$-irreducible tree} is a tree $t\in {\mathbf {PBT}_{n}}$ satisfying that there does not exist a pair of planar rooted binary trees $u$ and $w$ such that $t = u\rightthreetimes w$.\end{definition}

\begin{remark} \label{rem:irrele} The tree $\vert \in {\mathbf {PBT}_{1}}$ is irreducible and there is no $\rightthreetimes$-irreducible tree in ${\mathbf {PBT}_{2}}$ because the tree $\scalebox{0.05}{	\includegraphics[width=0.5\textwidth]{dib4.eps}
}$ is $\vert \rightthreetimes\vert$.

For $n > 2$, it is easily seen that a tree $t\in {\mathbf {PBT}_{n}}$ is $\rightthreetimes$-irreducible if, and only if, $t = t^l\veebar t^r$ where $t^l$ is $\rightthreetimes$-irreducible and $\vert t^r\vert > 1$.

Moreover, a standard argument shows that any tree $t\in {\mathbf {PBT}_{n}}$ may be written in a unique way as $t = t_1\rightthreetimes t_2\rightthreetimes\dots \rightthreetimes t_r$ for a unique family of $\rightthreetimes$-irreducible trees $t_1,\dots t_r$ and a unique positive integer $r$. \end{remark}

We denote by ${\mbox {$\rightthreetimes$-Irr}_n}$ the set of all $\rightthreetimes$-irreducible trees in ${\mathbf {PBT}_{n}}$.

\begin{lemma} \label{lem:Moebirred} Let $t$ be a planar binary rooted tree. If $t=  t_1\rightthreetimes t_2\rightthreetimes\dots \rightthreetimes t_r$ with $t_i\in {\mbox {$\rightthreetimes$-Irr}_{n_i}}$ for $1\leq i\leq r$, then $M_t = M_{t_1}\rightthreetimes M_{t_2}\rightthreetimes\dots \rightthreetimes M_{t_r}
\bigskip$.\end{lemma}

\begin{proof} The result is evident for $r=1$ and for $t = {\mathbf 0}_n =\vert \rightthreetimes\dots \rightthreetimes \vert$, $n\geq 2$.
\medskip

Proceeding by recursion on $r$, it suffices to show that for any pair of trees $t_1\in {\mathbf {PBT}}_{n_1}$ and $t_2\in {\mathbf {PBT}}_{n_2}$, the element $M_{t_1\rightthreetimes t_2}$ is equal to $M_{t_1}\rightthreetimes M_{t_2}$.

For any internal node $v$ of a tree $z$ let $z_v$ be the subtree of $z$ whose root is $v$. Note that $w\lessdot z$ implies that there exists an internal node $v$ of $z$ where 
$z_v = z\rq\veebar (u_1\veebar u_2)$, such that $w$ is obtained from $z$ by replacing the subtree $z_v$ by the tree $(z\rq\veebar u_1)\veebar u_2$.

If a tree $z$ is of the form $z = z_1\rightthreetimes z_2 = z_2\circ_1(z_1\veebar \vert)$, then applying the argument above we get that any tree $w$ satisfying $w\lessdot z$ is of the form $w_1 \rightthreetimes w_2$ with $\vert w_i\vert = \vert z_i\vert$, for $i = 1, 2$, because $w$ cannot be obtained by modifying the subtree $z_1\veebar \vert$ of $z$. 

So, any element $w\leq _T t_1\rightthreetimes t_2$ is of the form $w = w_1\rightthreetimes w_2$, with $\vert w_i\vert = \vert z_i\vert$ and $w_i\leq _T z_i$, for $i = 1,2$, which ends the proof.
\end{proof}

\bigskip

\section{Magmatic Infinitesimal bialgebras}

Our aim is to compute the subspace of primitive element of a coassociative coalgebra, equipped with many magmatic products which satisfy the infinitesimal unital relation. Theorem \ref{th:formAS} gives a natural basis for the subspace of primitive elements of any free magmatic algebra.

 We recall the definition of unital infinitesimal magmatic bialgebra and the structure theorem for conilpotent unital infinitesimal magmatic bialgebras proved in \cite{HoLoRo}. 

\begin{definition}\label{def:unitinfibial} A {\it unital magmatic infinitesimal bialgebra} over a field $\K$ is a vector space $H$ equipped with a coassociative coproduct $\Delta: H\longrightarrow H\otimes H$ and a binary product $\cdot : H\otimes H\longrightarrow H$ satisfying that 
\begin{equation*} \Delta (x\cdot y) = \sum x_{(1)}\otimes (x_{(2)}\cdot y) + \sum (x\cdot y_{(1)})\otimes y_{(2)} - x\otimes y,\end{equation*}
for $x, y\in H$, where $\Delta (x) = \sum x_{(1)}\otimes x_{(2)}$.

Given a set $S$, an {\it $S$-magmatic algebra} is a vector space $A$ equipped with a family of binary products $\{\cdot_s\}_{s\in S}$.
A {\it $S$-unital magmatic infinitesimal bialgebra} is an $S$-magmatic algebra $(H, \{\cdot_s\}_{s\in S})$  equipped with a coassociative coproduct $\Delta$ satisfying that $(H, \cdot_s, \Delta)$ is a unital magmatic infinitesimal bialgebra, for $s\in S$.\end{definition}

For instance, the coalgebra $(\K[{\mathbf{PBT}_n}],  \overset{\circledast}{\Delta})$ equipped with the product $\veebar$ is an example of unital magmatic infinitesimal bialgebra, which is not a unital infinitesimal bialgebra because $\veebar$ is not associative.

 The following result was proved in \cite{HoLoRo}. 
 \begin{theorem} \label{teo:estmagmatic}
Any conilpotent magmatic infinitesimal bialgebra $H$ is isomorphic, as a coalgebra, to the cotensorial coalgebra $T^c({\mbox{Prim}(H)})$.\end{theorem}

\begin{corollary} \label{cor:primel} The subspace ${\mbox{Prim}(\K[{\mathbf {PBT}}])}$ of primitive elements of the coalgebra $(\K[{\mathbf {PBT}}], \overset{\circledast}{\Delta})$ is spanned by the set of elements 
$\{ M_t\mid t\in \bigcup_{n\geq 1}{\mbox {$\rightthreetimes$-Irr}_n}\}$.\end{corollary}

\begin{proof} Theorem \ref{th:formAS} implies that the subspace spanned by $\bigcup_{n\geq 1} {\mbox {$\rightthreetimes$-Irr}_n}$ is contained in the subspace of primitive elements of 
$\K[{\mathbf {PBT}}]$. 

On the other hand, we know that $\{M_t\mid t\in \bigcup_{n\geq 1} {\mathbf{PBT}_n}\}$ is a basis of $\K[{\mathbf {PBT}}]$ and that the vector space $\K[{\mathbf {PBT}}]$ is isomorphic to $T^c({\mbox{Prim}(\K[{\mathbf {PBT}}])})$. The result follows easily using that the product $\rightthreetimes$ is associative, and that $\K[{\mathbf {PBT}}]$ is isomorphic to
 $T^c(\bigoplus_{n\geq 1} \K[ {\mbox {$\rightthreetimes$-Irr}_n}])$.\end{proof}

\bigskip

\subsection{Free $S$-magmatic infinitesimal algebras}
\medskip

In the present section we work on planar binary rooted trees with the internal vertices colored by the elements of some set $S$. Recall that any element in ${\mbox{\bf PBT}_n}$ has $n-1$ internal vertices, for $n\geq 1$. 

\begin{notation} \label{notn:Strees} Let $S$ be a set. For $n\geq 2$, ${\mbox{\bf PBT}_n^{S}}$ denotes the set of planar rooted binary trees with $n$ leaves, with the internal vertices colored by the elements of $S$. For $n=1$, ${\mbox{\bf PBT}_n^{S}} = {\mbox{\bf PBT}_1} = \{ \vert\}$.

As in section $1$, we denote by ${\mbox{\bf PBT}^{S}}$ the disjoint union $\bigcup _{n\geq 1} {\mbox{\bf PBT}_n^{S}}$, and by ${\mathbb K}[{\mbox{\bf PBT}^{S}}]$ the ${\mathbb K}$-vector space spanned by the graded set ${\mbox{\bf PBT}_n^{S}}$.

An element in ${\mbox{\bf PBT}_n^{S}}$ is a pair ${\underline t} = (t, (s_1,\dots ,s_{n-1}))$, where $t$ is a planar binary rooted tree with $n$ leaves and $(s_1,\dots ,s_{n-1})\in S^{n-1}$ is the set of colors on the vertices of $t$ enumerated from left to right. \end{notation}

For example
\begin{center}
	\includegraphics[width=0.5\textwidth]{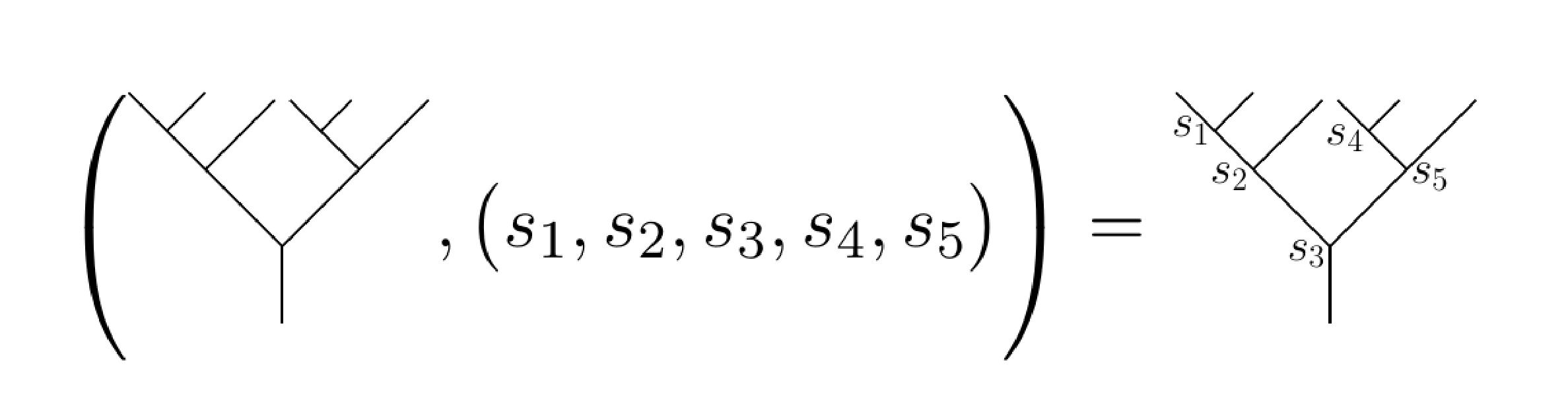}
\end{center}

\medskip

\begin{definition}\label{def:wedgecolored} Let $S$ be a set. For any pair of elements ${\underline t} =(t, (s_1,\dots ,s_{n-1}))$ and ${\underline w} = (w, (r_1,\dots ,r_{m-1}))$ in ${\mbox{\bf PBT}^{S}}$, and any $s\in S$,\begin{enumerate}[(a)]
\item  the $s$-wedge of ${\underline t}$ and ${\underline w}$ is the colored tree ${\underline t}\veebar_s{\underline w}$ obtained by:\begin{enumerate}[(i)]
\item the underline tree of ${\underline t}\veebar_s{\underline w}$ is $t\veebar s$,
\item the colors on the vertices of $t\veebar w$ are given by the sequence $(s_1,\dots, s_{n-1},s,r_1,\dots ,r_{m-1})$,
\end{enumerate}
\item the product $ {\underline t}\rightthreetimes _s {\underline w}$ is the colored tree 
\begin{equation*} {\underline t}\rightthreetimes _s {\underline w} : = (t \rightthreetimes s, (s_1,\dots, s_{n-1},s,r_1,\dots ,r_{m-1})).\end{equation*}
\item the {\it Tamari order} on the set  ${\mbox{\bf PBT}_n^{S}}$ is the product of the partially ordered sets $({\mbox{\bf PBT}_n}, \leq _T)$ and $(S^{\times(n-1)}, \leq)$, where $\leq$ is trivial order on the cartesian product.
In order to simplify notation, when no confusion may arise, we denote this order by $\leq_T$.
\end{enumerate}
\end{definition}

Both products $\veebar_s$ and $\rightthreetimes _s$ extend by linearity to binary products on the vector space ${\mathbb K}\bigl[{\mbox{\bf PBT}^{S}}\bigr]$, for $s\in S$. 

Let $\K[S]$ denote the vector space spanned by the set $S$. For any tree $t\in {\mbox{\bf PBT}_n}$ and elements $x_1,\dots ,x_{n-1}\in \K[S]$ with $x_i=\sum_{s\in S}a_{s}^is$ for $1\leq i\leq n$, we define the element 
\begin{equation*} (t, (x_1,\dots ,x_{n-1})) :=\sum_{(s_1,\dots ,s_n)\in S^n}a_{s_1}^1\dots a_{s_{n-1}}^{n-1}
  (t, (s_1, \dots , s_{n-1}))\end{equation*}
in ${\mbox{\bf PBT}_n^{S}}$.

\begin{remark}\label{rem:graftcolored} Note that\begin{enumerate}[(a)] 
\item the binary operation $\rightthreetimes _s$ is a graded associative product, for $s\in S$. Moreover, we have that
\begin{equation*} \rightthreetimes _r\circ (\rightthreetimes _s \otimes {\mbox{id}}) = \rightthreetimes _s\circ ({\mbox{id}}\otimes \rightthreetimes _r),\end{equation*}
for any elements $r, s\in S$.
\item for any colored tree ${\underline t}\in {\mbox{\bf PBT}^{S}}$ there exists a unique element $s\in S$ and unique colored trees ${\underline t^l}$ and ${\underline t^r}$ in ${\mbox{\bf PBT}^S}$ satisfying that ${\underline t} = {\underline t^l}\veebar_s {\underline t^r}$.
 \end{enumerate}
\end{remark}

Clearly, the free $S$-magmatic algebra over one element is the vector space ${\mathbb K}\bigl[{\mbox{\bf PBT}^{S}}\bigr]$ equipped with the products $\{\veebar _s \}_{s\in S}$, where we identify the generator with the tree $\vert \in {\mbox{\bf PBT}_1}$. 
\medskip

For any vector space $V$, the free $S$-magmatic algebra over $V$ is the vector space $\bigoplus_{n\geq 0}{\mathbb K}\bigl[{\mbox{\bf PBT}_n^{S}}\bigr]\otimes V^{\otimes n}$ with the product $\veebar_s$ given by
\begin{align*}
(t, (s_1, \dots ,s_{n-1}))\otimes (&v_1\otimes \dots \otimes v_n)\ \veebar_s \ (w, (r_1, \dots ,r_{m-1}))\otimes (u_1\otimes \dots \otimes u_m) :=\\
&(t\veebar w, (s_1,\dots, s_{n-1},s, r_1,\dots ,r_{m-1}))\otimes (v_1\otimes \dots \otimes v_n\otimes u_1\otimes \dots \otimes u_m),\end{align*}
for $s\in S$, we denote it by $\mbox{{\it S}-Mag}_{\K}(V)$.
\medskip

For any vector space $V$, the coproduct introduced in Definition \ref{def:topbot} extends to $\mbox{{\it S}-Mag}_{\K}(V)$ as follows:
\begin{align*}
{\overset{\circledast}{\Delta}}&((t, (s_1,\dots, s_{n-1}))\otimes (v_1\otimes \dots \otimes v_n)) =\\
&\sum_{i=0}^{n} ((t_{(1)}^i, (s_1,\dots , s_{i-1}))\otimes (v_1\otimes \dots \otimes v_i))\otimes ((t_{(2)}^i, (s_{i+1},\dots, s_{n-1}))\otimes (v_{i+1}\otimes \dots \otimes v_n)),
\end{align*} 
where the notation is the one introduced in Remark \ref{rem:descriptcoprod}, for any tree $t\in {\mbox{\bf PBT}_n}$, any elements $(s_1,\dots ,s_{n-1})\in S^{\times (n-1)}$ and $v_1,\dots ,v_n\in V$. 
\medskip

The following result is a straightforward consequence of Lemma \ref{lem:adcoprod}

\begin{lemma} \label{lem:extofcoprod}  Let $S$ be a set, for any vector space $V$ and any element $s\in S$, the data \\$({\mbox{{\it S}-Mag}_{\K}(V)}, \veebar_s , {\overset{\circledast}{\Delta}})$ is a magmatic infinitesimal bialgebra .
\end{lemma}
\bigskip

\subsection{Primitive elements and Aguiar-Sottile\rq s formula}

Let $S$ be a set. From Theorem \ref{th:formAS} and Lemma \ref{lem:extofcoprod} , we get that the element 
$(M_{t}, (s_1,\dots ,s_n))$ is primitive in ${\mathbb K}\bigl[{\mbox{\bf PBT}_{n+1}^{S}}\bigr]$ for any $t\in {\mbox{$\rightthreetimes$-Irr}_n}$ and any $(s_1,\dots ,s_n)\in S^{\times n}$. 
\medskip

Let $S$ be a non-empty set and let $s_{\bullet}\in S$ be a fixed element. For any pair of colored trees ${\underline t}$ and ${\underline w}$ in ${\mathbb K}\bigl[{\mbox{\bf PBT}^{S}}\bigr]$, and any $s\neq s_{\bullet}$, define 
\begin{equation*} {\underline t}\veebar_{\hat{s}}{\underline w} :=  {\underline t}\veebar_s {\underline w} - {\underline t}\veebar_{s_{\bullet}}{\underline w},\ {\rm for}\ s\neq s_{\bullet}.
\end{equation*} 

It is immediate to see that, as the coproduct ${\overset{\circledast}{\Delta}}$ satisfies the unital infinitesimal condition with any product $\veebar _s$, the element 
${\underline t}\veebar_{\hat{s}} {\underline w}$ is primitive whenever ${\underline t}$ and ${\underline w}$ are primitive elements in ${\mathbb K}\bigl[{\mbox{\bf PBT}^{S}}\bigr]$.
\medskip

\begin{remark} \label{rem:competc} Let ${\underline t} \in {\mathbb K}\bigl[{\mbox{\bf PBT}_n^{S}}\bigr]$ and let $V$ be a vector space. If ${\underline t}$ is a primitive element in ${\mathbb K}\bigl[{\mbox{\bf PBT}^{S}}\bigr]$, then $({\underline t}, v_1\otimes \dots \otimes v_n))$ is a primitive element in $\mbox{{\it S}-Mag}_{\K}(V)$ for any collection of elements $v_1,\dots ,v_n\in V$.
\medskip

Moreover, Corollary \ref{cor:delteofgraft} implies that if ${\underline t}, {\underline w}_2,\dots  , {\underline w}_{n-1}$ is a family of primitive elements in ${\mathbb K}\bigl[{\mbox{\bf PBT}^{S}}\bigr]$ such that $\vert{\underline t}\vert =n$, then the element ${\underline t} {\underline{\circ}} (\vert, {\underline w}_2,\dots ,{\underline w}_{n-1}, \vert)$ is primitive too.\end{remark}

\bigskip

We want to compute the subspace of primitive elements of the coalgebra $({\mathbb K}\bigl[{\mbox{\bf PBT}^{S}}\bigr],  {\overset{\circledast}{\Delta}})$ for any set $S$.  From Theorem \ref{teo:estmagmatic} we know that ${\mathbb K}\bigl[{\mbox{\bf PBT}^{S}}\bigr]$ is isomorphic as a coalgebra to the cotensor coalgebra $T^c({\mbox{Prim}}({\mathbb K}\bigl[{\mbox{\bf PBT}^{S}}\bigr]))$.

\medskip

\begin{definition} \label{def:idealprim} For $n\geq 1$, define the subset ${\mathcal I}_n^S$ of $\K[ {\mbox{\bf PBT}_n^{S}}]$ as follows:\begin{enumerate}[(i)]
\item ${\mathcal I}_1^S:=\{\vert\} = {\mbox{\bf PBT}_1^{S}}$, 
\item ${\mathcal I}_2^S:=\{ \vert \veebar_{\hat{s}}\vert\}_{s\in S\setminus \{s_{\bullet}\}}$,
\item for $n>2$, ${\mathcal I}_n^S$ is the union of the subsets \begin{itemize}
\item $\{ (M_{t}, (s_1,\dots ,s_{n-1}))\mid t\in  {\mbox{$\rightthreetimes$-Irr}_{n}}, (s_1,\dots ,s_{n-1})\in S^{n-1}\}$,
\item $\{ {\underline x}  {\underline{\circ}} ({\underline w}_1,\dots .{\underline w}_r)\mid {\underline x}, {\underline w}_1,\dots , {\underline w}_r\in \bigcup_{i=1}^{n-1}{\mathcal I}_i^S,\ {\rm with}\ 
\sum_{j=1}^r \vert w_j\vert = n\}$,\end{itemize}\end{enumerate}
and let $\K[{\mathcal I}_n^S]$ be the subspace of $\K[{\mbox{\bf PBT}_n^{S}}]$ spanned by ${\mathcal I}_n^S$.\end{definition}

Theorem \ref{th:formAS} and Corollary \ref{cor:delteofgraft} state that the subspace $\bigoplus_{n\geq 1} \K[{\mathcal I}_n]$ is contained in ${\mbox{Prim}}({\mathbb K}\bigl[{\mbox{\bf PBT}^{S}}\bigr])$. Our aim is to see that they are equal.
\medskip

\begin{lemma} \label{lem:propfirststep} Let $S$ be a set and let $s_{\bullet}\in S$ be a fixed element. For $n\geq 0$ and $s\in S$, $s\neq s_{\bullet}$, the colored tree $(\vert\veebar {\mathbf 0}_{n}, (s, s_{\bullet},\dots ,s_{\bullet}))\in {\mbox{\bf PBT}_{n+1}^{S}}$ satisfies the equality
\begin{equation*}(\vert\veebar{\mathbf 0}_{n}, (s, s_{\bullet},\dots ,s_{\bullet})) =\sum_{j=1}^{n+1}({\mathbf 0}_{j}, (s_{\bullet},\dots ,s_{\bullet}))\circ_1 (M^{n+2-j}, (s, s_{\bullet}, \dots ,s_{\bullet})),\end{equation*}
where $(M^k, (s,s_{\bullet},\dots ,s_{\bullet}))$ is defined as follows:\begin{enumerate}[(i)]
\item $(M^1,\emptyset ) := \vert$ and $(M^2, (s)):=\vert \veebar_{\hat {s}} \vert = (\scalebox{0.05}{	\includegraphics[width=0.4\textwidth]{dib4.eps}
}, ({\hat{s}}))$ ,
\item $(M^k, (s, s_{\bullet},\dots ,s_{\bullet})) := (M_{\vert\veebar {\mathbf 0}_{k-1}}, (s, s_{\bullet},\dots ,s_{\bullet}))$, for $k\geq 3$.\end{enumerate}\end{lemma}

Note that the element $M^k \in {\mathcal I}_{k}^S$, for $k\geq 1$.

\begin{proof} For $n =1$, we have that 
\begin{align*}(\vert\veebar{\mathbf 0}_{1}, (s))  &= (\scalebox{0.05}{	\includegraphics[width=0.4\textwidth]{dib4.eps}
}, (s)) =\\
&(M^2, (s)) + ({\mathbf 0}_{2}, (s_{\bullet}))= \vert \circ_1 (M^2, (s)) +  ({\mathbf 0}_{2}, (s_{\bullet}))\circ_1 (M^1, \emptyset), \end{align*} and the result is true.

Suppose that the formula holds for all $0\leq r <  n$, we have that 
\begin{equation*} M_{\vert \veebar {\mathbf 0}_{n} } = \vert \veebar {\mathbf 0}_{n} - (\vert \veebar  {\mathbf 0}_{n-1} )\veebar \vert.\end{equation*}
So, we get that
\begin{equation*} (\vert \veebar {\mathbf 0}_{n}, (s,s_{\bullet},\dots ,s_{\bullet})) = (M_{\vert \veebar {\mathbf 0}_{n} }, (s,s_{\bullet},\dots ,s_{\bullet})) + (\vert \veebar  {\mathbf 0}_{n-1}, (s, s_{\bullet},\dots ,s_{\bullet}) )\veebar_{s_{\bullet}} \vert.\end{equation*}

Applying a recursive argument. we have that
\begin{align*}(\vert \veebar {\mathbf 0}_{n-1}, (s, s_{\bullet},\dots ,s_{\bullet}) )\veebar_{s_{\bullet}} \vert &=
\sum_{j=1}^n(({\mathbf 0}_{j}, (s_{\bullet},\dots ,s_{\bullet}))\circ_1 (M_{s}^{n+1-j}, (s, s_{\bullet}, \dots ,s_{\bullet}))) \veebar_{s_{\bullet}} \vert =\\
&\sum_{j=1}^n ({\mathbf 0}_{j+1}, (s_{\bullet},\dots ,s_{\bullet}))\circ_1 (M_{s}^{n+1-j}, (s, s_{\bullet}, \dots ,s_{\bullet}))=\\
&\sum_{j=2}^{n+1} ({\mathbf 0}_{j}, (s_{\bullet},\dots ,s_{\bullet}))\circ_1 (M_{s}^{n+2-j}, (s, s_{\bullet}, \dots ,s_{\bullet})).\end{align*}

Therefore, we get
\begin{align*} (\vert \veebar {\mathbf 0}_{n}, (s, s_{\bullet},\dots ,s_{\bullet})) &=
 (M_{\vert \veebar {\mathbf 0}_{n} }, (s,s_{\bullet},\dots ,s_{\bullet})) + (\vert \veebar  {\mathbf 0}_{n-1}, (s, s_{\bullet},\dots ,s_{\bullet}) )\veebar_{s_{\bullet}} \vert=\\
&\sum_{j=1}^{n+1} ({\mathbf 0}_{j}, (s_{\bullet},\dots ,s_{\bullet}))\circ_1 (M^{n+2-j}, (s, s_{\bullet}, \dots ,s_{\bullet})),\end{align*}
which ends the proof.
\end{proof}

\medskip

We use the previous lemma to get the following result.

\begin{proposition}\label{prop:structtheo} 
Let $S$ be a non-empty set and $s_{\bullet}\in S$. Any element in $\K[ {\mbox{\bf PBT}_n^{S}}]$ is a linear combination of elements of type
\begin{equation*} ({\mathbf 0}_{r}, (s_{\bullet},\dots ,s_{\bullet})){\underline{\circ}} ({\underline t}_1,\dots ,{\underline t}_r),\end{equation*}
where ${\underline t}_i\in \K[{\mathcal I}_{m_i}^S]$ for $1\leq i \leq r$, with $\sum_{i=1}^r m_i = n$.\end{proposition}

\begin{proof} For $n=1$ the result is immediate because $\vert \in {\mathcal I}_1^S$.

For $s\neq s_{\bullet}$,  the tree 
\begin{equation*} (\scalebox{0.05}{	\includegraphics[width=0.4\textwidth]{dib4.eps}
}, (s))= (\scalebox{0.05}{	\includegraphics[width=0.4\textwidth]{dib4.eps}
}, ({\hat{s}})) + ({\mathbf 0}_{2}, (s_{\bullet})),\end{equation*}
 and for $s=s_{\bullet}$, the tree $(\scalebox{0.05}{	\includegraphics[width=0.4\textwidth]{dib4.eps}
}, (s_{\bullet})) = ({\mathbf 0}_{2}, (s_{\bullet}))$. 
\medskip

For $n > 2$, assume that any element in $\K[ {\mbox{\bf PBT}_r^{S}}]$ satisfies the result, for $ r < n$. As any colored tree ${\underline t} = {\underline t}^l\veebar_{s}{\underline t}^r$, with $\vert {\underline t}^l\vert < \vert {\underline t}\vert $ and $\vert {\underline t}^r\vert < \vert {\underline t}\vert $, we have that ${\underline t}^l$ and ${\underline t}^r$ satisfy the proposition. 
\medskip

Moreover, as the subspace $\K[{\mathcal I}^S] :=\bigoplus_{n\geq 1} \K[{\mathcal I}_n^S]$ is closed under the composition operators $\circ _i$ and ${\underline{\circ}}$, it suffices to prove that the result is true for any tree of the form 
${\mathbf 0}_{p}\veebar_s{\mathbf 0}_{q}$, for integers $p, q\geq 0$.
\medskip

We proceed by induction on $p$. For $p = 1$, the tree $t =\vert\veebar_s{\mathbf 1}_{q}$ and the result is a straightforward consequence of Lemma \ref{lem:propfirststep}.

For $p > 1$, as $M_{{\bf 0}_{q}} = {\bf 0}_{q}$, Proposition \ref{prop:Mtforanyt} states that
\begin{equation*} M_{{\mathbf 1}_{p+1} \circ_{p+1}{\mathbf 0}_{q}} = M_{{\mathbf  1}_{p+1}}\circ _{p+1}{\bf 0}_{q}.\end{equation*} 

Propositions \ref{prop:Mtforanyt} and \ref{Moebiuisofcomb} imply that
\begin{align*} M_{{\mathbf 1}_{p+2}\circ_{p+1}{\mathbf 0}_{q-1}} &= \sum_{i=1}^{p+1}(-1)^{i-1} ( {\mathbf 0}_{ i}\veebar M_{{\mathbf 1}_{p+2-i}})\circ_{p+1} {\mathbf 0}_{q-1} =\\
\sum_{i=1}^{p-1}(-1)^{i-1} {\mathbf 0}_{ i}\veebar &(M_{{\mathbf 1}_{p+2-i}}\circ _{p+1-i} {\mathbf 0}_{q-1})  + (-1)^{p-1} {\mathbf 0}_{p}\veebar (M_{{\mathbf 1}_{2}}\circ _1 {\mathbf 0}_{q-1}) + (-1)^p ({\mathbf 0}_{p+1}\veebar \vert)\circ_{p+1}
{\mathbf 0}_{q-1}=\\
& \sum_{i=1}^{p-1}(-1)^{i-1} {\mathbf 0}_{i}\veebar (M_{{\mathbf 1}_{p+2-i}}\circ _{p+1-i} {\mathbf 0}_{q-1}) + (-1)^{p-1} {\mathbf 0}_{p}\veebar {\mathbf 0}_{q} + (-1)^p ({\mathbf 0}_{p}\veebar {\mathbf 0}_{q-1})\veebar \vert,\end{align*}
which implies that 
\begin{align*} & ({\mathbf 0}_{p}\veebar  {\mathbf 0}_{q}, (s_{\bullet},\dots, s_{\bullet},s,s_{\bullet},\dots ,s_{\bullet}))  =  ({\mathbf 0}_{p}, (s_{\bullet},\dots ,s_{\bullet}))\veebar_s ({\mathbf 0}_{q}, (s_{\bullet},\dots ,s_{\bullet}))=\\
&(-1)^{p-1}\bigl((M_{{\mathbf 1}_{p+2}\circ_{p+1}{\mathbf 0}_{q-1}}, (s_{\bullet},\dots, s_{\bullet},s,s_{\bullet},\dots ,s_{\bullet})) +\\
& \sum_{i=1}^{p-1}(-1)^{p+i-1} ({\mathbf 0}_{i}, (s_{\bullet},\dots ,s_{\bullet}))\veebar _{s_{\bullet}}((M_{{\mathbf 1}_{p+2-i}},(s_{\bullet},\dots ,s_{\bullet}, s, s_{\bullet}))\circ _{p+1-i} ({\mathbf 0}_{q-1}, (s_{\bullet},\dots ,s_{\bullet}))\bigr) +\\
& ({\mathbf 0}_{ p},(s_{\bullet},\dots ,s_{\bullet}))\veebar_s ({\mathbf 0}_{q-1}, (s_{\bullet},\dots ,s_{\bullet})))\veebar_{s_{\bullet}} \vert.\end{align*}

Note that \begin{enumerate}[(i)]
\item the element $(M_{{\mathbf 1}_{p+2}\circ_{p+1}{\mathbf 0}_{q-1}}, (s_{\bullet},\dots, s_{\bullet},s,s_{\bullet},\dots ,s_{\bullet}))$ belongs to ${\mathcal I}_{p+q}^S$,
\item as $\vert (M_{{\mathbf 1}_{p+2-i}},(s_{\bullet},\dots ,s_{\bullet}, s, s_{\bullet}))\circ _{p+1-i} ({\mathbf 0}_{q-1}, (s_{\bullet},\dots ,s_{\bullet}))\vert < n$, we assume that it satisfies the proposition. So, applying a recursive argument on $i$, we get that \begin{equation*} ({\mathbf 0}_{i}\veebar (M_{{\mathbf 1}_{p+2-i}}\circ _{p+1-i} {\mathbf 0}_{q-1}, (s_{\bullet},\dots ,s_{\bullet},s,s_{\bullet},\dots ,s_{\bullet})))\end{equation*}
is a linear combination of elements of type
\begin{equation*} ({\mathbf 0}_{r}, (s_{\bullet},\dots ,s_{\bullet})){\underline{\circ}} ({\underline t}_1,\dots ,{\underline t}_r),\end{equation*} for $1\leq i\leq p-1$.
\item As $ ({\mathbf 0}_{r}(s_{\bullet},\dots , s_{\bullet}))\veebar_s ({\mathbf 0}_{q-1}, (s_{\bullet},\dots ,s_{\bullet}))$ satisfies the proposition, the element
\begin{equation*} ({\mathbf 0}_{ p},(s_{\bullet},\dots ,s_{\bullet}))\veebar_s ({\mathbf 0}_{q-1}, (s_{\bullet},\dots ,s_{\bullet})))\veebar_{s_{\bullet}} \vert,\end{equation*} 
satisfies it, too. 
\end{enumerate}

\end{proof}
\medskip

\begin{theorem} \label{th:structure} For any set $S$, the subspace of primitive elements of the coalgebra $(\K[ {\mbox{\bf PBT}^{S}}], {\overset{\circledast}{\Delta}})$ is the vector space $\K[{\mathcal I}^S]$. Moreover, the coalgebra $(\K[ {\mbox{\bf PBT}^{S}}], {\overset{\circledast}{\Delta}})$ is isomorphic to the cotensor coalgebra $T^c(\K[{\mathcal I}^S])$.\end{theorem}

\begin{proof} Assume first that $S$ is finite. We know that $\K[{\mathcal I}^S]\subseteq {\mbox{Prim}(\K[ {\mbox{\bf PBT}^{S}}])}$.
Proposition \ref{prop:structtheo} states that, as a vector space, $\K[ {\mbox{\bf PBT}^{S}}]$ is isomorphic to the tensor vector space $T(\K[{\mathcal I}^S])$. On the other hand, Theorem \ref{teo:estmagmatic} implies that $\K[ {\mbox{\bf PBT}^{S}}]$ is isomorphic to the cotensor coalgebra over the subspace ${\mbox{Prim}(\K[ {\mbox{\bf PBT}^{S}}])}$.

So, as the subspaces of homogeneous elements $\K[{\mathcal I}_n^S]$ and ${\mbox{Prim}(\K[ {\mbox{\bf PBT}^{S}}])_n}$ are both finite dimensional and Proposition \ref{prop:structtheo} implies that they have the same dimension, for $n\geq 1$.  So, we conclude that they are equal.
\medskip

Any set $S$ is the colimit of all its finite subsets, and our constructions commute with colimits, which implies that the vector space ${\mbox{Prim}(\K[ {\mbox{\bf PBT}^{S}}])}$ and 
$\K[{\mathcal I}^S]$ are equal for any set $S$.
\medskip

For the last result, it suffices to note that, if ${\underline t}_1,\dots ,{\underline t}_r$ are primitive elements of $\K[ {\mbox{\bf PBT}^{S}}]$, then 
\begin{enumerate}[(i)]\item 
\begin{align*}  {\overset{\circledast}{\Delta}}(({\mathbf 0}_{2}, (s_{\bullet})){\underline{\circ}} ({\underline t}_1,{\underline t}_2)) &=  {\overset{\circledast}{\Delta}}({\underline t}_1\veebar_{s_{\bullet}}{\underline t}_2) =\\
& ({\underline t}_1\veebar_{s_{\bullet}} {\underline t}_2)\otimes 1_{\K} + {\underline t}_1\otimes {\underline t}_2 + 1_{\K}\otimes ({\underline t}_1\veebar_{s_{\bullet}} {\underline t}_2),\end{align*}

\item \begin{align*} &{\overset{\circledast}{\Delta}}(({\mathbf 0}_{r}, (s_{\bullet},\dots ,s_{\bullet})){\underline{\circ}} ({\underline t}_1,\dots ,{\underline t}_r)) =\\
& {\overset{\circledast}{\Delta}}(({\mathbf 0}_{r-1}, (s_{\bullet},\dots ,s_{\bullet})){\underline{\circ}} ({\underline t}_1,\dots ,{\underline t}_{r-1}))\ \veebar_{s_{\bullet}}\ {\underline t}_r=\\
&(({\mathbf 0}_{r}, (s_{\bullet},\dots ,s_{\bullet})){\underline{\circ}} ({\underline t}_1,\dots ,{\underline t}_r)) \otimes 1_{\K} +  {\overset{\circledast}{\Delta}}( (({\mathbf 0}_{r-1}, (s_{\bullet},\dots ,s_{\bullet})){\underline{\circ}} ({\underline t}_1,\dots ,{\underline t}_{r-1})))\ \veebar_{s_{\bullet}}\ {\underline t}_r,\end{align*}
for any $s_{\bullet}\in S$, where $(x\otimes y) \veebar_{s_{\bullet}} z := x\otimes (y\veebar_{s_{\bullet}} z)$ for any elements $x, y, z\in \K[ {\mbox{\bf PBT}^{S}}]$.\end{enumerate}

Therefore, by recursion on $r$, we get
\begin{align*}  {\overset{\circledast}{\Delta}}(({\mathbf 0}_{r}, (s_{\bullet},\dots ,s_{\bullet})){\underline{\circ}} &({\underline t}_1,\dots ,{\underline t}_{r})) =\\
&\sum_{i=0}^r
(({\mathbf 0}_{i}, (s_{\bullet},\dots ,s_{\bullet}))\circ ({\underline t}_1,\dots ,{\underline t}_{i}))\ \otimes\ (({\mathbf 0}_{r-i}, (s_{\bullet},\dots ,s_{\bullet})){\underline{\circ}} ({\underline t}_{i+1},\dots ,{\underline t}_{r})),\end{align*}
which implies that, as a coalgebra, $\K[ {\mbox{\bf PBT}^{S}}]$ is isomorphic to the cotensor coalgebra over the space $\K[{\mathcal I}^S]$.\end{proof}

\begin{corollary} \label{cor:struc} Given a vector space $V$, the subspace of primitive elements of the underlying coalgebra of ${\mbox{{\it S}-Mag}_{\K}(V)}$ is the vector space $\bigoplus_{n\geq 1} \K[{\mathcal I}_n^S]\otimes V^{\otimes n}$.\end{corollary}

\begin{proof} As any element $v\in V$ is primitive, we get that the subspace $\bigoplus_{n\geq 1} \K[{\mathcal I}_n^S]\otimes V^{\otimes n}$ is contained in ${\mbox{Prim} {\mbox{{\it S}-Mag}_{\K}(V)}}$. But, if $V$ is finite dimensional, applying Proposition \ref{prop:structtheo} we get that the subspaces of homogeneous elements of degree $n$ of both subspaces have the same dimension, for $n\geq 1$, which proves the result in this case.

As any vector space is the colimit of its subspaces of finite dimension and our constructions commute with colimits, the result is true for any vector space $V$.\end{proof}

\begin{remark} \label{rem:relprim} The $S$-magmatic algebra $\K[ {\mbox{\bf PBT}^{S}}]$ acts on any $S$-magmatic algebra $(A, \{\cdot_s\}_{s\in S})$ by\begin{enumerate}[(a)]
\item $\scalebox{0.05}{	\includegraphics[width=0.4\textwidth]{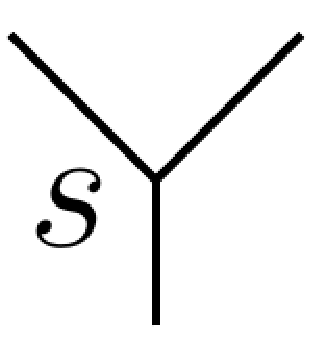}
} \cdot (x_1, x_2) = (\scalebox{0.05}{	\includegraphics[width=0.4\textwidth]{dib4.eps}
}, s)\cdot (x_1, x_2):= x_1\cdot_{s} x_2$,
\item $({\underline t}^l\veebar_s {\underline t}^r)\cdot (x_1,\dots ,x_n) := ({\underline t}^l \cdot (x_1,\dots ,x_r))\cdot_s ({\underline t}^r \cdot (x_{r+1},\dots ,x_n))$, for ${\underline t}^l = (t^l, (s_1, \dots ,s_{r-1}))\in \K[{\mathcal I}_r^S]$, ${\underline t}^r = (t^r, (s_{r+1}, \dots ,s_{n-1})\in \K[{\mathcal I}_{n-r}^S]$,
\end{enumerate}
for any elements $s\in S$ and $x_1,\dots ,x_n\in A$.

For any $s_{\bullet}\in S$ fixed, the elements $\scalebox{0.05}{\includegraphics[width=0.5\textwidth]{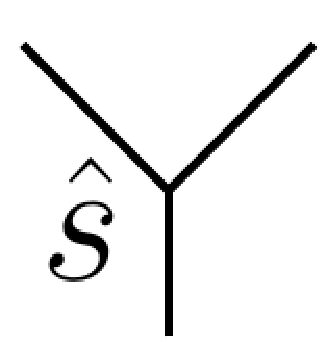}
}$ and $(M_t, (s_1, \dots ,s_{n-1}))$ define $\K$-linear endomorphisms of $A$, for any $t\in  {\mbox{\bf PBT}_n}$ and any $s, s_1,\dots , s_{n-1}\in S$.
\medskip

If $(H, \{\cdot_s\}_{s\in S}, \Delta )$ is a unital $S$-magmatic infinitesimal bialgebra and $s_{\bullet}\in S$ is a fixed element, the elements \begin{enumerate}[(i)]
\item $\scalebox{0.05}{\includegraphics[width=0.5\textwidth]{dib10.eps}
}$, for $s\in S\setminus \{s_{\bullet}\}$,
\item $(M_{t}, (s_1,\dots , s_{n-1}))$, for $t\in  {\mbox{$\rightthreetimes$-Irr}_n}$ and $s_1,\dots ,s_{n-1}\in S$, with $n\geq 2$\end{enumerate} 
define operations on the subspace ${\mbox{Prim}(H)}$ of primitive elements of $H$.
\end{remark}

\begin{notation} \label{not:ops} Let $A$ be an $S$-magmatic algebra, and let $s_{\bullet}\in S$ be a fixed element. For any elements $s\neq s_{\bullet}\in S$ and a colored tree 
${\underline t} = (t, (s_1, \dots ,s_{n-1}))\in {\mbox{\bf PBT}^{S}}$. we denote by\begin{enumerate}[(i)]
\item $M(s):= \cdot_s - \cdot _{s_{\bullet}}$ the binary product defined by the action of the tree $\scalebox{0.05}{\includegraphics[width=0.5\textwidth]{dib10.eps}
}  $. 
\item $M_t(s_1,\dots ,s_{n-1})$ the $n$-ary product given by the action of the element $(M_t, (s_1,\dots ,s_{n-1}))$ on $A$, where $t$ is $\rightthreetimes$-irreducible with $n$ leaves.\end{enumerate}\end{notation}
\medskip

When we look at the elements in ${\mbox{\bf PBT}_n^{S}}$ as $n$-ary operations in any $S$-magmatic infinitesimal algebra, some of them may be obtained from others by composition. Note for instance that the composition of these operations correspond to the products $\circ_i$ and $\underline{\circ}$ described in Definition \ref{def:wedge}.
\medskip 

\begin{remark} \label{rem:formaelem} Let $S$ be a set and $s_{\bullet}\in S$ a fixed element. For any collection $(s_1,\dots ,s_{n-1})\in S^{\times (n-1)}$ such that $s_i\neq s_{\bullet}$, for $1\leq i\leq n-1$, and for any tree $t\in {\mbox{\bf PBT}_n}$ the $n$-ary product $M_t(\hat{s}_1,\dots ,\hat{s}_{n-1})$ is a linear combination of compositions of the binary products $M({s}_1),\dots , M({s}_{n-1})$. Furthermore, we have that
\begin{equation*}M_t(s_1,\dots ,s_r) = M_t(\hat{s}_1,\dots ,\hat{s}_{n-1}) +\sum _{\emptyset \subsetneq J\subseteq [n-1]} M_t(w_1,\dots ,w_{n-1}),\end{equation*}
where $w_i := \begin{cases} s_i,&{\rm for}\ i\notin J\\ s_{\bullet},&{\rm for}\ i\in J.\end{cases}$\end{remark}

\begin{lemma} \label{lem:generatingprods} Let $S$ be a set and $s_{\bullet}\in S$ be a fixed element of $S$, and let $(A , \{\cdot_s\}_{s\in S} )$ be a unital $S$-magmatic infinitesimal algebra.  Let $t = {\mathbf 1}_k \underline{\circ} (t_1,\dots , t_{k-1},\vert)$ be a tree in 
$ {\mbox{\bf PBT}_n}$, with $\vert t_i\vert = m_i$ for $1\leq i < k$. The products $M(s)$ and $M_t(s_1,\dots ,s_{n-1})$ satisfy the following relations\begin{enumerate}[(a)]
\item For any family of elements $s_1,\dots ,s_{n-1}\in S$ and any $1\leq i<k$ , 
\begin{align*} &M_t(s_1,\dots ,s_{n-1}) =\\
 &M_{{\mathbf 1}_k\underline{\circ}(t_1,\dots ,t_{i-1}, \vert , t_{i+1},\dots, t_{k-1},\vert)}(s_1,\dots ,s_{l},s_{l+m_i +1}, \dots ,s_{n -1})\circ ({\mbox{Id}^{\otimes l}}\otimes M_{t_i}(s_{l+1},\dots ,s_{l+m_i-1}) \otimes {\mbox{Id}^{\otimes(n-l-m_i)}}),\end{align*}
 where $l := m_1+\dots +m_{i-1}$. Moreover, we have that
\begin{equation*} M_t(s_1,\dots ,s_{n-1}) =( M_{{\mathbf 1}_k} (s_{m_1},s_{m_1+m_2},\dots ,s_{m_1+\dots + m_{k-2}},\dots , s_{n-1})\circ N_1\circ \dots \circ N_{k-1},\end{equation*}
where $N_i := {\mbox{Id}^{\otimes l_i}}\otimes M_{t_i}(s_{l_i+1},\dots ,s_{l_{i+1}-1})\otimes {\mbox{Id}^{\otimes (n-l_{i+1})}}$, for $l_i := m_1 + \dots + m_{i-1}$ and $1\leq i < k$. 

\item Let $w = w_1\rightthreetimes w_2$, with $w_1\in {\mbox{$\rightthreetimes$-Irr}}$ and $\vert w_2\vert \geq 1$. For any $1\leq i < k$, we have that  
\begin{align*} &M_{{\mathbf 1}_k\underline{\circ} (t_1,\dots ,t_{i-1}, w, t_{i+1},\dots ,t_{k-1},\vert)}(s_1,\dots ,s_{n-1}) =\\
& M_{{\mathbf 1}_k\underline{\circ} (t_1,\dots ,t_{i-1}, \vert\rightthreetimes w_2, t_{i+1},\dots ,t_{k-1},\vert)}(s_1,\dots ,s_{l+1}, s_{l+\vert w_1\vert-1},\dots  ,s_{n-1})
\circ_i M_{w_1}(s_{l+1},\dots ,s_{l+\vert w_1\vert -1}),\end{align*}
where $l := m_1+ \dots + m_{i-1}$, for any family of elements $s_1,\dots ,s_{n-1}\in S$.\end{enumerate}\end{lemma}

\begin{proof} Point $(a)$ is a straightforward consequence of Proposition \ref{prop:Mtforanyt}. For $(b)$, Lemma \ref{lem:Moebirred} implies that:
\begin{equation*} M_{w_1\rightthreetimes w_2} = M_{w_1}\rightthreetimes M_{w_2} =  M_{w_2}\circ_1 (M_{w_1}\veebar \vert ) = (\vert \rightthreetimes M_{w_2})\circ _1 M_{w_1}.\end{equation*}

The result follows from the previous paragraph and point $(a)$.
\end{proof}

Consider the set  $\{M_t(s_1,\dots, s_{n-1})\mid t\in {\mbox{$\rightthreetimes$-Irr}_n}\ {\rm and}\ s_1,\dots, s_{n-1}\in S\}\cup \{M(s): s\neq s_{\bullet}\in S  \}$ of operations defined on any $S$-magmatic algebra. Using Lemma \ref{rem:formaelem} we eliminate all the operations which are compositions of other products of the same set. Nevertheless, we are not able to state that this set is minimal because we do not know if the relations described are exhaustive. 
\medskip

 Let $S$ be a set and $s_{\bullet}\in S$ be a fixed element of $S$, and let $(H , \{\cdot_s\}_{s\in S}, \Delta )$ be a unital $S$-magmatic infinitesimal bialgebra. Let $\K[\hat{\mathbf I}]$ be subspace of $\oplus_{n\geq 1}{\mbox{Hom}_{\K}(H^{\otimes n}, H)}$ generated by the compositions of the following operations\begin{enumerate}[(a)]
\item the binary products $M(s)$, with $s\in S$ and $s\neq s_{\bullet}$,
\item the $n$-ary products $M_{{\mathbf 1}_n}(s_1,\dots ,s_{n-1})$, $n\geq 2$, for any family of elements $s_1,\dots, s_n$ satisfying that there exists at least one integer $1\leq i < n$ such that $s_i = s_{\bullet}$,
\item for $n\geq 3$, the $n$-ary products $M_t(s_1,\dots ,s_{n-1})$, where $t= {\mathbf 1}_k\underline{\circ} (\vert, t_2,\dots ,t_{k-1}, \vert)$ satisfying that\begin{enumerate}[(i)]
\item  either $t_j = \vert$, or $t_j= \vert \rightthreetimes t_j\rq\in {\mbox{\bf PBT}_{m_j}}$ with $\vert t_j\rq \vert = m_j-1\geq 1$, for $2\leq j\leq k-1$, 
\item $\vert t_j\rq\vert >1$ for at least one integer $2\leq j < k$,
\item the colors of the nodes are given by
\begin{equation*}(u_1, s_{\bullet},r_2^2,\dots ,r_{m_2-2}^2,u_2,s_{\bullet}, r_2^3,\dots , r_{m_2-2}^3, u_2,\dots ,u_{k-2}, s_{\bullet}, r_2^{k-1},\dots ,r_{m_{k-1}-2}^{k-1}, u_{k-1}),\end{equation*}
where the nodes of ${\mathbf 1}_k$ are colored from left to right by the elements $u_1,\dots ,u_k$ in $S$ and the colored tree $(t_j, (s_{\bullet}, r_2^j,\dots ,r_{m_j-1}^j) )= (t_j\rq, (r_2^j,\dots ,r_{m_j-1}^j))\circ_1 ( \scalebox{0.5}{\includegraphics[width=0.04\textwidth]{dib4.eps}}, (s_{\bullet}))$.
 \end{enumerate}
\end{enumerate}
\medskip

\begin{proposition} \label{prop:genofprim} The subspace $\K[\hat{\mathbf I}]$ of $\bigoplus_{n\geq 1}{\mbox{Hom}_{\K}(H^{\otimes n}, H)}$ coincides with the subspace linearly generated by the set of operations $\{M_t(s_1,\dots, s_{n-1})\mid t\in {\mbox{$\rightthreetimes$-Irr}_n}\ {\rm and}\ s_1,\dots, s_{n-1}\in S\}\cup \{M(s): s\neq s_{\bullet}\in S  \}$.\end{proposition}

\begin{proof} We need to show that any operator of the form $M_{t}(s_1,\dots ,s_{n-1})$ is a composition of operations in $\K[\hat{\mathbf I}]$, for any tree $t\in {\mbox{$\rightthreetimes$-Irr}_n}$ and any collection of elements $s_1,\dots ,s_{n-1}$ in $S$. We proceed by recursion on $\vert t\vert$.

The result is immediate for $\vert t\vert =2$. For $\vert t\vert  > 2$, we have that $t = {\mathbf 1}_k\circ (t_1,\dots ,t_{k-1}, \vert)$, with $t_1$ irreducible. 
\bigskip 

If $\vert t\vert = k$, then $t ={\mathbf 1}_k$. By Remark \ref{rem:formaelem}, we get that $M_t(s_1,\dots ,s_{k-1})$ is a linear combination of a composition of the products $M(s_1),\dots ,M(s_{k-1})$ and operators $M_{{\mathbf 1}_k}(u_1,\dots ,u_{k-1})$ satisfying that at least for one integer $j\in \{ 1,\dots ,k-1\}$ the elements $s_j = s_{\bullet}$, and $M_{{\mathbf 1}_k}(u_1,\dots ,u_{k-1})$ belongs to $\K[\hat{\mathbf I}]$.
\medskip

If $\vert t\vert > k$, then we have to consider different cases\begin{enumerate}[(a)]
\item If there exists $i$ such that $t_i$ is $\rightthreetimes$-irreducible and $\vert t_i\vert > 1$, then by Lemma \ref{lem:generatingprods} we have that 
$ M_t(s_1,\dots ,s_{n-1})$ equals to:\\
 $M_{t\rq} (s_1,\dots ,s_l, s_{l+m_i},\dots , s_{n-1})\circ ({\mbox{Id}^{\otimes l}}\otimes M_{t_i}(s_{l+1},\dots ,s_{l+m_i-1})\otimes {\mbox{Id}^{\otimes n-l-m_i}}),$\\
where $t\rq =  {\mathbf 1}_k\circ (t_1,\dots , t_{i-1},\vert ,t_{i+1},\dots  ,t_{k-1}, \vert)$, $m_j $ is the degree of the tree $t_j$ for $1\leq j < k$, and $l = m_1+\dots +m_{i-1}$.

In this case, $\vert t_i\vert < \vert t\vert$ and $\vert t\rq \vert < \vert t\vert$, so the result is true applying the recursive hypothesis.

\item If there is no tree $i$ such that $t_i$ is $t_i\neq \vert$ and  $t_i$ is $\rightthreetimes$-irreducible, then there exists at least one $t_i = z_1\rightthreetimes z_2$ with $z_1\in {\mbox{$\rightthreetimes$-Irr}_r}$, for some $r\geq 1$, because $\vert t\vert > k$. Moreover, as $t$ is $\rightthreetimes$-irreducible, we get that $t_1 = \vert$.

We have that $t = t\rq \circ_{i}t_{i}$, where $t\rq = {\mathbf 1}_k\circ (t_1,\dots ,t_{i-1},\vert ,t_{i+1},\dots ,t_{k-1},\vert )$.  Note that $\vert t\rq \vert < \vert t\vert$, so by a recursive argument, we know that $M_{t_{rq}}(s_1,\dots ,s_{i_1-1},s_{i_1+l-1},\dots ,s_{n-1})\in \K[\hat{\mathbf I}]$.
\medskip

 If $z_1 = \vert$, then $t_i = z_2\circ_1  \includegraphics[width=0.02\textwidth]{dib4.eps}$. Applying Lemma \ref{lem:Moebirred}, we have that 
\begin{equation*}M_{t_i} = \vert \rightthreetimes M_{z_2} = M_{z_2}\circ_1 \includegraphics[width=0.02\textwidth]{dib4.eps}.\end{equation*} 
Therefore
\begin{equation*} M_{t_i}(s_{l+1},\dots ,s_{l+m_i-1}) = M_{z_2}(s_{l+2}, \dots ,s_{l+m_i-1})\circ ((M(s_l)+ \includegraphics[width=0.02\textwidth]{dib4.eps}(s_{\bullet}))\otimes 
 {\mbox{Id}^{\otimes (m_i-1)}}.\end{equation*}
So, we get that $M_t(s_1,\dots ,s_{n-1}) =$
\begin{align*} &M_{t\rq}(s_1,\dots ,s_{l}, s_{l+m_i},\dots ,s_{n-1}) \circ ({\mbox{Id}^{\otimes (l)}}\otimes (\vert \rightthreetimes M_{z_2}(s_{l+1},\dots ,s_{l+m_i-1}))\otimes {\mbox{Id}^{\otimes (n-l-m_i)}})=\\
&M_{t\rq \circ_{l+1} z_2}(s_1,\dots ,s_{l},s_{l+2},\dots ,s_{n-1})\circ ({\mbox{Id}^{\otimes l}}\otimes M(s_{l+1})\otimes {\mbox{Id}^{\otimes (n-l-2)}}) +\\
&M_{t\rq}(s_1,\dots ,s_{l},s_{l+m_i},\dots ,s_{n-1})\circ ({\mbox{Id}^{\otimes l}}\otimes (M_{z_2}\circ _1\includegraphics[width=0.02\textwidth]{dib4.eps}(s_{\bullet}, s_{l+2},\dots ,s_{l+m_i-1}))\otimes {\mbox{Id}^{\otimes (n-l-m_i)}}),\end{align*}
where $m_i =\vert t_i\vert$ and $l = m_1+\dots +m_{i-1}$.

As $\vert t\rq\circ _{l+1} z_2\vert = \vert t\vert -1$, applying a recursive argument, we have that the operator 
\begin{equation*} M_{t\rq \circ_{l+1} z_2}(s_1,\dots ,s_{l},s_{l+2},\dots ,s_{n-1}) \in \K[\hat{\mathbf I}],\end{equation*}
 which implies that the composition 
\begin{equation*} M_{t\rq \circ_{l+1} z_2}(s_1,\dots ,s_{l},s_{l+2},\dots ,s_{n-1})\circ ({\mbox{Id}^{\otimes (l+1)}}\otimes M(s_{l+1})\otimes {\mbox{Id}^{\otimes (n-l-2)}}) \end{equation*} is an element of $\K[\hat{\mathbf I}]$, too.

On the other hand, the operation 
\begin{equation*} M_{t\rq}(s_1,\dots ,s_{l},s_{l+m_i},\dots ,s_{n-1})\circ ({\mbox{Id}^{\otimes l}}\otimes (M_{z_2}\circ _1\includegraphics[width=0.02\textwidth]{dib4.eps}(s_{\bullet}, s_{l+1},\dots ,s_{l+m_i-1}))\otimes {\mbox{Id}^{\otimes (n-l-m_i)}})\end{equation*}
 also belongs to $\K[\hat{\mathbf I}]$, by the definition of $\K[\hat{\mathbf I}]$.
 \medskip

If $\vert z_1\vert > 1$, then Lemma \ref{lem:Moebirred} states that $M_{t_i} = M_{z_1}\rightthreetimes M_{z_2} = M_{\vert \rightthreetimes z_2}\circ_1 M_{z_1}$.

By Lemma \ref{lem:generatingprods}, we get that $M_t(s_1,\dots ,s_{n-1}) =$
\begin{equation*} 
 M_{t\rq\circ_{l+1}(\vert \rightthreetimes z_2)}(s_1,\dots ,s_{l}, s_{l+h_1 -1},\dots ,s_{n-1})\circ ({\mbox{Id}^{\otimes l}}\otimes 
M_{z_1}(s_{l+1},\dots ,s_{l+\vert z_1\vert -1})\otimes {\mbox{Id}^{\otimes (n-l-h_1)}}),\end{equation*}
where $h_1=\vert z_1\vert$.
But we prove in the previous case that the element 
\begin{equation*} M_{t\rq\circ_{l+1}(\vert \rightthreetimes z_2)}(s_1,\dots ,s_{l}, s_{l+h_1 -1},\dots ,s_{n-1})\end{equation*} belongs to $\K[\hat{\mathbf I}]$, and $M_{z_1}(s_{l+1},\dots ,s_{l+h_1 -1})\in \K[\hat{\mathbf I}]$ because $z_1\in {\mbox{$\rightthreetimes$-Irr}}$ and $h_1 = \vert z_1\vert < \vert t\vert$.
\medskip

Therefore, the element 
\begin{equation*}M_{t\rq}(s_1,\dots ,s_{l},s_{l+m_i},\dots ,s_{n-1})\circ ({\mbox{Id}^{\otimes l}}\otimes (M_{z_2}\circ _1\includegraphics[width=0.02\textwidth]{dib4.eps}(s_{\bullet}, s_{l+2},\dots ,s_{l+m_i-1}))\otimes {\mbox{Id}^{\otimes (n-l-m_i)}})\end{equation*}
also belongs to $\K[\hat{\mathbf I}]$, which ends the proof.

 \end{enumerate}

\end{proof}
The following is an explicit example of how to apply the previous results.
\begin{example}
	
Let	
\begin{equation*}\scalebox{0.4}{	\includegraphics[width=0.6\textwidth]{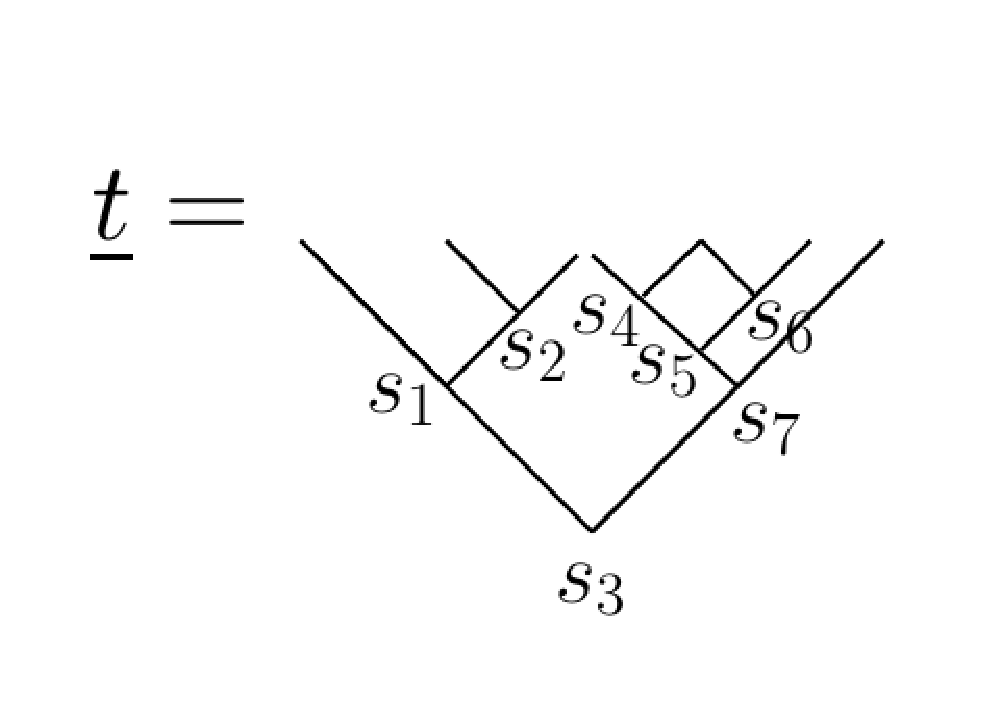}}\end{equation*}

\begin{center}
\scalebox{1}{	\includegraphics[width=0.7\textwidth]{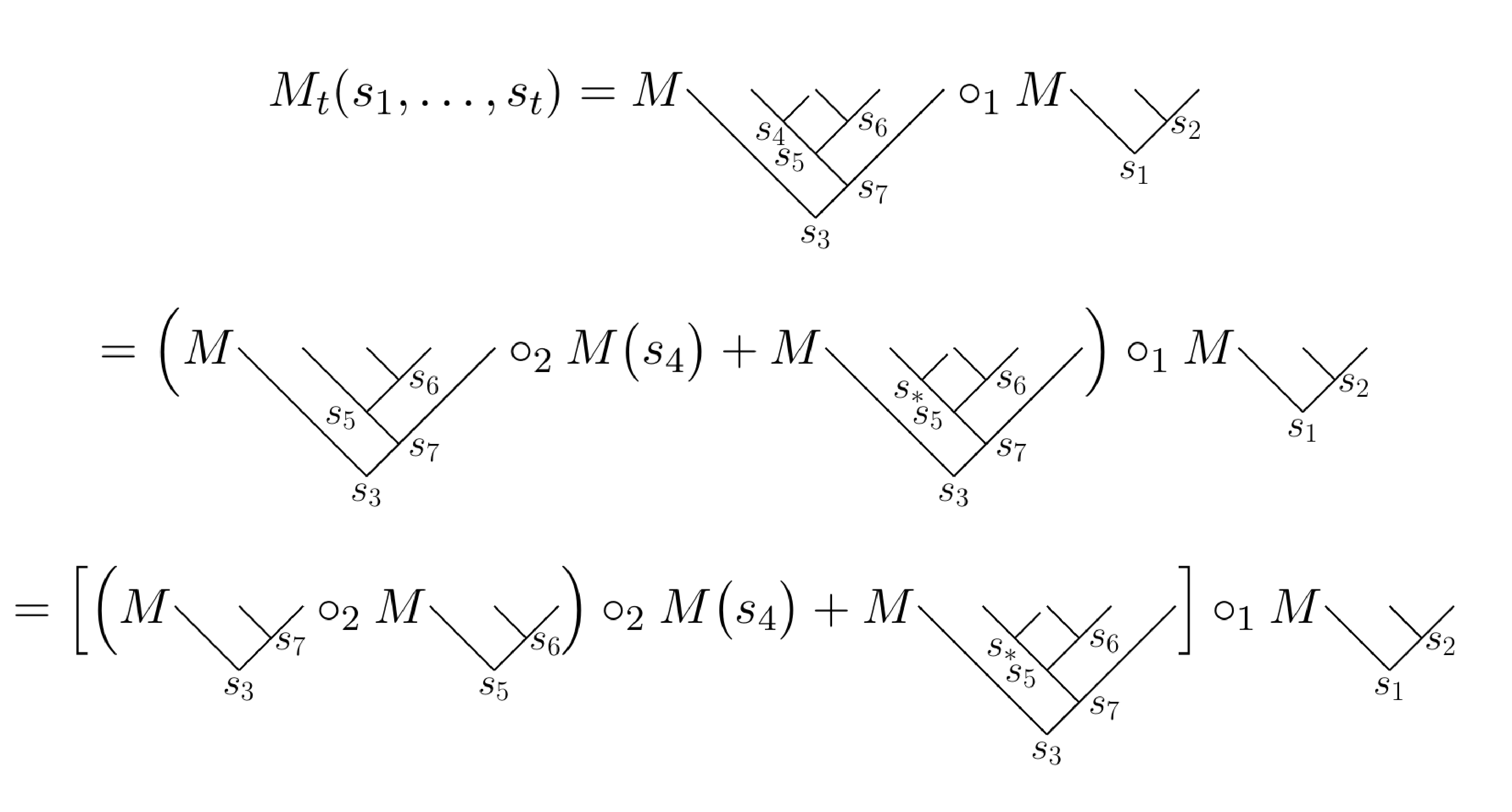}}
\end{center}	
\begin{center}
\scalebox{1}{	\includegraphics[width=0.6\textwidth]{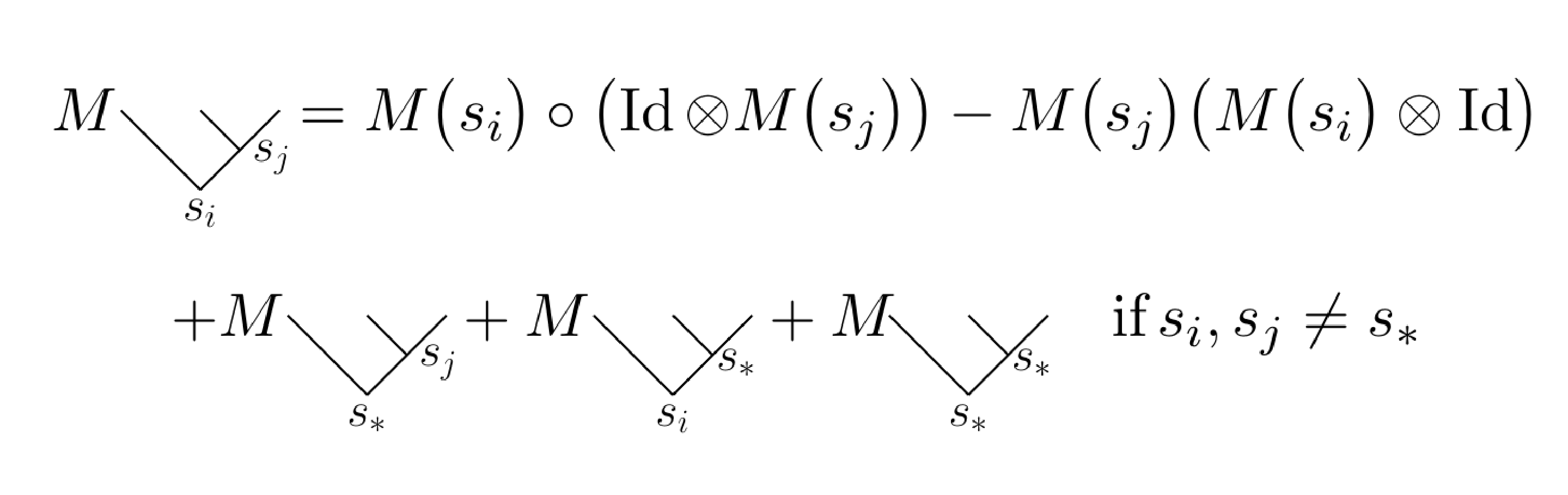}}
\end{center}	
then  we can write $M_{t}( s_1,\cdots, s_{n-1})$ as a linear combination of $\K[\hat{\mathbf I}]$.
\end{example}

\bigskip

\bigskip

\section{An example: the coalgebra of integer relations} 
\medskip

For any nonnegative integer $n$, let $[n]:=\{1,\dots ,n\}$ be the set of the first $n$ natural numbers.

\begin{definition} \label{def:intrel} For $n\geq 1$, an {\it integer relation of size $n$} is a reflexive relation on the set $[n]$, that is a subset $R$ of the cartesian product $[n]\times [n]$. 

The set of integer relations of size $n$ is denoted ${\mathcal R}_n$. 

 Given a permutation $\sigma\in \Sigma_n$ and a relation $R\in {\mathcal R}_n$, let $\sigma\cdot R$ be the unique relation in ${\mathcal R}_n$ satisfying that $(i,j)\in R$ if, and only if, $(\sigma (i), \sigma (j))$ belongs to $\sigma\cdot R$. 
 
 For any $S\subseteq [n]$ and any relation $R\in  {\mathcal R}_n$, the {\it restriction} of $R$ to $S$ is the relation $R\vert_S\in {\mathcal R}_{\vert S\vert}$ defined by the condition
 \begin{equation*} (i,j)\in R\vert_S\ {\rm if,\ and\ only\ if}\ (s_i,s_j)\in R,\end{equation*}
 where $S = \{ s_1<\dots < s_{\vert S\vert}\}$.\end{definition}

In our pictures, we always represent an integer relation $R$ as follows:  the numbers $1,\cdots, n$ are written from left to right, and the increasing relations $(i, j)\in R$ are represented by an arc above the nodes, while an arc below the nodes represents a decreasing relation $(j, i)\in R$ below, for $i< j$. Although all the relations are reflexive, we omit the relations $(i, i)$ in the pictures.
\medskip
 
 In the present section we define a collection $S_{\mathcal R}$ of binary magmatic products on the dual coalgebra of Pilaud and Pons Hopf algebra ${\mathbb H}_{{\mathcal R}_{PP}}^*$, which generates the underlying vector space $\K[{\mathcal R}]$ as an $S_{\mathcal R}$-magmatic algebra over the unique element $\#\in {\mathcal R}_1$. However, this $S_{\mathcal R}$-magmatic algebra is not free. 
 
 As a consequence of the previous section, we get that ${\mathbb H}_{{\mathcal R}_{PP}}^*$ is isomorphic, as a coalgebra, to the cotensor algebra over its subspace ${\mbox{Prim}({\mathbb H}_{{\mathcal R}_{PP}}^*)}$ of primitive elements, and that  ${\mbox{Prim}({\mathbb H}_{{\mathcal R}_{PP}}^*)}$ is generated by the unique element $\#$ of ${\mathcal R}_1$ under the action of the products defined in the previous section.

\subsection{The Hopf algebra ${\mathbb H}_{{\mathcal R}_{PP}^*}$}
\medskip

The coproduct $\Delta_{\mathcal R}$ on ${\mathbb H}_{{\mathcal R}_{PP}^*}$ is defined by the equation
\begin{equation*} \Delta_{\mathcal R} (R) = \sum_{i=0}^n R\vert_{[i]}\otimes R\vert_{\{i+1,\dots ,n\}},\end{equation*}
for any $R\in R\in {\mathcal R}_n$.
\medskip

It is immediate to verify that $\Delta_{\mathcal R}$ is coassociative and the dual of the product defined by V. Pilaud and V. Pons in \cite{PiPo1}. 
\medskip

To define the corresponding product, we will introduce some binary products on the vector space $\K[{\mathcal R}]$.

\begin{definition} \label{def:asprodinrel} For any pair of reflexive relations $R\in {\mathcal R}_n$ and $Q\in {\mathcal R}_m$, define the binary associative products $\sqcup$, $\uparrow$, $\downarrow$ and $\updownarrow$ as follows\begin{enumerate}
\item $R\sqcup Q$ is the reflexive relation on $[n+m]$ given by
\begin{equation*} R\sqcup Q :=\begin{cases} (i,j)\in R, &\ {\rm for}\ 1\leq i,j\leq n,\\
(i-n, j-n)\in Q, , &\ {\rm for}\ n+1\leq i,j\leq n+m.\end{cases}\end{equation*}
\item $R\uparrow Q$ is the reflexive relation on $[n+m]$ given by
\begin{equation*} R\uparrow Q :=\begin{cases} (i,j)\in R, &\ {\rm for}\ 1\leq i,j\leq n,\\
(i-n, j-n)\in Q, &\ {\rm for}\ n+1\leq i,j\leq n+m,\\
(i,j), &\ {\rm for}\ 1\leq i\leq n\ {\rm and}\ n+1\leq j\leq n+m.\end{cases}\end{equation*}
\item $R\downarrow Q$ is the reflexive relation on $[n+m]$ given by
\begin{equation*} R\downarrow Q :=\begin{cases} (i,j)\in R, &\ {\rm for}\ 1\leq i,j\leq n,\\
(i-n, j-n)\in Q, &\ {\rm for}\ n+1\leq i,j\leq n+m,\\
(i,j), &\ {\rm for}\ n+1\leq i\leq n+m\ {\rm and}\ 1\leq j\leq n.\end{cases}\end{equation*}
\item $R\updownarrow Q$ is the reflexive relation on $[n+m]$ given by
\begin{equation*} R\updownarrow Q :=\begin{cases} (i,j)\in R, &\ {\rm for}\ 1\leq i,j\leq n,\\
(i-n, j-n)\in Q, &\ {\rm for}\ n+1\leq i,j\leq n+m,\\
(i,j), &\ {\rm when\ either}\ 1\leq i\leq n\ {\rm and}\ n+1\leq j\leq n+m,\ {\rm or}\ \\ 
& n+1\leq i\leq n+m\ {\rm and}\ 1\leq j\leq n.\end{cases}\end{equation*}\end{enumerate}\end{definition}

Note that the products $\sqcup$, $\uparrow$, $\downarrow$ and $\updownarrow$ are associative, but they do not verify 
\begin{equation*} x\star (y*z) = (x\star y)*z,\end{equation*}
for $x,y,z\in \K[{\mathcal R}]$, $\star\neq *$ and $\star, *\in \{ \sqcup, \uparrow, \downarrow, \updownarrow\}$.
\medskip

\begin{example}\label{ex:downarrow}
For $P= Q = \# \sqcup \#$, we get that 
\begin{equation*}\scalebox{0.3}{\includegraphics[width=0.6\textwidth]{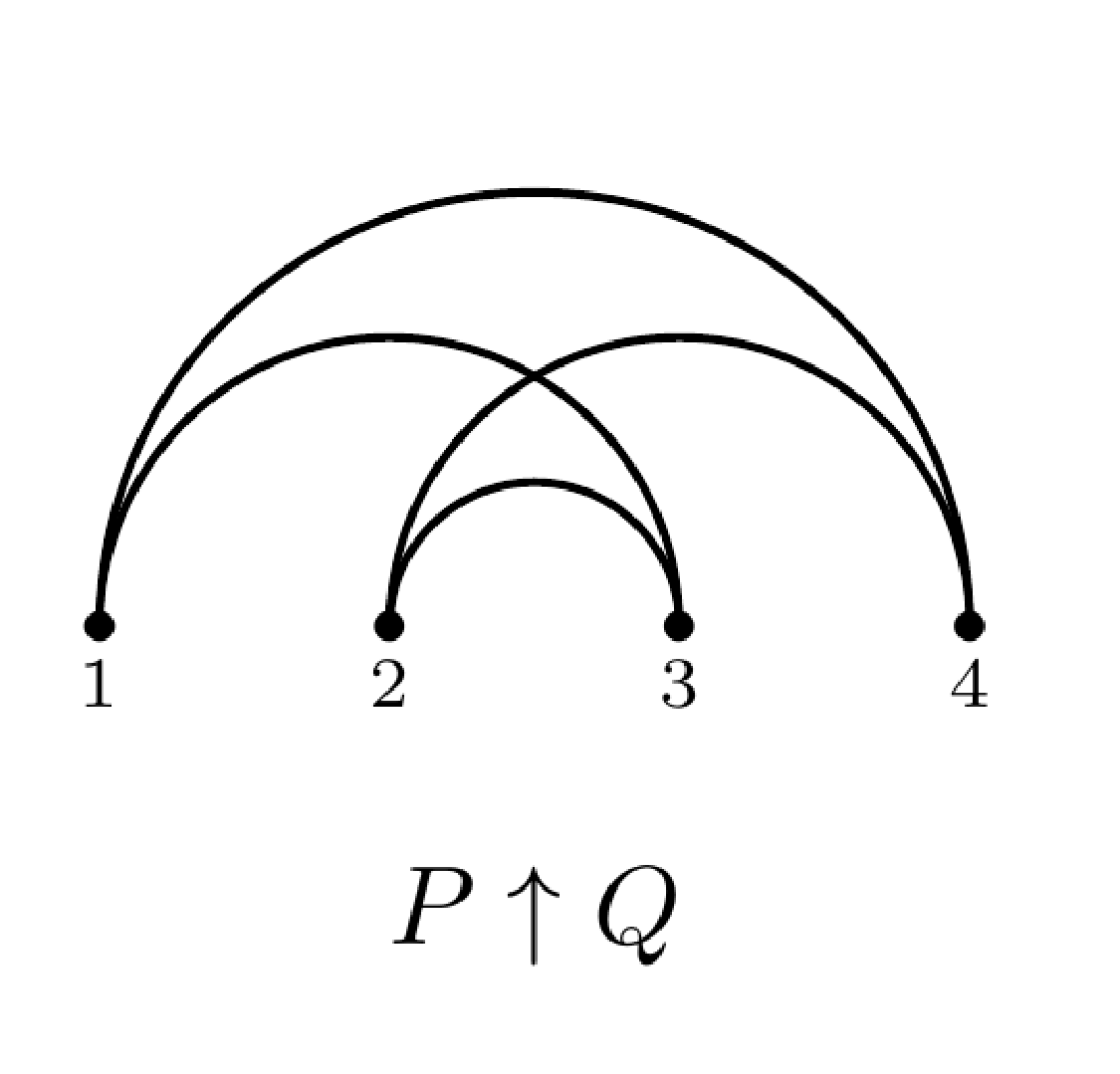}}\end{equation*}	
	
\end{example}

\begin{lemma} \label{lem:infoprel} The coproduct $\Delta_{\mathcal R}$ satisfies the unital infinitesimal relation with the products $\sqcup$, $\uparrow$, $\downarrow$ and $\updownarrow$.\end{lemma}
\begin{proof} Note that for any product $\star \in \{ \sqcup, \uparrow, \downarrow, \updownarrow\}$ and a pair of relations $R\in{ \mathcal R}_n$ and $Q\in {\mathcal R}_m$ we have that
\begin{equation*}( P\star Q)\vert_{[i]} = \begin{cases} P\vert_{[i]},&{\rm for}\ 0\leq i\leq n,\\
P\star Q\vert_{\{n+1,\dots ,i\}}, &{\rm for}\ n+1\leq i\leq n+m.\end{cases}\end{equation*}
while 
\begin{equation*}( P\star Q)\vert_{\{i+1,\dots ,n+m\}} = \begin{cases} P\vert_{\{i+1,\dots ,n\}}\star Q,&{\rm for}\ 0\leq i\leq n,\\
Q\vert_{\{i+1,\dots ,n+m\}}, &{\rm for}\ n+1\leq i\leq n+m.\end{cases}\end{equation*}

So, \begin{align*} \Delta_{\mathcal R}(P\star Q) = &\sum_{i=0}^{n} P\vert_{[i]}\otimes (P\vert_{\{i+1,\dots ,n\}}\star Q) + \sum_{i=n+1}^{n+m} (P\star Q\vert_{\{n+1,\dots ,i\}})\otimes Q\vert_{\{i+1,\dots ,n+m\}} =\\
&\sum P_{(1)}\otimes (P_{(2)}\star Q) + \sum (P\star Q_{(1)})\otimes Q_{(2)} - P\otimes Q,\end{align*}
because $P\otimes Q$ appears twice in $\sum P_{(1)}\otimes (P_{(2)}\star Q) + \sum (P\cdot Q_{(1)})\otimes Q_{(2)}$, where $\Delta_{\mathcal R}(P) = \sum P_{(1)}\otimes P_{(2)}$ for any $P\in {\mathcal R}$.\end{proof}
\medskip

Let us show that, using the shuffle, we may associate to each product $\star \in  \{ \sqcup, \uparrow, \downarrow, \updownarrow\}$ another associative product $\diamond_{\star}$ satisfying that the coproduct $\Delta_{\mathcal R}$ is an algebra homomorphism for $\diamond_{\star}$, that is $(\K[{\mathcal R}], \diamond_{\star}, \Delta_{\mathcal R})$ is a graded Hopf algebra in the usual sense.
\medskip 

\begin{notation} \label{notn:permparc} Let $\sigma\in {\mbox{sh}(n,m)}$ be an $(n,m)$-shuffle. For any $0\leq i\leq n+m$, let $0\leq p\leq n$ and $0\leq q\leq m$ be the unique integers such that $\sigma(p)\leq i <\sigma(p+1)$ and $\sigma (n+q)\leq i\leq \sigma(n+m)$. We denote by 
$\sigma_{(1)}^i\in \Sigma_ i$ and $\sigma_{(2)}^i\in \Sigma_{n+m-i}$ the permutations $\sigma_{(1)}^i:=\sigma^{\{1,\dots , p, n+1,\dots ,n+q\}}$ and $\sigma_{(2)}^i:=\sigma^{\{p+1,\dots , n, n+q+1,\dots ,n+m\}}$.\end{notation}

\begin{remark} \label{rem:diagshuff} Let $R$ and $Q$ be two relations, with $R\in {\mathcal R}_n$ and $Q\in {\mathcal R}_m$. For any permutation $\sigma $ which is a $(n,m)$-shuffle, any product $\star \in  \{ \sqcup, \uparrow, \downarrow, \updownarrow\}$ 
and any integer $0\leq i\leq n+m$, we have that \begin{enumerate}[(i)]
\item $( \sigma\cdot (R\star Q))\vert_{[i]} = \sigma_{(1)}^i\cdot  (R\vert _{[p]} \star Q\vert_{[q]})$,
\item $(\sigma\cdot  (R\star Q))\vert_{\{i+1,\dots ,n+m\}} = \sigma_{(2)}^i\cdot (R\vert_{\{p+1,\dots ,n\}}\star Q\vert_{\{q+1,\dots ,m\}})$,\end{enumerate}
where $p$ is the maximal integer satisfying that $1\leq p\leq n$ and $\sigma(p)\leq i$, $q$ is the maximal integer satisfying that $1\leq q\leq m$ and $\sigma(n+q)\leq i$, and $\sigma_{(1)}^i$ and $\sigma_{(2)}^i $ are the permutations defined in Notation \ref{notn:permparc}. Moreover, note that $\sigma_{(1)}^i\in {\mbox{sh}(p,q)}$ and $\sigma_{(2)}^i\in {\mbox{sh}(n-p,m-q)}$.\end{remark}
\medskip

\begin{definition} \label{def:proddiamond} For any product $\star\in \{ \sqcup, \uparrow, \downarrow, \updownarrow\}$, define the product $\diamond_{\star}$ as follows
\begin{equation*} R\diamond_{\star} Q:= \sum_{\sigma\in {\mbox{sh}(n,m)}}\sigma\cdot (R\star Q),\end{equation*}
for any pair of relations $R\in{ \mathcal R}_n$ and $Q\in {\mathcal R}_m$.\end{definition}

It is easy to see that the associativity of the products $ \sqcup,\  \uparrow,\  \downarrow$ and $\updownarrow$ imply that $\diamond_{\star}$ is associative for $\star\in \{ \sqcup, \uparrow, \downarrow, \updownarrow\}$. Moreover, we have the following result

\begin{proposition} \label{prop:shuffbialg} For any $\star\in \{ \sqcup, \uparrow, \downarrow, \updownarrow\}$, the data $(\K[{\mathcal R}], \diamond_{\star}, \Delta_{\mathcal R})$ is a Hopf algebra.\end{proposition}
\begin{proof} As $\K[{\mathcal R}]$ is graded, it suffices to show that $\Delta_{\mathcal R}$ is an algebra homomorphism. 

Note first that, for any pair of positive integers $n, m$ there exists a bijective map 
\begin{equation*} \beta: \{(p,q,\sigma_{(1)}, \sigma_{(2)})\mid 0\leq p\leq n,\ 0\leq q\leq m,\ \sigma_{(1)}\in {\mbox{sh}(p,q)},\sigma_{(2)}\in {\mbox{sh}(n-p,m-q)}\}\longrightarrow {\mbox{sh}(n,m)},\end{equation*}
given by $\beta (p,q,\sigma_{(1)}, \sigma_{(2)})(j):=\begin{cases}\sigma_{(1)}(j),&\ {\rm for}\ 1\leq j\leq p\\
\sigma_{(2)}(j-p)+p+q,&\ {\rm for}\ p+1\leq j\leq n\\ 
\sigma_{(1)}(j-n+p),&\ {\rm for}\ n+1\leq j\leq n+q\\
\sigma_{(2)}(j-q-p)+p+q,&\ {\rm for}\ n+q+1\leq j\leq n+m.\end{cases}$

The result follows immediately from the paragraph above and Remark \ref{rem:diagshuff}.\end{proof}

\begin{remark} \label{rem:PilPonHopf} The Hopf algebra of integer relations ${\mathbb H}_{{\mathcal R}_{PP}}$ is the dual of the Hopf algebra $(\K[{\mathcal R}], \diamond_{\uparrow}, \Delta_{\mathcal R})$.\end{remark}
\bigskip

\subsection{Magmatic binary products generating the vector space $\K[{\mathcal R}]$}

Let $\#$ be the unique element of ${\mathcal R}_1$. Even if the products $\sqcup, \uparrow, \downarrow$ and $\updownarrow$, as well as their associated shuffles $\diamond_{\sqcup}, \diamond_{\uparrow}, \diamond_{\downarrow}$ and $\diamond_{\updownarrow}$, may be extended by linearity to the whole vector space $\K[{\mathcal R}]$, it is clear that their action on $\#$ does not span the whole vector space.
 
Let us define a collection of binary magmatic products satisfying that the vector space $\K[{\mathcal R}]$ is generated by them from $\#$.
\medskip

Let us denote by ${\mathbb O}_{\mathcal R}$ the set of binary products $\{ \sqcup, \uparrow, \downarrow, \updownarrow\}$.

\begin{definition}\label{def:oprel} For $r\geq 1$ and a map $\alpha :[r] \longrightarrow {\mathbb O}_{\mathcal R}$ such that $\alpha (r)\neq \sqcup$, define the graded product $*_{\alpha}$ as follows:
\begin{equation*} R *_{\alpha} Q :=\begin{cases} (i,j)\in R,&{\rm for}\ 1\leq i,j\leq n,\\
(i-n,j-n)\in Q,&{\rm for}\ n < i,j\leq n+m,\\
(n,n+i)&{\rm for}\ 1\leq i\leq {\mbox{min}\{r,m\}}\ {\rm and}\ \alpha(i) =\uparrow,\\
(n+i,n),&{\rm for}\ 1\leq i\leq  {\mbox{min}\{r,m\}}\ {\rm and}\ \alpha(i) =\downarrow,\\
(n, n+i)\ {\rm and}\ (n+i, n),&{\rm for}\ 1\leq i\leq  {\mbox{min}\{r,m\}}\ {\rm and}\ \alpha(i) =\updownarrow.\end{cases}\end{equation*}

For a map $\alpha:[r] \longrightarrow {\mathbb O}_{\mathcal R}$, we say that the {\it size} of $\alpha$ is $r$ and denote it by $s(\alpha)$. The size of the disjoint union $\sqcup$ is $0$.\end{definition}

\begin{example} \label{ex:diffprods}
Consider the following posets\\
\begin{center}
\begin{pspicture}(0,-1)(3,2)\psdots[dotsize=3pt](0,0.5)(1,0.5)(2,0.5)(3,0.5)
\rput(0,0.3){\tiny{1}}
\rput(1,0.3){\tiny{2}}
\rput(2,0.3){\tiny{3}}
\rput(3,0.3){\tiny{4}}
\psarc[linestyle=solid]{<-}(0.5,0.4){0,5}{180}{360}
\psarc[linestyle=solid]{<-}(1,0.4){1}{180}{360}
\psarc[linestyle=solid]{<-}(1.5,0.6){1.5}{0}{180}
\rput(0.7,-0.9){$P$}
\end{pspicture}\hspace{9mm}
\begin{pspicture}(0,-1)(3,2)\psdots[dotsize=3pt](0,0.5)(1,0.5)
\rput(0,0.8){\tiny{1}}
\rput(1,0.8){\tiny{2}}
\psarc[linestyle=solid]{<-}(0.5,0.4){0,5}{180}{360}
\psarc[linestyle=solid]{<-}(0.5,0.6){0,5}{0}{180}
\rput(0.5,-0.8){$Q$}
\end{pspicture}
\end{center}

If $\alpha:[3]\to \bigl\{\amalg, \uparrow, \downarrow, \updownarrow\bigr\}$ is defined as 
$\alpha(1)=\uparrow$, $\alpha(2)=\amalg$ and $\alpha(3)=\downarrow$ then $P\ast_\alpha Q$ is the following poset 
	
\begin{center}
		
\begin{pspicture}(0,-1)(4,3)
\psdots[dotsize=3pt](0,0.5)(1,0.5)(2,0.5)(3,0.5)(4,0.5)(5,0.5)
\psarc[linestyle=solid]{<-}(0.5,0.4){0,5}{180}{360}
\psarc[linestyle=solid]{<-}(1,0.4){1}{180}{360}
\psarc[linestyle=solid]{<-}(1.5,0.6){1.5}{0}{180}
\rput(0,0.3){\tiny{1}}
\rput(1,0.3){\tiny{2}}
\rput(2,0.3){\tiny{3}}
\rput(3,0.8){\tiny{4}}
\rput(4,0.8){\tiny{5}}
\rput(5,0.8){\tiny{6}}
\psarc[linestyle=solid]{<-}(4.5,0.4){0,5}{180}{360}
\psarc[linestyle=solid]{<-}(4.5,0.6){0,5}{0}{180}
\psarc[linestyle=solid,linecolor=red]{<-}(3.5,0.5){0.5}{0}{180}
\end{pspicture}
\end{center}	
If we change the posets for the following ones 
\begin{center}
\begin{pspicture}(0,-1)(3,2)\psdots[dotsize=3pt](0,0.5)(1,0.5)
\rput(0,0.8){\tiny{1}}
\rput(1,0.8){\tiny{2}}
\psarc[linestyle=solid]{<-}(0.5,0.4){0,5}{180}{360}
\rput(0.5,-0.8){$P$} 
\end{pspicture}\hspace{6mm}
\begin{pspicture}(0,-1)(3,2)\psdots[dotsize=3pt](0,0.5)(1,0.5)(2,0.5)
\rput(0,0.3){\tiny{1}}
\rput(1,0.3){\tiny{2}}
\rput(2,0.3){\tiny{3}}
\psarc[linestyle=solid]{<-}(0.5,0.6){0,5}{0}{180}
\psarc[linestyle=solid]{<-}(1,0.6){1}{0}{180}
\psarc[linestyle=solid]{<-}(1.5,0.4){0,5}{180}{360}
\rput(0.7,-0.9){$Q$}
\end{pspicture}
\end{center}	 
Then $P\ast_\alpha Q$ is
 
\begin{center}
		
\begin{pspicture}(0,-1)(4,3)
\psdots[dotsize=3pt](0,0.5)(1,0.5)(2,0.5)(3,0.5)(4,0.5)
\psarc[linestyle=solid]{<-}(0.5,0.4){0,5}{180}{360}
\psarc[linestyle=solid]{<-}(3,0.6){1}{0}{180}
\psarc[linestyle=solid]{<-}(2.5,0.6){0,5}{0}{180}
\psarc[linestyle=solid]{<-}(3.5,0.4){0.5}{180}{360}
\rput(0,0.3){\tiny{1}}
\rput(1,0.3){\tiny{2}}
\rput(2,0.3){\tiny{3}}
\rput(3,0.8){\tiny{4}}
\rput(4,0.8){\tiny{5}}
\psarc[linestyle=solid,linecolor=red]{<-}(2.5,0.5){1.5}{180}{360}
\psarc[linestyle=solid,linecolor=red]{<-}(1.5,0.4){0.5}{0}{180}
\end{pspicture}	
\end{center}
\end{example}

\begin{remark} \label{rem:opbasic} 
Given three elements $R\in {\mathcal R}_n$, $Q\in {\mathcal R}_m$ and $P\in {\mathcal R}_r$, and a pair of maps $\alpha$ and $\gamma$ in the disjoint union $ \bigcup_{p} {\mathbb O}_{\mathcal R}^{[p]} \bigcup \{\sqcup\}$,
we have that
\begin{equation*} (R*_{\alpha} Q)*_{\gamma} P = R*_{\alpha}(Q*_{\gamma}P),\end{equation*}
if, and only if, $s(\alpha) \leq m$ .\end{remark}

In order to simplify notation, we denote by $\sqcup$ both the constant map ${\mathbb N}\longrightarrow {\mathbb O}_{\mathcal R}$, which maps any positive integer to $\sqcup$, and its associated product $*_{\sqcup}$, which sends a pair $P$ and $Q$ of integer relations to their disjoint union $P\sqcup Q$.
\medskip

\begin{proposition} \label{prop:generado}  Let $S_{\mathcal R} :=\bigcup_{p\geq 1} {\mathbb O}_{\mathcal R}^{[p]} \bigcup \{\sqcup\}$ be the set generated by the maps ${\alpha}$ and the disjoint union $\sqcup$.\begin{enumerate}
\item $\K[{\mathcal R}]$ is generated as an $S_{\mathcal R}$-magmatic algebra by the element $\#\in {\mathcal R}_1$.
\item For any map $\alpha\in \bigcup_{p\geq 1}{\mathbb O}_{\mathcal R}^{[p]}$, the product $*_{\alpha}$ satisfies the unital infinitesimal relation with respect to the coproduct $\Delta_{\mathcal R}$.\end{enumerate}\end{proposition}

\begin{proof} We need to see that any element $R\in {\mathcal R}_n$ may be obtained from $\#$ by applying the products $*_{\alpha}$, for $\alpha\in S_{\mathcal R}$ or the disjoint union $\sqcup$. We proceed by recurrence on $n\geq 1$.

For $n =1$ or $ 2$, the result is evident.

For $n > 2$, let $R\in {\mathcal R}_n$ be a relation. The relation $R\vert_{\{2,\dots ,n\}}$ belongs to ${\mathcal R}_{n-1}$, so we assume that it is obtained from the element $\#$ by applying the products $*_{\alpha}$ or $\sqcup$. 
If $R$ is the disjoint union of $\#$ and $R\vert_{\{2,\dots ,n\}}$, the result is true.

Otherwise, let $p$ be the largest integer such that either $(1, p+1)\in R$ or $(p+1,1)\in R$. Let $\gamma:[p]\longrightarrow {\mathbb O}_{\mathcal R}$ be the map 
\begin{equation*} \gamma(i):=\begin{cases} \sqcup , &{\rm if}\ (1,i+1)\notin R\ {\rm and} \ (i+1,1)\notin R,\ {\rm or}\  i > p, \\
\uparrow , &{\rm if}\ (1, i+1)\in R\ {\rm and} (i+1,1)\notin R,\\
\downarrow , &{\rm if}\ (1, i+1)\notin R\ {\rm and}\ (i+1,1)\in R,\\
\updownarrow  , &{\rm if}\ (1, i+1)\in R\ {\rm and}\ (i+1,1)\in R,\end{cases}\end{equation*}
for $1\leq i\leq p$.

It is easily seen that $R = \# *_{\gamma}R\vert_{\{ 2,\dots ,n\}}$. 

Remark \ref{rem:opbasic} and the previous equality prove the first point.
\bigskip

The second point is an straightforward consequence of Lemma \ref{lem:infoprel} and the following equalities\begin{enumerate}[(i)]
\item $(P*_{\alpha}Q)\vert_{[i]} =\begin{cases}P\vert_{[i]},&\ {\rm for}\ 0\leq i\leq n,\\
P*_{\alpha} (Q\vert_{[i-n]}),&\ {\rm for}\ n+1\leq i\leq n+m,\end{cases}$
\item  $(P*_{\alpha}Q)\vert_{\{ i+1,\dots, n+m\}} = \begin{cases} P\vert_{\{ i+1,\dots ,n\}} *_{\alpha} Q,&\ {\rm for}\ 0\leq i\leq n,\\
Q\vert_{\{i+1,\dots ,n+m\}},&\ {\rm for}\ n+1\leq i\leq n+m.\end{cases}$
\end{enumerate}
\end{proof}
\bigskip

\subsection{Primitive elements in $({\mathbb H}_{{\mathcal R}_{PP}}^*, \Delta_{\mathcal R})$}
\medskip

Using Notation \ref{not:ops}, Theorem \ref{th:structure} states that, as a coalgebra, $({\mathbb H}_{{\mathcal R}_{PP}}^*, \Delta_{\mathcal R})$ is the cotensor coalgebra over the subspace of its primitive elements. The following result is a strightforward consequence of Theorem \ref{th:structure} and Proposition \ref{prop:genofprim}. 

\begin{proposition}\label{prop:primrel} Let $s_{\bullet}$ be the disjoint union product $\sqcup\ \in S_{\mathcal R}$. The subspace ${\mbox{Prim}({\mathbb H}_{{\mathcal R}_{PP}}^*)}$ of primitive elements of the coalgebra $({\mathbb H}_{{\mathcal R}_{PP}}^*, \Delta_{\mathcal R})$ is generated by the element $\#$ under the action of the products  \begin{enumerate}[(i)]
\item $M(\alpha) = *_{\alpha} - \sqcup$, with $\alpha \in \bigcup_p{\mathbb O}_{\mathcal R}^{[p]}$, 
\item $M_t(\alpha_1,\dots ,\alpha_{n-1})$, where $t={\mathbf 1}_k\circ (\vert ,t_2,\dots , t_{k-1},\vert)\in {\mbox{\bf PBT}_n}$ with $t_i = \vert \rightthreetimes w_i$ for $2\leq i\leq k-1$, satisfying that $\alpha_j\in S_{\mathcal R}$, for $1\leq j < n$, and that the first node of $t_i$ is colored with $s_{\bullet} = \sqcup$ whenever $w_i\neq 1_{\K}$.\end{enumerate}\end{proposition}
\medskip

We want to construct a basis of the subspace ${\mbox{Prim}({\mathbb H}_{{\mathcal R}_{PP}}^*)}$. Note that the dimension of the vector space $\K[{\mathcal R}_n]$ is $4^{\frac{n(n-1)}{2}}$.
\medskip

\begin{remark}\label{rem:disjprim} The disjoint union $\sqcup$ of integer relations is an associative product. And it is easily seen that $(\K[{\mathcal R}], \sqcup )$ is the free associative algebra spanned by the $\sqcup$-irreducible elements, that is the relations $R$ satisfying that whenever $R = R_1\sqcup R_2$ then either $R_1 = R$ or $R_2 = R$. We denote by
${\mbox{$\sqcup$-Irr}_n}$ the set of $\sqcup$-irreducible integer relations on $[n]$.

Therefore, we get that the vector spaces $\K[{\mathcal R}]$, $T(\bigoplus_{n\geq 1}\K[{\mbox{$\sqcup$-Irr}_n}])$ and $T({\mbox{Prim}({\mathbb H}_{{\mathcal R}_{PP}}^*}))$ are isomorphic.\end{remark}
\medskip

Remark \ref{rem:disjprim} states that the dimension of the subspace ${\mbox{Prim}({\mathbb H}_{{\mathcal R}_{PP}}^*)}_n$ is equal to the number of $\sqcup$-irreducible relations on $[n]$, for $n\geq 1$. We are going to associate to any $\sqcup$-irreducible relation $R$ on $[n]$  a primitive element $x_R\in ({{\mathbb H}_{{\mathcal R}_{PP}}}^*)_n$, and prove that the collection $\{ x_R\}_{R\in {\mbox{$\sqcup$-Irr}_n}}$ is a basis of the vector space ${\mbox{Prim}({\mathbb H}_{{\mathcal R}_{PP}}^*)}$.
\medskip

\begin{notation}\label{not:opchap} Let $t\in {\mbox{\bf PBT}_n}$, given maps $\alpha_1,\dots ,\alpha_{n-1}\in \bigcup _{p\geq 1}{\mathbb O}_{\mathcal R}^{[p]}\cup \{\sqcup\}$. We denote by\begin{enumerate}[(i)]
\item $[t, \alpha_1,\dots ,\alpha_{n-1}]$ the action of the colored tree $(t, (\alpha_1,\dots ,\alpha_{n-1}))$ on an element of $\K[{\mathcal R}]^{\otimes n}$,
\item  $\Omega(t, {\alpha}_1,\dots ,{\alpha}_{n-1})$ the operation given by the action of tree $t$ with the nodes colored, from left to right, by the binary products $M(\alpha_1),\dots , M(\alpha_{n-1})$.
\end{enumerate} 
Note that the subspace ${\mbox{Prim}({\mathbb H}_{{\mathcal R}_{PP}}^*)}$ is closed under the action of $\Omega(t, {\alpha}_1,\dots ,{\alpha}_{n-1})$.
\medskip

Sometimes, when no confusion may arise, we denote an element $\alpha \in {\mathbb O}_{\mathcal R}^{[p]}$ by its image $\alpha(1)\dots \alpha(p)$. For instance, the map $\alpha$ given by $\alpha(1) = \uparrow$, $\alpha(2) = \sqcup$ and $\alpha(3) = \updownarrow$, will be denoted by the word $\uparrow \sqcup \updownarrow$. 
\medskip

Let $P$ and $R$ be two relations in ${\mathcal R}_n$, for some $n\geq 2$. We say that $P\subsetneq R$ when $P$ is strictly contained in $R$ as subsets $[n]\times [n]$.
\end{notation}
\medskip

\begin{remark} \label{rem:Rasprod} Proposition \ref{prop:generado} implies that any integer relation $R\in {\mathcal R}_{n}$ may be written as 
\begin{equation*} R = \# *_{\alpha_{n-1}} (\# *_{\alpha_{n-2}}(\dots *_{\alpha_{2}}(\# *_{\alpha_{1}} \#))),\end{equation*}
for unique elements $\alpha_1,\dots ,\alpha_{n-1}\in \bigcup _{p\geq 1}{\mathbb O}_{\mathcal R}^{[p]}\bigcup \{\sqcup\}$, where the size of $\alpha_i$ is smaller or equal to $i$.\end{remark}
\medskip

\begin{lemma} \label{lem:trees} For any $n\geq 3$, an element $(R_{n},\dots , R_1)\in (\bigcup_{m\geq 1}{\mathcal R}_m)^{\times n}$ and any family of elements $\alpha_1,\dots ,\alpha_{n-1}\in S_{\mathcal R}$, satisfying that there exists $2 \leq i< n$ such that $s(\alpha_i)\leq \vert R_{i}\vert$, we have that
\begin{equation*} M_{{\mathbf 1}_n}(\alpha_{n-1},\dots ,\alpha_1)(R_n,\dots ,R_1) = 0.\end{equation*}
If $s(\alpha_i) > \vert R_i\vert$, for all $2\leq i\leq n-1$, then 
\begin{equation*} M_{{\mathbf 1}_n}(\alpha_{n-1},\dots ,\alpha_1)(R_n,\dots , R_1) =  R +\sum_{P\subsetneq R} b_P P,\end{equation*}
where $R:=R_n *_{\alpha_{n-1}} (R_{n-1} *_{\alpha_{n-2}}(\dots *_{\alpha_{2}}(R_2 *_{\alpha_{1}} R_1)))$, and $b_P\in \{ -1, 0, 1\}$ for $P\subsetneq R$.\end{lemma}

\begin{proof} We prove the first assertion by induction on $n$.

For $n =3$, we have that $R_3*_{\alpha_2} (R_2*_{\alpha _1} R_1) = (R_3*_{\alpha_2} R_2)*_{\alpha_1} R_1$, because $\vert R_2\vert \geq s(\alpha_2)$. Therefore, we get that
\begin{equation*}M_{{\mathbf 1}_3}(\alpha_2, \alpha_1) (R_3, R_2, R_1) = R_3*_{\alpha_2} (R_2*_{\alpha _1} R_1) - (R_3*_{\alpha_2} R_2)*_{\alpha_1} R_1 =0.\end{equation*}
\medskip

For $n > 3$, we have to consider two cases,\begin{enumerate}[(a)]
\item If $s(\alpha_{n-1}) \leq \vert R_{n-1}\vert$ and $s(\alpha_i) > \vert R_i\vert$, for $2\leq i\leq n-2$, then Proposition \ref{Moebiuisofcomb} states that 
\bigbreak
$
M_{{\mathbf 1}_n}(\alpha_{n-1},\dots ,\alpha_1) (R_n,\dots ,R_1) =R_n *_{\alpha _{n-1}} (R_{n-1}*_{\alpha_{n-2}} M_{{\mathbf 1}_{n-2}}(\alpha_{n-3},\dots ,\alpha_1) (R_{n-2},\dots ,R_1)) \\-
(R_n *_{\alpha _{n-1}} R_{n-1})*_{\alpha_{n-2}} M_{{\mathbf 1}_{n-2}}(\alpha_{n-3},\dots ,\alpha_1) (R_{n-2},\dots ,R_1)\\-
R_{n}*_{\alpha_{n-1}}((R_{n-1}*_{\alpha_{n-2}}R_{n-2})*_{\alpha_{n-3}}M_{{\mathbf 1}_{n-3}}(\alpha_{n-4},\dots ,\alpha_1)(R_{n-3}, \dots ,R_{1}))\\+
(R_n*_{\alpha_{n-1}}(R_{n-1}*_{\alpha_{n-2}}R_{n-2}))*_{\alpha_{n-3}}M_{{\mathbf 1}_{n-3}}(\alpha_{n-4},\dots ,\alpha_1)(R_{n-3}, \dots ,R_{1}).
$
\bigbreak
But, since $s(\alpha_{n-1}) \leq \vert R_{n-1}\vert$,  $s(\alpha_{n-1}) \leq \vert R_{n-1}*_{\alpha_{n-2}}R_{n-2}\vert $, too. Therefore, we get\begin{enumerate}[(i)]
\item \begin{align*} R_n *_{\alpha _{n-1}} (R_{n-1}*_{\alpha_{n-2}} &M_{{\mathbf 1}_{n-2}}(\alpha_{n-3},\dots ,\alpha_1) (R_{n-2},\dots ,R_1)) =\\
&(R_n *_{\alpha _{n-1}} R_{n-1})*_{\alpha_{n-2}} M_{{\mathbf 1}_{n-2}}(\alpha_{n-3},\dots ,\alpha_1) (R_{n-2},\dots ,R_1),\end{align*}

\item \begin{align*}R_{n}*_{\alpha_{n-1}}&((R_{n-1}*_{\alpha_{n-2}}R_{n-2})*_{\alpha_{n-3}}M_{{\mathbf 1}_{n-3}}(\alpha_{n-4},\dots ,\alpha_1)(R_{n-3}, \dots ,R_{1}))=\\
&(R_n*_{\alpha_{n-1}}(R_{n-1}*_{\alpha_{n-2}}R_{n-2}))*_{\alpha_{n-3}}M_{{\mathbf 1}_{n-3}}(\alpha_{n-4},\dots ,\alpha_1)(R_{n-3}, \dots ,R_{1}),\end{align*}
\end{enumerate}
which implies that $M_{{\mathbf 1}_n}(\alpha_{n-1},\dots ,\alpha_1) (R_n,\dots ,R_1) = 0$.
\item If $\vert R_{n-1}\vert > s(\alpha_{n-1})$, then there exists $2\leq i < n-1$, such that $s(\alpha_{i}) \leq \vert R_{i}\vert$. So, a recursive argument states that 
\begin{equation*} M_{{\mathbf 1}_{n-1}}(\alpha_{n-2},\dots ,\alpha_1)(R_{n-1},\dots ,R_1) = 0,\end{equation*} and also that 
\begin{equation*} M_{{\mathbf 1}_{n-1}}(\alpha_{n-2},\dots ,\alpha_1)(R_n*_{\alpha_{n-1}}R_{n-1},R_{n-2}, \dots ,R_1) = 0.\end{equation*}
Applying again Proposition \ref{Moebiuisofcomb}, we have that
\begin{align*} M_{{\mathbf 1}_n}(\alpha_{n-1},\dots ,\alpha_1)& (R_n,\dots ,R_1) = R_n*_{\alpha_{n-1}}M_{{\mathbf 1}_{n-1}}(\alpha_{n-2},\dots ,\alpha_1)(R_{n-1},\dots ,R_1) - \\
&M_{{\mathbf 1}_{n-1}}(\alpha_{n-2},\dots ,\alpha_1)(R_n*_{\alpha_{n-1}}R_{n-1},R_{n-2},\dots ,R_1)=0,\end{align*}
which ends the proof of the first assertion.
\end{enumerate}
\bigskip

  For the second point, we proceed again by induction. For $n = 3$, we get that
 \begin{equation*} (R_3*_{\alpha_2} R_2)*_{\alpha_1}R_1\subsetneq R_3*_{\alpha_2} (R_2*_{\alpha_1}R_1),\end{equation*}
 because $s(\alpha_2) > \vert R_2\vert$ implies that either the pair $(n_3, n_3 + s(\alpha_2))$ or the pair $( n_3 + s(\alpha_2), n_3)$, or both of them, belong to $R_3*_{\alpha_2} (R_2*_{\alpha_1}R_1)$ and do not belong to $(R_3*_{\alpha_2} R_2)*_{\alpha_1}R_1$, where $\vert R_3\vert = n_3$.
 
From Proposition \ref{Moebiuisofcomb}, we have that 

\begin{align*} M_{{\mathbf 1}_n}(\alpha_{n-1},\dots ,\alpha_1)(R_n,\dots , R_1) &= R_n*_{\alpha _{n-1}} M_{{\mathbf 1}_{n-1}}(\alpha_{n-2}, \dots ,\alpha_1)(R_{n-1},\dots , R_1)
\\&- M_{{\mathbf 1}_{n-1}}(\alpha_{n-2}, \dots ,\alpha_1)(R_n*_{\alpha_{n-1}} R_{n-1},R_{n-2}, \dots , R_1).\end{align*}
\medskip

Using a recursive argument, we have that \begin{enumerate}[(a)]
\item as $M_{{\mathbf 1}_{n-1}}(\alpha_{n-2},\dots ,\alpha_1)(R_{n-1},\dots , R_1) =$
\begin{equation*} ({\mathbf 1}_{n-1}, (\alpha_{n-2},\dots ,\alpha_1))(R_{n-1}, \dots , R_1)  + \sum _{P\rq \subsetneq R\rq} b_{P\rq} P\rq, \end{equation*}
where $R\rq := ({\mathbf 1}_{n-1}, (\alpha_{n-2},\dots ,\alpha_1))(R_{n-1}, \dots , R_1)$. 
So,\\ $R_n*_{\alpha _{n-1}} M_{{\mathbf 1}_{n-1}}(\alpha_{n-2}, \dots ,\alpha_1)(R_{n-1},\dots , R_1) =$
\begin{equation*} ( {\mathbf 1}_n, (\alpha _{n-1},\dots ,\alpha_1))(R_n, \dots ,R_1) +
\sum_{P_j \subsetneq R\rq} b_{P_j} R_n*_{\alpha _{n-1}}P_j,\end{equation*}
with $R_n*_{\alpha_{n-1}} P_j \subsetneq ({\mathbf 1}_{n}, (\alpha_{n-1},\dots ,\alpha_1))(R_{n}, \dots , R_1)$ for all $j$.
\item On the other hand, the element $M_{{\mathbf 1}_{n-1}}(\alpha_{n-2}, \dots ,\alpha_1)(R_n*_{\alpha_{n-1}} R_{n-1},\dots , R_1) $ is, by recursive hypothesis,  a linear combination of relations, 
\begin{align*} M_{{\mathbf 1}_{n-1}}(\alpha_{n-2}, \dots ,\alpha_1)(R_n*_{\alpha_{n-1}} R_{n-1},\dots , R_1) &= ({\mathbf 1}_{n-1},(\alpha_{n-2}, \dots ,\alpha_1))(R_n*_{\alpha_{n-1}}R_{n-1}, R_{n-2},\dots ,R_1) \\&+  \sum _{Q_k \subsetneq R^2} c_{Q_k} Q_k, \end{align*}
where $R^2 := ({\mathbf 1}_{n-1},(\alpha_{n-2}, \dots ,\alpha_1))(R_n*_{\alpha_{n-1}}R_{n-1}, R_{n-2},\dots ,R_1) $.
But, as $s(\alpha_{n-1}) > \vert R_{n-1}\vert $, we get that $R^2 \subsetneq R= ({\mathbf 1}_{n}, (\alpha_{n-1},\dots ,\alpha_1))(R_n,\dots , R_1)$. 
\end{enumerate}
So, 
\begin{align*} &M_{{\mathbf 1}_n}(\alpha_{n-1},\dots ,\alpha_1)(R_n,\dots , R_1) = ( {\mathbf 1}_n, (\alpha _{n-1},\dots ,\alpha_1))(R_n, \dots ,R_1) +\\
&\sum_{P_j \subsetneq R\rq} b_{P_j} R_n*_{\alpha _{n-1}}P_j - ({\mathbf 1}_{n-1},(\alpha_{n-2}, \dots ,\alpha_1))(R_n*_{\alpha_{n-1}}R_{n-1}, R_{n-2},\dots ,R_1) +\\
&  \sum _{Q_k \subsetneq R^2} c_{Q_k} Q_k, \end{align*}    
where $R_n*_{\alpha _{n-1}}P_j\subsetneq R$, $({\mathbf 1}_{n-1},(\alpha_{n-2}, \dots ,\alpha_1))(R_n*_{\alpha_{n-1}}R_{n-1}, R_{n-2},\dots ,R_1)\subsetneq R$ and $Q_k\subsetneq R$ for all $j$ and $k$, and the proof is over.
\end{proof}
\medskip

Note that Lemma \ref{lem:trees} shows that $\K[{\mathcal R}]$ is not the free $S_{\mathcal R}$-magmatic algebra.

\begin{remark} \label{rem:irreducibles} A relation $R=\#*_{\alpha_{n-1}}(\#*_{\alpha_{n-2}}(\dots *_{\alpha_2}(\# *_{\alpha_1}\#)))$ is $\sqcup$-irreducible if, and only if, it satisfies one of the following conditions\begin{enumerate}[(a)]
\item $\alpha_i\neq \sqcup$, for $1\leq i< n-1$,
\item if $\alpha_i =\sqcup$ there exists $i < j<n$ such that $s(\alpha_j) > j-i$.\end{enumerate}\end{remark}
\medskip

Let us define a map $\Xi : \sqcup-{\mbox{Irr}_n}\longrightarrow {\mbox{Prim}({\mathbb H}_{{\mathcal R}_{PP}}^*)}$.\begin{enumerate}
\item $\Xi (\#) := \#$, and $\Xi (\#*_{\alpha}\#) := M(\alpha)(\#, \#) = \#*_{\alpha}\# - \# \sqcup \#$, for $\alpha:[1]\longrightarrow  \{\uparrow , \downarrow, \updownarrow\}$,
\item For $R = \#*_{\alpha}R\rq$ with $R\rq$ a $\sqcup$-irreducible integer relation and $s(\alpha)\geq 1$, define
\begin{equation*}\Xi(R) := M(\alpha)(\#, \Xi(R\rq)).\end{equation*}
\item Otherwise, assume that $R = \#*_{\alpha}R\rq$, with $R\rq $ a $\sqcup$-reducible integer relation. Using that $\sqcup$ is associative, it is easily seen that there exist a unique collection $R_1,\dots ,R_q$ of $\sqcup$-irreducible integer relations such that $R\rq = R_1\sqcup \dots \sqcup R_q$, for $q >1$. Moreover, as $R$ is $\sqcup$-irreducible, we have that $s(\alpha) > \vert R_1\vert +\dots +\vert R_{q-1}\vert$. The element $\Xi (R)$ is given by
\begin{equation*} \Xi(R) := M_{\vert \veebar{\mathbf 0}_{q-1}}(\alpha, \sqcup , \dots, \sqcup)(\#, \Xi(R_1),\dots , \Xi(R_q)).\end{equation*}
Note that $M_{\vert \veebar{\mathbf 0}_{q-1}}(\alpha, \sqcup , \dots, \sqcup) = (\vert \veebar{\mathbf 0}_{q-1}(\alpha, \sqcup , \dots, \sqcup)) - (\vert \veebar{\mathbf 0}_{q-2} (\alpha, \sqcup , \dots, \sqcup))\veebar _{\sqcup} \vert$ is a primitive operator and the elements $ \Xi(R_1),\dots , \Xi(R_q)$ are primitive, which implies that $ \Xi(R)$ is a primitive element, too.
\end{enumerate}
\medskip

We want to prove that the set $\{\Xi(R)\}_{R\in \sqcup-{\mbox{Irr}}}$ is a basis of ${\mbox{Prim}({\mathbb H}_{{\mathcal R}_{PP}}^*)}$. We know that the dimension of the vector space ${\mbox{Prim}({\mathbb H}_{{\mathcal R}_{PP}}^*)}_n$ is equal to the number of $\sqcup$-irreducible relations of size $n$, for $n\geq 1$. So, it suffices to see that the set $\{\Xi(R)\}_{R\in \sqcup-{\mbox{Irr}_n}}$ is linearly independent, for $n\geq 1$. This result is a straightforward consequence of the following proposition.
\medskip

\begin{proposition} \label{prop:final} For any $\sqcup$-irreducible relation $R$, we have that:
\begin{equation*} \Xi(R) = R +\sum_{j}b_j P_j,\end{equation*}
with $P_j\subsetneq R$, for all $j$.\end{proposition}

\begin{proof} 
\begin{enumerate}
\item For $n = 1, 2$ the result is immediate. 

For $n > 2$, we proceed by recursion on $n$. We have to consider different cases
\item If $R = \# *_{\alpha} R\rq$, with $R\rq$ a $\sqcup$-irreducible integer relation, we assume that $\Xi(R\rq) = R\rq +\sum_{j}b_j P_j$, with $P_j\subsetneq R\rq$ whenever $b_j\neq 0$. So, we get that
\begin{align*}\Xi (R) =& M(\alpha)(\#, \Xi (R\rq)) = \\
&\#*_{\alpha} R\rq +\sum_{j}b_j \# _{\alpha} P_j - \# \sqcup R\rq - \sum_{j}b_j \# \sqcup P_j.\end{align*}
Clearly, $\#\sqcup R\rq \subsetneq R$ and $\#\sqcup P_j\subsetneq R$  if $b_p\neq 0$. Notice that, since  $P_j\subsetneq R\rq$ then $ \# _{\alpha} P_j \subsetneq R$ and $\Xi(R)$ is of the desired form in this case.
\item Let $R = \#*_{\alpha} (R_1\sqcup \dots\sqcup R_q)$, with $s(\alpha) > \vert R_1\vert + \dots +\vert R_{q-1}\vert$ and $R_1,\dots ,R_q$ a collection of $\sqcup$-irreducible integer relations, for $q > 1$.
Applying a recursive argument, we get that $\Xi (R_i) = R_i + \sum_{1\leq j\leq m_i} b_j^i P_j^i$, with $P_j^i\subsetneq R_i$ whenever $b_j^i\neq 0$, for $1\leq i\leq q$.
Recall that
\begin{align*} \Xi (R) &= \\
&(\vert \veebar {\mathbf 0_{q-1}} (\alpha,\sqcup, \dots, \sqcup))(\#, \Xi (R_1),\dots ,\Xi(R_q)) - ((\vert \veebar {\mathbf 0_{q-2}})\veebar \vert (\alpha , \sqcup , \dots , \sqcup))(\#, \Xi (R_1),\dots ,\Xi(R_q)).\end{align*}
We have that \begin{enumerate} [(i)]
\item \begin{align*} (\vert \veebar {\mathbf 0_{q-1}}, (\alpha,\sqcup, \dots, \sqcup))&(\#, \Xi (R_1),\dots ,\Xi(R_q)) = R +
&\sum_{\substack {j_1,\dots ,j_k,\\ 0\leq j_l\leq m_l\\ j_1+\dots +j_q\geq 1} } b_{j_1}^1\cdot \dots \cdot b_{j_q}^qR_{j_1,\dots ,j_q} ,\end{align*}
where the sum is taken over all elements $(j_1,\dots ,j_q)$ satisfying that $0\leq j_l\leq m_l$ and at least one integer $j_l > 0$, and 
\begin{equation*} R_{j_1,\dots ,j_q} :=  (\vert \veebar {\mathbf 0_{q-1}}, (\alpha,\sqcup, \dots, \sqcup))(\#, P_{j_1}^1,\dots ,P_{j_q}^q)),\end{equation*}
where $P_0^i = R_i$ and $b_0^i = 1$, for $1\leq i\leq q$.

As at least one element $j_l\neq 0$, we have that $R_{j_1,\dots ,j_q}\subsetneq R$.
\item On the other hand, we have that 
\begin{equation*} ((\vert \veebar {\mathbf 0_{q-2}})\veebar \vert (\alpha , \sqcup , \dots , \sqcup))(\#, Q_1,\dots ,Q_q) =  (\vert \veebar {\mathbf 0_{q-2}} (\alpha , \sqcup , \dots , \sqcup))(\#, Q_1,\dots ,Q_{q-1}))\sqcup Q_q,\end{equation*}
for any family of integer relations $Q_1,\dots ,Q_q$. Therefore, if $s(\alpha ) > \vert Q_1\vert + \dots +\vert Q_{q-1}\vert $, then it is clear that
\begin{equation*} ((\vert \veebar {\mathbf 0_{q-2}})\veebar \vert (\alpha , \sqcup , \dots , \sqcup))(\#, Q_1,\dots ,Q_q)\subsetneq (\vert \veebar {\mathbf 0_{q-1}} (\alpha, \sqcup,\dots ,\sqcup)) (\#, Q_1,\dots ,Q_q),\end{equation*}
which implies that \begin{itemize}
\item $((\vert \veebar {\mathbf 0_{q-2}})\veebar \vert (\alpha , \sqcup , \dots , \sqcup))(\#, R_1,\dots ,R_q)\subsetneq R$,
\item For any $R_{j_1,\dots ,j_q} :=  (\vert \veebar {\mathbf 0_{q-1}} (\alpha,\sqcup, \dots, \sqcup))(\#, P_{j_1}^1,\dots ,P_{j_q}^q))$,
where $P_0^i = R_i$ and $b_0^i = 1$ for $1\leq i\leq q$, we get that
\begin{equation*}((\vert \veebar {\mathbf 0_{q-2}})\veebar \vert (\alpha , \sqcup , \dots , \sqcup))(\#, P_{j_1}^1,\dots ,P_{j_q}^q)) \subsetneq (\vert \veebar {\mathbf 0_{q-1}} (\alpha,\sqcup, \dots, \sqcup))(\#, P_{j_1}^1,\dots ,P_{j_q}^q)).\end{equation*}
\end{itemize}

So, we may conclude that $\Xi (R) = R + \sum_{j} a_j Q_j$, with $Q_j\subsetneq R$ for $a_j\neq 0$.

\end{enumerate}
\end{enumerate}
As $\subseteq $ defines a partial order on the set of integer relations, the previous arguments show that the set $\{\Xi (R)\}_{R\in {\mbox{$\sqcup$-Irr}}}$ is a basis of the subspace of primitive elements of ${\mbox{Prim}({\mathbb H}_{{\mathcal R}_{PP}}^*)}$.

\end{proof}

\section*{Acknowledgments} The authors research is supported by the Project MathAmSud {\it Algebraic Structures Supported on Families of Combinatorial Objects}, AmSud 210017.
The authors wish to thank Rafael Gonz\'alez D\rq Le\'on for many useful comments, and Eugenia Ellis for the organization of the meeting {\it Matem\'aticas en el Cono Sur} where our joint work began.


\end{document}